\documentclass[letterpaper, 11pt]{amsart}
\usepackage{mathtools}
\usepackage{amsmath}
\usepackage{amssymb}
\usepackage{yhmath}
\usepackage{comment}
\usepackage{graphicx} 
\usepackage{ mathrsfs }
\usepackage{bbm}
\usepackage{xcolor}
\usepackage{tikz-cd}
\usepackage{tikz}
\usetikzlibrary{patterns}
\usepackage{hyperref}

\DeclareMathAlphabet{\mathpzc}{OT1}{pzc}{m}{it}

\usepackage{thmtools}
\usepackage{thm-restate}

\usepackage{caption}

\newtheorem{theorem}{Theorem}[section]

\newtheorem*{claim*}{Claim}

\newtheorem{lemma}[theorem]{Lemma}
\newtheorem{lem}[theorem]{Lemma}

\newtheorem{Cor}[theorem]{Corollary}
\newtheorem{cor}[theorem]{Corollary} 

\newtheorem{proposition}[theorem]{Proposition}

\newtheorem{prop}[theorem]{Proposition}

\theoremstyle{definition}
\newtheorem{definition}[theorem]{Definition}

\newtheorem{example}[theorem]{Example}

\theoremstyle{remark}
\newtheorem{remark}[theorem]{Remark}
\newtheorem{rmk}[theorem]{Remark}

\numberwithin{equation}{section}

\newcommand{\abs}[1]{\lvert#1\rvert}
\newcommand{\norm}[1]{\lVert#1\rVert}
\newcommand{\op}{\operatorname}

\newcommand{\be}{\begin{equation}}
\newcommand{\ee}{\end{equation}}
\newcommand{\Ga}{\Gamma}

\newcommand{\R}{\mathbb R}
\renewcommand{\H}{\mathbb H}
\newcommand{\Z}{\mathbb Z}
\newcommand{\N}{\mathbb N}
\newcommand{\ga}{\gamma}

\newcommand{\la}{\lambda}
\newcommand{\La}{\Lambda}

\newcommand{\ba}{\backslash}
\newcommand{\ov}{\overline}

\newcommand{\Isom}{\op{Isom}}
\newcommand{\PSL}{\op{PSL}}

\newcommand{\diam}{\op{Diam}}

\newcommand{\T}{\mathcal{T}}

\renewcommand{\S}{\mathbb S}

\newcommand{\id}{\op{id}}

\newcommand{\ess}{\mathsf{E}}

\newcommand{\Hor}{\mathcal{H}}

\newcommand{\Leb}{\op{Leb}}

\newcommand{\PMF}{\mathcal{PMF}}
\newcommand{\PML}{\mathcal{PML}}
\newcommand{\MF}{\mathcal{MF}}
\newcommand{\ML}{\mathcal{ML}}

\newcommand{\UE}{\mathcal{UE}}
\newcommand{\Mod}{\operatorname{Mod}}

\newcommand{\Ext}{\operatorname{Ext}}
\newcommand{\HM}{\operatorname{HM}}

\newcommand{\A}{\mathsf{A}}
\newcommand{\Lip}{\operatorname{Lip}}

\newcommand{\CAT}{\operatorname{CAT}}

\newcommand{\Spec}{\operatorname{Spec}_{sq}}

\newcommand{\SL}{\operatorname{SL}}

\begin{document}

\title[Invariant Radon measures on $\ML$]{Invariant measures\\on the space of measured laminations\\for subgroups of mapping class group}

\author{Inhyeok Choi}
\address{School of Mathematics, KIAS, Hoegi-ro 85, Dongdaemun-gu, Seoul 02455, South Korea}
\email{inhyeokchoi48@gmail.com}

\author{Dongryul M. Kim}
\address{Department of Mathematics, Yale University, New Haven, CT 06511}
\email{dongryul.kim@yale.edu}
 
\begin{abstract}

For a non-elementary subgroup of the mapping class group of a surface, we study its invariant Radon measures on the space of measured laminations, by classifying them on the recurrent measured laminations. In particular, given a divergence-type subgroup, we show the uniquely ergodic by explicitly constructing the ergodic measure. This generalizes Lindenstrauss--Mirzakhani's result and Hamenst{\"a}dt's result for the full mapping class group, in which case the ergodic measure is the Thurston measure. As a special case, we deduce that for a convex cocompact subgroup, every invariant ergodic Radon measure on the space of all measured laminations is either the unique measure on recurrent measured laminations, or  a counting measure on the orbit of a non-recurrent measured lamination.

Our method is geometric and does not rely on continuous or homogeneous flows on the ambient space or a dynamical system associated with a finite measure space. This leads to a unifying approach for various metric spaces, including Teichm\"uller spaces and partially $\operatorname{CAT}(-1)$ spaces.

\end{abstract}

\maketitle

\vspace{-3em}
\tableofcontents
\vspace{-3em}

\section{Introduction}

Let $S$ be a connected orientable surface of genus $g$ and with $p$ punctures, where $3g - 3+p \ge 1$, i.e., $S$ is a finite-type surface which is not a sphere with at most 3 punctures or a torus. The Teichm{\"u}ller space $\T = \mathcal{T}(S)$ is the space of all marked Riemann surface structures on $S$, or equivalently, the space of all marked hyperbolic structures on $S$. The mapping class group $\Mod(S)$ is the group of isotopy classes of orientation-preserving self-diffeomorphisms of $S$. The Teichm\"uller space $\T$ is equipped with a natural metric called the Teichm\"uller metric, under which the natural action of $\Mod(S)$ is isometric. The $\Mod(S)$-action is properly discontinuous, and its quotient $\mathcal{M} := \Mod(S) \ba \T$ is the moduli space of Riemann surfaces.

In Thurston's theory on surfaces (\cite{1979travaux}, \cite{thurston0the-geometry}), the mapping class group and the Teichm\"uller space of $S$ are closely related to \emph{measured laminations} on $S$. Fixing a complete hyperbolic structure on $S$ of finite volume, a geodesic lamination on $S$ is a compact subset of $S$ foliated with simple geodesics. A \emph{measured lamination} is a geodesic lamination equipped with a transverse measure. This notion generalizes simple closed curves.

The space of measured laminations on $S$ is denoted by $\ML = \ML(S)$. Endowing $\ML$ with the weak*-topology, it turns out that $\ML$ is homeomorphic to $ \S^{6g - 7 + 2p} \times (0, + \infty) \simeq \R^{6g - 6 + 2p} \smallsetminus \{0\}$, where the $(0, +\infty)$-component corresponds to the scaling of transverse measures on each geodesic lamination. Here, the unit sphere $\S^{6g - 7 + 2p}$ is identified with the space $\PML = \PML(S)$ of \emph{projective measured laminations} on $S$. The space $\PML$ can be regarded as the boundary of $\T$.

Another interpretation of $\ML$ is given by the celebrated theorem of Hubbard and Masur \cite{hubbard1979quadratic}. Fixing a Riemann surface structure $x \in \T$, the Hubbard--Masur theorem asserts that the space $\ML$ is homeomorphic to the space of holomorphic quadratic differentials $\mathcal{Q}(S, x)$ on the Riemann surface $(S, x)$. Each quadratic differential $q \in \mathcal{Q}(S, x)$ determines a Teichm\"uller geodesic ray emanating from  $x \in \T$ and 
the projective class $[\xi] \in \PML$ of the measured lamination $\xi \in \ML$ corresponding to $q$, given by the Hubbard--Masur theorem, is accumulated by this Teichm\"uller geodesic ray.

The space $\ML$  admits a natural $\Mod(S)$-action, which encodes the global geometry and dynamics of the Teichm\"uller space. Indeed, the induced $\Mod(S)$-action on $\PML$ is a continuous extension of the $\Mod(S)$-action on $\T$. Using the train track coordinates on $\ML$, or a natural symplectic structure on $\ML$, Thurston defined the $\Mod(S)$-invariant measure
$$
\mu_{\rm Th} \quad \text{on} \quad \ML
$$
which belongs to the Lebesgue measure class \cite{thurston0the-geometry}. The measure $\mu_{\rm Th}$ is now called \emph{Thurston measure}, and the $\Mod(S)$-action on $(\ML, \mu_{\rm Th})$ is ergodic as shown by Masur \cite[Theorem 2]{masur1985ergodic}. The Thurston measure has been a central object in the study of the geometry and dynamics of $\Mod(S)$ and $\T$. For example,  Mirzakhani \cite{mirzakhani2008growth},  Athreya--Bufetov--Eskin--Mirzakhani \cite{athreya2012lattice}, and Erlandsson--Souto \cite{erlandsson2016counting} studied counting and equidistribution of $\Mod(\Sigma)$-orbits using $\mu_{\rm Th}$. It also played a major role in Mirzakhani's ergodic theory of Thurston's earthquake flow \cite{mirzakhani2008ergodic}.

\subsection{Measure classification for full $\Mod(S)$: finite covolume}
It is natural to ask whether the Thurston measure is the only measure on $\ML$ that is suited to the study of the dynamics of $\Mod(S)$. Lindenstrauss and  Mirzakhani \cite{lindenstrauss2008ergodic}, and Hamenst\"adt \cite{hamenstadt2009invariant}, independently classified all $\Mod(S)$-invariant ergodic Radon measures on $\ML$. While the $\Mod(S)$-action on $\ML$ is not uniquely ergodic due to the cuspidal feature of $\Mod(S)$ or of $\mathcal{M} = \Mod(S) \ba \T$, the Thurston measure turns out to be the unique $\Mod(S)$-invariant Radon measure on the region where the Teichm\"uller geodesic flow shows interesting dynamics.

More precisely, using the identification $\ML = \mathcal{Q}(S, x)$ provided by the Hubbard--Masur theorem, we define the \emph{recurrence locus} for $\Mod(S)$ as the set  $\mathcal{R}_{\Mod(S)} \subset \ML$ of all measured laminations $\xi \in \ML$ such that  the Teichm\"uller geodesic ray determined by the corresponding quadratic differential $q_{\xi} \in \mathcal{Q}(S, x)$ recurs to a compact subset in the quotient $\mathcal{M} = \Mod(S) \ba \T$. One of the main steps in the measure classification by both Lindenstrauss--Mirzakhani and Hamenst\"adt is as follows.

\begin{theorem}[{\cite{lindenstrauss2008ergodic}, \cite{hamenstadt2009invariant}}] \label{thm:mcgcase}
Let $\mu$ be a $\Mod(S)$-invariant Radon measure on $\mathcal{R}_{\Mod(S)}$. Then $\mu$ is a constant multiple of $\mu_{\rm Th}$.
\end{theorem}

In the exceptional case when $S$ is a once-punctured torus or a 4-punctured sphere (i.e., $3g - 3 + p = 1$), homogeneous dynamics comes into play. In this case, the Teichm\"uller space is equal to the hyperbolic plane $\mathbb{H}^{2}$ and $\Mod(S)$ is equal to $\operatorname{SL}(2, \Z)$. Then the $\Mod(S)$-action on $\ML$ corresponds to the $\SL(2, \Z)$-action on the \emph{horospherical foliation} of the unit tangent bundle $\operatorname{T}^1 \H^2$, or the \emph{unipotent flow (or horocyclic flow)} on $\PSL(2, \Z) \ba \PSL(2, \R)$. 

Furstenberg first proved that the unipotent flow on $\Ga \ba \PSL(2, \R)$ is uniquely ergodic with respect to the Haar measure when $\Ga < \PSL(2, \R)$ is a uniform lattice \cite{furstenberg1973the-unique}. This was extended by Veech to uniform lattices in semi-simple Lie groups \cite{veech1977unique}. Generalizing these results, Dani proved for lattices in reductive groups that the Haar measure is the unique invariant ergodic Radon measure on the recurrence locus (\cite{dani1978invariant}, \cite{dani1981invariant}).

\subsection{Measure classification for subgroups: infinite covolume}

Given Theorem \ref{thm:mcgcase}, one can seek an analogous measure classification on $\ML$ for general subgroups of $\Mod(S)$, as remarked by Lindenstrauss and Mirzakhani \cite[Remark 1.4(2)]{lindenstrauss2008ergodic}. We aim to study this question in this paper.
 
Before we proceed, let us remark two important ingredients for both Lindenstrauss--Mirzakhani's and Hamenst{\"a}dt's complete measure classification for the full $\Mod(S)$-action. First, the $\Mod(S)$-action has finite covolume with respect to the Masur--Veech volume form on the fiber bundle $\mathcal{Q}^1\T$ of unit-area holomorphic quadratic differentials over $\T$ (\cite{masur1982interval}, \cite{veech1982gauss}). The bundle $\mathcal{Q}^1 \T$ can be regarded as the unit cotangent bundle over $\T$, and the Masur--Veech volume is also induced from the Thurston measure. 
Second, the Teichm{\"u}ller horocyclic flow on the quotient bundle $\Mod(S) \ba \mathcal{Q}^1 \T$ exhibits some non-divergence, as proved by Minsky and  Weiss \cite{minsky2002nondivergence}.

On the other hand,  the quotient $\Ga \ba \mathcal{Q}^1 \T$ has infinite volume for a general subgroup $\Ga < \Mod(S)$, and it is hard to expect the non-divergence of horocyclic flows on $\Ga \ba \mathcal{Q}^1 \T$. Even for the case that $S$ is a once-punctured torus or a 4-punctured sphere, when $\T$ is the hyperbolic plane, neither holds true. They serve as obstructions to applying the arguments of \cite{lindenstrauss2008ergodic} or \cite{hamenstadt2009invariant} to subgroups of $\Mod(S)$.

On the hyperbolic plane $\H^2$, the first infinite-covolume examples studied for this problem is due to Burger \cite{Burger_horoc}. Burger considered a \emph{convex cocompact} subgroup $\Gamma < \PSL(2, \mathbb{R})$ whose critical exponent is strictly bigger than $1/2$ and showed that there exists a unique $\Gamma$-invariant ergodic Radon measure on the recurrent horospherical foliation of $\mathbb{H}^{2}$. 
This was later generalized by Roblin \cite{Roblin2003ergodicite} for discrete groups of isometries on $\CAT(-1)$ spaces under assumptions on finite Bowen--Margulis--Sullivan measures and the existence of certain coverings.

Let us get back to the Teichm\"uller space $\T$, where tools from homogeneous dynamics or negatively curved geometry do not apply immediately.
Before presenting our precise statements that require the construction of a specific Radon measure, we summarize our measure classification results as follows:
\begin{enumerate}
    \item For non-elementary $\Ga < \Mod(S)$, there exists at most one $\Ga$-invariant Radon measure on $\ML$ supported on the recurrence locus for $\Ga$.
    \item For non-elementary $\Ga < \Mod(S)$ of divergence type, the $\Ga$-action on the recurrence locus is uniquely ergodic.\footnote{By unique ergodicity, we mean that there exists a unique ergodic invariant Radon measure up to a constant multiple.}
    \item For non-elementary convex cocompact subgroups $\Ga < \Mod(S)$, we classify all $\Ga$-invariant Radon measures on $\ML$.
\end{enumerate}

As we will see later, we in fact develop machineries for a general metric space with a partial hyperbolicity, without any assumption on its global geometry. We use them to deduce versions of (1) and (2) in that setting.

\subsection{Main statements}

We mainly consider a \emph{non-elementary} subgroup $\Ga < \Mod(S)$, i.e.,  $\Ga$ is not virtually cyclic and contains a pseudo-Anosov mapping class. There exists $0 < \delta_{\Ga} < + \infty$, called the critical exponent of $\Ga$, so that $\delta_{\Ga}$ is the abscissa of convergence of the Poincar\'e series $s \mapsto \sum_{g \in \Ga} e^{-s d(x, g x)}$, for (any) $x \in \T$. The finiteness of the critical exponent is due to Kaimanovich and Masur \cite{kaimanovich1996poisson}, and its positivity is by McCarthy \cite{mccarthy1985a-tits-alternative}. We say that $\Ga$ is of \emph{divergence type} if the Poincar\'e series diverges at $s = \delta_{\Ga}$. Otherwise, $\Ga$ is said to be of convergence type.

Generalizing the ergodicity of the $\Mod(S)$-action on $(\ML, \mu_{\rm Th})$ due to  Masur \cite{masur1985ergodic}, we construct a Radon measure for non-elementary subgroups of $\Mod(S)$ and show its ergodicity.

\begin{theorem}[Ergodicity] \label{thm:main1}
Let $\Ga < \Mod(S)$ be a non-elementary subgroup of divergence type. Then there exists a nonzero, $\Ga$-invariant Radon measure $\mu_{\Ga}$ on $\ML$ such that
$$
\text{the $\Ga$-action on $(\ML,  \mu_{\Ga})$ is ergodic.}
$$
\end{theorem}

Theorem \ref{thm:main1} also applies to certain normal subgroups of a divergence-type group, which are not necessarily of divergence type. See Theorem \ref{thm:ergodicnormal} and Section \ref{subsec:theoremformcg}. 

The measure $\mu_{\Ga}$ is very explicit. Delaying its construction, we first discuss unique ergodicity.
Similar to $\mathcal{R}_{\Mod(S)}$, we define the \emph{recurrence locus} $\mathcal{R}_{\Ga} \subset~\ML$ for a subgroup $\Ga < \Mod(S)$ as follows: fixing $x \in \T$,
$$
\mathcal{R}_{\Ga} := \left\{ \xi \in \ML :
\begin{matrix}
    \text{Teichm\"uller geodesic ray given by } q_{\xi} \in \mathcal{Q}(S, x) \\
    \text{recurs to a compact subset in } \Ga \ba \T
\end{matrix} \right\}
$$
where $q_{\xi} \in \mathcal{Q}(S, x)$ is the quadratic differential on the Riemann surface $(S, x)$ corresponding to $\xi \in \ML$, given by the Hubbard--Masur theorem.
This set is $\Ga$-invariant and does not depend on the choice of $x \in \T$. We construct $\mu_{\Ga}$ so that it is $\Ga$-invariant and is supported on the recurrence locus $\mathcal{R}_{\Ga}$. Moreover, we show that  $\mu_{\Gamma}$ possesses unique ergodicity.

\begin{theorem}[Unique ergodicity] \label{thm:main1Conv} 
Let $\Ga < \Mod(S)$ be a non-elementary subgroup. Suppose that there exists a nonzero, $\Ga$-invariant Radon measure $\mu$  on $\ML$ that is supported on $\mathcal{R}_{\Ga}$.  Then $\Ga$ is of divergence type and $$\mu \text{ is a constant multiple of } \mu_{\Ga}.$$
In other words,
\begin{enumerate}
\item if $\Ga$ is of convergence type, then there does not exist nonzero, $\Ga$-invariant Radon measure on $\mathcal{R}_{\Ga}$.
\item if $\Ga$ is of divergence type, then
$$
\text{the $\Ga$-action on $(\mathcal{R}_{\Ga}, \mu_{\Ga})$ is uniquely ergodic.}
$$
\end{enumerate}
\end{theorem}

The full mapping class group $\Mod(S)$ is of divergence type thanks to \cite{masur1982interval} and \cite{veech1982gauss}. Thus, Theorem \ref{thm:main1Conv} extends Theorem \ref{thm:mcgcase}.

Our approach to Theorem \ref{thm:main1Conv} is to show that $\mu$ is quasi-invariant under the sclaing of the measured laminations. It is now standard that once we have this quasi-invariance, $\mu$ must be of the desired form. In a general abstract setting, this implication from the quasi-invariance was studied under the name of Maharam measures by Aaronson, Nakada, Sarig, and Solomyak \cite[0.1 Basic Lemma]{aaronson2002invariant}. See also \cite{sarig2004invariant} for the more explicit case of abelian covers of a closed hyperbolic surface.

\begin{example}[Divergence-type subgroups]
Examples of divergence-type subgroups of $\Mod(S)$ include (1) the full  group $\Mod(S)$ and (2) convex cocompact subgroups \cite{gekhtman2013dynamics}, as discussed in more detail below.

(3)
Stabilizers of many Teichm{\"u}ller disks are other sources of subgroups of divergence type, which are also referred to as Veech groups. To give explicit examples of pseudo-Anosovs, Thurston considered in \cite{thurston1988classification} the subgroup $\Gamma < \Mod(S)$ generated by multitwists about a filling pair of multicurves on $S$. Then he showed that $\Gamma$ stabilizes a Teichm{\"u}ller disk, which is an isometrically embedded copy of $\mathbb{H}^{2}$ in $\T$. Namely, $\Gamma$ is a geometrically finite Fuchsian group and is of divergence type with respect to the Teichm{\"u}ller metric \cite{patterson1976the-limit}. Later, Veech constructed many other families of lattice stabilizers of Teichm{\"u}ller disks \cite{veech1989teichmuller}. Note that these subgroups contain multitwists and are not convex cocompact, and also that any finitely generated Veech group is of divergence type. As  Teichm{\"u}ller disks are much smaller than the ambient Teichm{\"u}ller space, they are ``infinite-covolume subgroups". 

(4) Given a collection $\mathcal{C} \subset \Mod(S)$ of independent pseudo-Anosovs $\{f_{1}, f_{2}, \ldots\}$, a ping-pong argument using the contracting properties of $f_{i}$'s guarantees the following: for suitably large $n_{i} \in \N$ for each $i$, $\langle f_{i}^{n_{i}} : f_i \in \mathcal{C} \rangle$ becomes a free subgroup of $\Mod(S)$ (cf. \cite{farb2002convex}, \cite{Mosher_Schottky}). When $\mathcal{C}$ is finite, the resulting subgroup is convex cocompact. When $\mathcal{C}$ is infinite, one can choose $n_{i}$'s carefully so that the resulting subgroup is of divergence type, which is not convex cocompact since it is not finitely generated.

(5) Moreover, many divergence-type subgroups can be constructed using covering between surfaces \cite{BS_covering}.

(6) Being divergence type is also related to the growh of orbits. Concretely, when a non-elementary subgroup $\Ga < \Mod(S)$ has \emph{purely exponential growth}, i.e., $\# \{ g \in \Ga : d(x, g x) \le R \} \asymp e^{\delta_{\Ga} R}$ as $R \to + \infty$, for (any) $x \in \T$, it is easy to see that $\Ga$ is of \emph{divergence type}. This is indeed the case for $\Mod(S)$  \cite[Theorem 1.2]{athreya2012lattice} and for convex cocompact subgroups \cite[Theorem 1.1]{gekhtman2013dynamics}.
 
This motivated the studies of Schapira and Tapie \cite{schapira2021regularity}, and of Yang \cite{yang2019statistically}, on strongly positively recurrent (SPR) groups, also known as statistically convex cocompact (SCC) groups. These groups are known to have purely exponential growth \cite[Theorem B]{yang2019statistically}  (cf. \cite[Theorem 7.26]{schapira2021regularity}), and therefore they are of \emph{divergence type}. Yang also constructed many examples of SCC subgroups in \cite{yang2019statistically}. We will not define the notion of SPR=SCC groups, but let us mention that there are non-convex-cocompact examples which are SCC \cite[Proposition 6.6]{yang2019statistically}. Our theory applies to these examples as well.

\end{example}

We now come back to our discussion of measure classifications. Farb and Mosher \cite{farb2002convex} introduced another important  family of  subgroups of divergence type. A finitely generated $\Ga < \Mod(S)$ is called \emph{convex cocompact} if it has a quasi-convex orbit in $\T$. 
We establish a measure classification for $\Gamma$ on the \emph{entire} space of measured laminations.

\begin{theorem}[Convex cocompact subgroups] \label{thm:mainCC}
    Let $\Ga < \Mod(S)$ be a non-elementary convex cocompact subgroup. Then every $\Ga$-invariant ergodic Radon measure on $\ML$ is either
    $$
    \mu_{\Ga} \quad \text{or} \quad \sum_{g \in \Ga} D_{g \cdot \xi} \text{ for some } \xi \in \ML \smallsetminus \mathcal{R}_{\Ga}
    $$
    up to a constant multiple, where $D_{g \cdot \xi}$ is the Dirac measure at $g \cdot \xi \in \ML  \smallsetminus \mathcal{R}_{\Ga}$.
\end{theorem}
In Theorem \ref{thm:orbitclosure}, we also obtain the classification of $\Ga$-orbit closures in $\ML$. This is equivalent to considering the quotient $\Ga \ba \mathcal{Q}^1 \T$ and classifying stable manifolds under the Teichm\"uller geodesic flow.

Our complete measure classification for convex cocompact $\Gamma$ is based on explicit understanding of the $\Gamma$-action on $\ML  \smallsetminus \mathcal{R}_{\Ga}$. It would be of independent interest to understand such an action for a general subgroup, which is also relevant to the structure of limit sets in the Thurston boundary $\PML$.

\begin{remark}
We note that Theorem \ref{thm:mainCC} can also be studied using the finiteness of the Bowen--Margulis--Sullivan (BMS) measures for convex cocompact subgroups due to Gekhtman \cite{gekhtman2013dynamics}, and their strong geometric properties deduced from quasi-convexity of orbits in $\T$. Indeed, Roblin provided in \cite{Roblin2003ergodicite} a strategy based on finite BMS measures and certain geometric conditions. The authors appreciate Gekhtman and Oh for sharing this viewpoint with the authors. It would be interesting to do this alternative approach and see any byproducts.
\end{remark}

\subsection{Explicit construction of invariant Radon measures} \label{subsec:constructmeasureintro}

We now explain the construction of the measure $\mu_{\Ga}$ for non-elementary $\Ga < \Mod(S)$.

Fixing a basepoint $x \in \T$, a Borel probability measure $\nu$ on $\PML$ is called a \emph{$\delta_{\Ga}$-dimensional conformal measure} of $\Ga$ if for every $g \in \Ga$,
$$
\frac{d g_* \nu}{d \nu}([\xi]) = \left(\sqrt{ \frac{\Ext_{x}(\xi)}{\Ext_{g x}(\xi)} }\right)^{\delta_{\Ga}} \quad \text{for a.e. }[\xi] \in \PML
$$
where $\Ext_{x}(\xi)$ denotes the extremal length of $\xi \in \ML$ on the Riemann surface $(S, x)$ (see Equation \eqref{eqn:extlength} for the precise definition). Such a measure was constructed by Athreya--Bufetov--Eskin--Mirzakahni \cite{athreya2012lattice} for $\Ga = \Mod(S)$, by Gekhtman \cite{gekhtman2013dynamics} for $\Ga$ convex cocompact, and by Coulon \cite{coulon2024patterson-sullivan} and by Yang \cite{yang2024conformal} in general. 

As we fix a basepoint $x \in \T$, we can explicitly write the homeomorphism $\ML \to \PML \times (0, + \infty)$ using the Hubbard--Masur theorem as follows:
$$
\xi \mapsto \left([\xi], \sqrt{\Ext_{x}(\xi)} \right).
$$
For a non-elementary subgroup $\Ga < \Mod(S)$ and a $\delta_{\Ga}$-dimensional conformal measure $\nu$ of $\Ga$ on $\PML$, we define the Radon measure $\mu_{\nu}$ on $\ML$ by 
$$
d\mu_{\nu}(\xi) := \left(\sqrt{\Ext_{x}(\xi)} \right)^{\delta_{\Ga} - 1} \cdot d \nu([\xi]) \, d\Leb_{\R} \left( {\textstyle \sqrt{\Ext_{x}(\xi)}} \right).
$$
It follows from the conformality of $\nu$ that $\mu_{\nu}$ is $\Ga$-invariant. Moreover, while the conformal measure $\nu$ depends on the choice of $x \in \T$, the measure $\mu_{\nu}$ is independent of $x$ since the above map $\ML \to \PML \times (0, + \infty)$ also depends on $x$.

When $\Ga < \Mod(S)$ is a non-elementary subgroup of divergence type, Coulon \cite{coulon2024patterson-sullivan} and Yang \cite{yang2024conformal} showed that the $\delta_{\Ga}$-dimensional conformal measure $\nu$ of $\Ga$ is unique. Hence, we write
$$
\mu_{\Ga} := \mu_{\nu} \quad \text{on} \quad \ML
$$
in this case. Furthermore,  $\mu_{\Ga}$ is supported on the recurrence locus $\mathcal{R}_{\Ga}$.

\subsection{Horospherical foliations of $\CAT(-1)$ spaces}

Before discussing the proof, we present an application to $\CAT(-1)$ spaces. 

Let $X$ be a proper geodesic $\CAT(-1)$ space. Following \cite{Roblin2003ergodicite},  the space $\op{S}X$ of isometries $\R \to X$ serves as the role of unit tangent bundle over $X$. It admits a geodesic flow $\op{S}X \curvearrowleft \R$, which we denote by $a_t : \op{S}X \to \op{S}X$ for $t \in \R$. For $u \in \op{S}X$, its stable horosphere (or stable manifold) is the set $H^-(u)$ of $v \in \op{S}X$ such that the distance between $ua_t$ and $va_t$ tends to $0$ as $t \to + \infty$. Stable horospheres form a foliation on $\op{S}X$, called the \emph{horospherical foliation} of $X$. The space of its leaves is denoted by $\mathcal{H}$, which admits a natural action of the isometry group $\Isom(X)$.

Roblin studied measure classification for discrete subgroups of $\Isom(X)$, acting on $\mathcal{H}$. The notions of non-elementary and divergence-type subgroups are defined analogously. For a discrete subgroup $\Ga < \Isom(X)$, its recurrence locus $\mathcal{R}_{\Ga} \subset \mathcal{H}$ is also defined similarly:
$$
\mathcal{R}_{\Ga} := \{ H^-(u) \in \mathcal{H} : u a_{t} \text{ recurs to a compact subset in } \Ga \ba X \text{ as }t \to + \infty \}.
$$

We say that $\Ga$ has \emph{non-arithmetic} length spectrum if the stable translation lengths of elements in $\Ga$ generate a dense additive subgroup of $\R$.

\begin{theorem}[$\CAT(-1)$ spaces] \label{thm:maincat-1}
Let $X$ be a proper geodesic $\CAT(-1)$ space and $\Ga < \Isom(X)$ a non-elementary discrete subgroup. Suppose that 
\begin{enumerate}
    \item $\Ga$ is of divergence type, and 
    \item $\Ga$ has non-arithmetic length spectrum.
\end{enumerate}
Then
$$
\text{the $\Ga$-action on $\mathcal{R}_{\Ga}$ is uniquely ergodic.}
$$
\end{theorem}
 
An analogous problem was extensively studied for rank-one symmetric space; see Burger \cite{Burger_horoc}, Roblin \cite{Roblin2003ergodicite}, Sarig \cite{Sarig_invariant}, Ledrappier--Sarig \cite{LedrappierSarig_horo}, Landesberg--Lindenstrauss \cite{LL_Radon}, Landesberg--Lee--Lindenstrauss--Oh \cite{LLLO_Horospherical} etc. For a general $\CAT(-1)$ space, this was proved by Roblin \cite{Roblin2003ergodicite} under a stronger assumption that $\Ga$ has a finite Bowen--Margulis--Sullivan measure and guarantees existence of a certain covering. The unique ergodic measures are called the Burger--Roblin measures in these settings.

Our novelty in this $\CAT(-1)$-setting is removing those finiteness and covering assumptions, which was possible because we do not rely on mixing of continuous flows or certain dynamical systems with finite measures. 

In \cite{CK_product}, we extend our measure classification to products of $\CAT(-1)$ spaces, by considering so-called transverse subgroups. Such a product space is $\CAT(0)$ and the plethora of higher dimensional flats forbids the space from possessing some negatively curved feature, even partially. Still, our techniques developed in this paper can be extended further in that setting by employing more geometric aspects of transverse subgroups.

\subsection{More examples}
Let us mention one interesting families of examples, in addition to the mapping class groups and $\CAT(-1)$ groups. In \cite{genevois2019partially}, Genevois and  Stocker considered a geodesically complete proper $\CAT(0)$ space $X$ with a point with a $\CAT(-1)$ neighborhood, and a non-elementary subgroup $\Gamma < \Isom(X)$ acting geometrically on $X$. They proved that $\Gamma$ contains a contracting isometry whose quasi-axis passes through the aforementioned $\CAT(-1)$ neighborhood several times. In fact, this contracting isometry is squeezing. 
Meanwhile, since $\Gamma$ acts non-elementarily and geometrically on $X$, it is necessarily of divergence type \cite[Theorem B]{yang2019statistically}.

\subsection{Metric spaces with squeezing isometries}

Our proofs of 
Theorem \ref{thm:main1}, Theorem \ref{thm:main1Conv}, and Theorem \ref{thm:maincat-1} are based on the special geometric feature of the axis of  a pseudo-Anosov mapping class and the axis of a loxodromic isometry on a $\CAT(-1)$ space, which we call \emph{squeezing property} in this paper.

In the rest of the introduction, let $(X, d)$ be a proper geodesic metric space. For a geodeic $\ga \subset X$, we denote by $\pi_{\ga} (\cdot)$ the \emph{nearest-point projection}  map onto $\gamma$. We consider two notions of hyperbolicity for $\gamma$. 
\begin{itemize}
    \item We say that a geodesic $\gamma\subset X$ is \emph{contracting} if there exists $K>0$ such that, for every geodesic $\eta \subset X$, points that are $K$-deep in the convex hull of $\pi_{\gamma}(\eta)$ is $K$-close to $\eta$.
\item We say that a geodesic $\gamma \subset X$ is \emph{squeezing} if for each $\epsilon>0$ there exists $K = K(\epsilon) > 0$ such that,  for every geodesic $\eta \subset X$, points that are $K$-deep in the convex hull of $\pi_{\gamma}(\eta)$ is $\epsilon$-close to $\eta$.
\end{itemize}
We say that an isometry $g \in \Isom(X)$ is contracting (squeezing, resp.) if it admits a contracting (squeezing, resp.) axis. The squeezing property can be regarded as a quantitative version of the contracting property. See Figure \ref{figure.contractingsqueezing} for a rough sketch, and Section \ref{section:prelim} for precise definitions.

\begin{figure}[h]
\begin{tikzpicture}[scale=0.45]
    \draw[thick] (-4, 0) -- (4, 0);
    \filldraw (-4, 0) circle(2pt);
    \filldraw (4, 0) circle(2pt);

    \draw (2, 2) node {\color{red}  \tiny $\eta$};

    \draw[dashed, teal, thick] (-6, 2) .. controls (-5, 2) and (-3.5, 1) .. (-3.5, 0);
    \draw[dashed, teal, thick] (6.5, 3) .. controls (5, 3) and (3.5, 1) .. (3.5, 0);

    \filldraw[teal] (-3.5, 0) circle(2pt);
    \filldraw[teal] (3.5, 0) circle(2pt);

    \filldraw[teal] (-1, 0) circle(2pt);
    \filldraw[teal] (1, 0) circle(2pt);

    \draw[red, thick] (-6, 2) .. controls (-5, 2)  and (-3.5, 1.5) .. (-3.3, 0.9) .. controls (-3.1, 0.3) and (-2.8, 0.9) .. (-2.5, 0.9) .. controls (-2.4, 0.9) and (-2.2, 0.7) .. (-2.1, 0.6) .. controls (-2, 0.5) and (-1.8, 0.5) .. (-1.7, 0.8) .. controls (-1.6, 1.1) and (-1.5, 1.1) .. (-1.4, 0.8) .. controls (-1.3, 0.5) and (-1, 0.5) .. (-0.9, 0.8) .. controls (-0.8, 1.0) and (-0.4, 1.0) .. (-0.1, 0.7) .. controls (0.1, 0.6) and (0.6, 0.6) .. (1, 0.8) .. controls (1.4, 1) and (2, 1) .. (2.4, 0.6) .. controls (2.6, 0.4) and (3, 0.4) .. (3.2, 0.8) .. controls (4.3, 3) and (6, 3) .. (6.5, 3);

    \filldraw (-6, 2) circle(2pt);
    \filldraw (6.5, 3) circle(2pt);

    \filldraw[color=blue!80, opacity=0.3] (-1, 1) arc (90:270:1) -- (1, -1) arc(-90:90:1) -- (-1, 1);

    \draw[<->, teal, thick] (-3.5, -0.25) -- (-1, -0.25);
    \draw[<->, teal, thick] (3.5, -0.25) -- (1, -0.25);

    \draw (2.55, -0.1) node[below] {\color{teal} \tiny $\ge K$};

    \draw (4, 0) node[right] {\tiny $\ga$};

\end{tikzpicture}
\qquad
\begin{tikzpicture}[scale=0.45]
    \draw[thick] (-4, 0) -- (4, 0);
    \filldraw (-4, 0) circle(2pt);
    \filldraw (4, 0) circle(2pt);

    \draw (2, 2) node {\color{red} \tiny$\eta$};

    \draw[dashed, teal, thick] (-6, 2) .. controls (-5, 2) and (-3.5, 1) .. (-3.5, 0);
    \draw[dashed, teal, thick] (6.5, 3) .. controls (5, 3) and (3.5, 1) .. (3.5, 0);

    \filldraw[teal] (-3.5, 0) circle(2pt);
    \filldraw[teal] (3.5, 0) circle(2pt);

    \draw[red, thick] (-6, 2) .. controls (-3.5, 2)  and (-3.3, 0.3) .. (-2.7, 0.3) .. controls (-2.5, 0.2) and (2.5, 0.2) .. (2.8, 0.5) .. controls (3.3, 0.7) and (3.5, 3) .. (6.5, 3);

    \filldraw (-6, 2) circle(2pt);
    \filldraw (6.5, 3) circle(2pt); 

    \filldraw (0, 0) circle(2pt);
    \filldraw[color=blue!80, opacity=0.3] (0, 0) circle(0.5);

    \draw[<->, teal, thick] (-3.5, -0.25) -- (0, -0.25);
        \draw[<->, teal, thick] (3.5, -0.25) -- (0, -0.25);

    \draw (2, -0.1) node[below] {\color{teal} \tiny $\ge K$};

    \draw (4, 0) node[right] {\tiny $\ga$}; 

\end{tikzpicture}
\caption{A contracting geodesic (left) and a squeezing geodesic (right)} \label{figure.contractingsqueezing}
\end{figure}
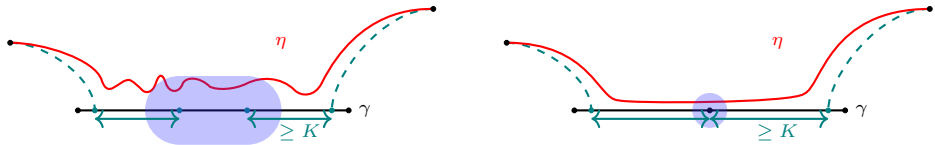

It is well known that every geodesic in a $\CAT(-1)$ space is squeezing. Hence, every loxodromic isometry is squeezing there. Furthermore, in the mapping class group viewed as the isometry group of Teichm\"uller space, every pseudo-Anosov mapping class is contracting  due to Minsky \cite{minsky1996quasi-projections}, and moreover, it is squeezing. See Section \ref{subsection:squeezepA} for further context.

We prove the aforementioned results for Teichm\"uller spaces and CAT($-1$) spaces by studying ergodic theory under the presence of squeezing isometries. We mainly consider a \emph{non-elementary} subgroup $\Ga < \Isom(X)$, i.e., a non-virtually-cyclic subgroup with a contracting isometry acting properly on $X$. The notion of divergence-type is defined analogously to the one for subgroups of $\Mod(S)$. We say that the \emph{squeezing spectrum} of $\Ga$
$$
\Spec(\Ga) := \{ \text{translation length of } g : g \in \Ga \text{ is squeezing}\} \subset \R
$$
is \emph{non-arithmetic} if $\Spec(\Ga)$ generates a dense additive subgroup of $\R$.

It remains to define an object corresponding to the space of measured laminations on a surface and to the horospherical foliation of a $\CAT(-1)$ space. We consider the product space \footnote{We use the same notation as in the $\CAT(-1)$ case, because this is precisely the horospherical foliation when $X$ is $\CAT(-1)$.}
$$
\mathcal{H} := \partial^{h} X \times \R
$$
where $\partial^{h} X$ is the horofunction boundary of $X$. Each point in $\partial^{h} X$ is identified with a 1-Lipschitz cocycle on $X$, and this induces a natural $\Isom(X)$-action on the space $\mathcal{H}$: the action on the $\R$-component is given by the cocycle determined by the $\partial^{h}X$-component. See Section \ref{subsection:horo} for details.

Theorem \ref{thm:main1} will follow from the following more general statement.

\begin{theorem}[Ergodicity] \label{thm:main2}
Let $(X, d)$ be a proper geodesic metric space and let $\Ga < \Isom(X)$ be a non-elementary subgroup. Suppose that
\begin{enumerate}
    \item $\Ga$ is of divergence type, and 
    \item the squeezing spectrum $\Spec(\Ga)$ is non-arithmetic.
\end{enumerate}
Then there exists a nonzero, $\Ga$-invariant Radon measure $\mu_{\Ga}$ on $\mathcal{H}$ such that
$$
\text{the $\Ga$-action on $(\mathcal{H}, \mu_{\Ga})$ is ergodic.}
$$
\end{theorem}

\begin{remark}

As for Theorem \ref{thm:main1},  our ergodicity result applies to normal subgroups of a divergence-type group, which are not necessarily of divergence type (Theorem \ref{thm:ergodicnormal}). 
\end{remark}

As in the setting of Teichm\"uller spaces, the measure $\mu_{\Ga}$ is explicitly constructed, and it turns out that $\mu_{\Ga}$ is supported on the region corresponding to the recurrence locus. This region is given by
$$
\La_{c}(\Ga) \times \R \subset \mathcal{H}
$$
where $\La_{c}(\Ga) \subset \partial^{h} X$ is the \emph{conical limit set} of $\Ga$ (Definition \ref{def:conical}). Delaying the construction of $\mu_{\Ga}$, we state the unique ergodicity theorem, from which Theorem \ref{thm:main1Conv} and Theorem \ref{thm:maincat-1} follow. 

\begin{theorem}[Unique ergodicity] \label{thm:main2Conv}
Let $(X, d)$ be a proper geodesic metric space and let $\Ga < \Isom(X)$ be a non-elementary subgroup such that $\Spec(\Ga)$ is non-arithmetic. Suppose that there exists a nonzero, $\Ga$-invariant Radon measure $\mu$ on $\mathcal{H}$  supported on $\La_{c}(\Ga) \times \R$. Then $\Ga$ is of divergence type and 
$$
\text{$\mu$ is a constant multiple of $\mu_{\Ga}$.}
$$
In other words, \begin{enumerate}
\item if $\Ga$ is of convergence type, then there does not exist a nonzero, $\Ga$-invariant Radon measure on $\La_{c}(\Ga) \times \R$.
\item if $\Ga$ is of divergence type, then
$$
\text{the $\Ga$-action on $(\La_{c}(\Ga) \times \R, \mu_{\Ga})$ is uniquely ergodic.}
$$
\end{enumerate}
\end{theorem}

We note that non-arithmeticity of the squeezing spectrum is essential. In general, the horospherical foliation is not uniquely ergodic  without non-arithmeticity: there are uncountably many mutually singular invariant measures on the horospherical foliation of a standard Cayley graph of a free group.
For non-elementary subgroups of mapping class groups, the non-arithmeticity was proved by Gekhtman \cite{gekhtman2013dynamics} and  is elaborated in the work of Gekhtman and Ma \cite{gekhtman2023dynamics}. In this sense, we use not only the ``partial hyperbolicity" of $\Mod(S)$ but also certain features of $\Mod(S)$ shared with Zariski dense subgroups of Lie groups.

\subsection{On the proof}
A key point of our argument is to relate $\Gamma$-invariant measures on the horospherical foliation with $\Gamma$-conformal measures on the horoboundary. As explained earlier, we establish this connection by showing the quasi-invariance of a given measure under the scaling of measured laminations.
The general theory that such an extra quasi-invariance gives a specific from of the measure was given in \cite{aaronson2002invariant} for skew-product dynamical systems, and an explicit negatively-curved setting was studied in \cite{sarig2004invariant}. This strategy for the Teichm{\"u}ller space was first studied by Hamenst{\"a}dt \cite{hamenstadt2009invariant} using train track theory.

Our contribution is to execute this strategy for subgroups of mapping class groups, using coarse geometry of the Teichm{\"u}ller space instead of train track theory. This is thanks to the recent advance of the Patterson--Sullivan theory on general metric spaces. Initially formulated by Patterson \cite{patterson1976the-limit} and Sullivan \cite{Sullivan1979density}, the Patterson--Sullivan theory provides a powerful tool to relate growth, orbit counting, and conformal densities for  groups acting on hyperbolic spaces. The general framework for $\CAT(-1)$ spaces is due to Roblin \cite{Roblin2003ergodicite}. Coornaert studied the Patterson--Sullivan theory for Gromov hyperbolic metrics that may not be $\CAT(-1)$ \cite{coornaert1993mesures}. Recently, Roblin's and Coornaert's theories were generalized to metric spaces with contracting isometries, independently by Coulon \cite{coulon2024patterson-sullivan} and  Yang \cite{yang2024conformal}.

An important component of the proof is to relate conformal measures with ergodic invariant measures on $\ML$. In contrast to previous literatures which heavily use dynamics of continuous flows, we do not make use of dynamical properties (ergodic theorems, mixing, equidistribution, etc.) of the geodesic/unipotent flows on the Teichm\"uller space, and the theory of train tracks and the curve complex. Instead, 
we employ a geometric apporach based on the contracting property of axes of pseudo-Anosov mapping classes due to Minsky \cite{minsky1996quasi-projections}, and our observation that the contracting property can be enhanced to the \emph{squeezing} property (Proposition \ref{prop.pAsqueezing}).

To elaborate further, we first show that every invariant  Radon measure on the recurrent measured laminations is moreover supported on a smaller subset of measured laminations, whose projective classes are accumulated by orbits fellow traveling the translates of a chosen pseudo-Anosov axis in the Thurston compactification (Theorem \ref{thm:radonCharge}). This step reflects complicated nature of the shape of the Teichm{\"u}ller space compared to that of Gromov hyperbolic spaces and homogeneous spaces. Indeed, this step does not appear in previous works on $\CAT(-1)$ or homogeneous settings that rely on dynamics; note that the recurrent measured laminations correspond to conical limit points, and the subset with ``fellow traveling accumulation'' is much smaller.

The first step above uses the contracting property of pseudo-Anosov axes. Next, using a stronger \emph{squeezing} property, we conduct a finer investigation on ``fellow traveling accumulations'' along the translates of a chosen pseudo-Anosov axis. Together with the concentration of the measure on fellow traveling regions in the first step, we show the quasi-invariance of the measure under the scaling by the stretch factor of each pseudo-Anosov mapping class (Theorem \ref{thm:trbytrlengthqi}).

Since we do not rely on dynamics of geodesic/unipotent flows, we can deal with \emph{general non-elementary subgroups} even under the non-homogeneous and non-negatively curved geometric feature of the Teichm\"uller space.

Finally, we emphasize that the current techniques can be extended further to yield horospherical measure classification results for certain infinite-covolume higher-rank discrete subgroups. See \cite{CK_product} for details.

\subsection{Non-geodesic spaces}

There is an important class of  metric spaces that we will not discuss in detail. Let $\Gamma$ be a non-elementary relatively hyperbolic group. Let $m$ be a finitely supported, admissible, symmetric probability measure on $\Gamma$. There is a left-invariant metric $d_{m}$ on $\Gamma$ associated with $m$, called the \emph{Green metric}. This metric is quasi-isometric to the word metric and is Gromov hyperbolic, but is not geodesic in general. Nonetheless, it is \emph{roughly geodesic}. In a sense, every loxodromic element of $\Gamma$ is a \emph{squeezing} isometry with respect to $d_{m}$, even if it does not possess an axis. Hence, under the non-arithmeticity assumption, we expect that there exists a unique nonzero, $\Gamma$-invariant ergodic Radon measure on the $d_{m}$-horospherical foliation. This unique measure is in the class of $(\emph{$m$-stationary measure}) \times \Leb_{\R}$.

In order to apply our theory to Green metrics, one needs to replace geodesics with rough geodesics. We leave it for future studies.

\subsection{Organization} 

In Section \ref{sec:Teichmueller}, we present a brief overview of the Teichm\"uller space, especially about its boundary. We define the Hubbard--Masur coordinates for the space of measured laminations in Section  \ref{sec:HubbardMasur}. In Section \ref{sec:pAsqueezing}, we discuss some aspects of pseudo-Anosov mapping classes, including contracting and squeezing properties of their axes. We discuss the geometry of a general metric space with contracting and squeezing isometries in  Section \ref{section:prelim}.
Section \ref{section:ps} is devoted to the Patterson--Sullivan theory with squeezing isometries.
In Section \ref{sec:UE}, we prove the unique ergodicity (Theorem \ref{thm:main2Conv}). The ergodicity (Theorem \ref{thm:main2}) is proved in 
Section \ref{sec:ergodicity}.
In Section \ref{sec:mcgsubgroup}, we deduce our measure classifications on the space of measured laminations (Theorem \ref{thm:main1}, Theorem \ref{thm:main1Conv}, and Theorem \ref{thm:mainCC}).
Our orbit closure classification is presented in  Section \ref{sec:orbitclosure}.

\subsection*{Acknowledgements} The authors would like to thank Yair Minsky and Hee Oh for helpful conversations and for many  useful comments on the earlier
version of this paper. Kim extends his special gratitude to his Ph.D. advisor Hee Oh for her encouragement and guidance.

 This work was initiated when Choi was working at Cornell University. 
Kim thanks Cornell University for the hospitality during the visit in April 2025, where we had many interesting discussions.

Choi was supported by the Mid-Career Researcher Program (RS-2023-00278510) through the National Research Foundation funded by the government of Korea, and by the KIAS individual grant (MG091901) at KIAS.

\section{Teichm\"uller theory} \label{sec:Teichmueller}

In this section, we review basic Teichm\"uller theory. Our explanation is minimal and we refer interested readers to \cite{masur2009geometry}, \cite{farb2012primer} and \cite{gekhtman2023dynamics}.

In the rest of this paper, let $S$ be a connected orientable surface of genus $g$ and with $p$ punctures with $3g - 3+ p \ge 1$.
 Recall from the introduction that the Teichm\"uller spcae $\T = \T(S)$ is the space of all marked Riemann surface structures on $S$. More precisely, 
$$
\T := \left\{(X, f) : 
\begin{matrix}
X \text{ is a Riemann surface and} \\
 f : S \to X \text{ is a diffeomorphism (marking)} 
\end{matrix} \middle\} \right/ \sim
$$
where $(X, f) \sim (Y, h)$  if $h \circ f^{-1} : X \to Y$ is isotopic to a bi-holomorphic diffeomorphism. We equip $\T$ with the \emph{Teichm{\"u}ller metric} $d_{\T}$, defined as
$$
d_{\T}((X, f), (Y, h)) := \frac{1}{2} \log \inf \left\{ K(\phi) : 
\begin{matrix}
\text{quasiconformal } \phi : X \to Y \\
\text{isotopic to } h \circ f^{-1}
\end{matrix} \right\}
$$
where $K(\phi) \ge 1$ denotes the quasiconformal dilatation of $\phi$.
The metric $d_{\T}$ is Finsler, proper, uniquely geodesic but not Riemannian.
The Teichm\"uller space $\T$ is homeomorphic to $\R^{6g-6+2p}$.

The mapping class group $\Mod(S)$ of $S$ is the group of isotopy classes of orientation-preserving diffeomorphisms of $S$. It acts on $(\T, d_{\T})$ by
$$
g \cdot (X, f) := (X, f \circ g^{-1}) \quad \text{for } g \in \Mod(S) \text{ and } (X, f) \in \T.
$$
We note that $\Mod(S)$ is more or less the full isometry group of $(\T, d_{\T})$, as shown by Royden \cite[Theorem 2]{royden1971automorphisms} and by  Earle and  Kra (\cite{earle1974on-holomorphic}, \cite{earle1974on-isometries}). See also Ivanov's work \cite{MR1819186}. The quotient
$$
\mathcal{M} = \mathcal{M}(S) := \Mod(S) \ba \T
$$
is the moduli space of Riemann surface structures on $S$.

Note that in the exceptional case when $3g - 3 + p = 1$ (i.e., $S$ is once-punctured torus or 4-punctured sphere), $\T = \H^2$. Hence, the major case of our interest is when $3g - 3 + p \ge 2$.

 \subsection{Thurston boundary} Recall the space $\ML = \ML(S)$ of measured laminations on $S$, which is homeomorphic to $\R^{6g - 6 + 2p} \smallsetminus \{0\}$. We denote by $\PML = \PML(S)$ the space of projective measured laminations on $S$, where the projectivization is given by the scaling of transverse measures. The space $\PML$ can also be identified with the unit sphere in $\ML \simeq \R^{6g - 7 + 2p} \smallsetminus \{0\}$, and hence $\PML \simeq \S^{6g- 7 + 2p}$.  Thurston compactified the Teichm\"uller space using $\PML$ and showed that the $\Mod(S)$-action on $\T$ continuously extends to the compactification $\T \cup \PML$. In this regard, $\PML$ is also referred to as the \emph{Thuston boundary} of $\T$ (\cite{thurston1988classification}, \cite{MR1435975}).

There is another notion called \emph{measured foliation} on surfaces. The space of measured foliations (up to equivalence) on $S$ is denoted by $\MF = \MF(S)$, and we also denote its projectivization by $\PMF = \PMF(S)$. In fact, $\PML$ and $\PMF$ give rise to the same compactification of $\T$. By this we mean that not only the two spaces are homeomorphic via a map $\mathcal{PML} \to \mathcal{PMF}$, but also that the convergence in $\T \cup \PML$ is equivalent to the convergence in $\T \cup \PMF$. 

While we stick to measured laminations throughout, they are interchangeable with measured foliations depending on readers' preference.

 \subsection{Gardiner--Masur boundary}
 
 Gardiner and  Masur proposed  another compactification of the Teichm\"uller space using the so-called \emph{Gardiner--Masur boundary} $\partial^{GM} \T$ in \cite{gardiner1991extremal}. This compactification $\T \cup \partial^{GM} \T$ is now called the Gardiner--Masur compactification. The $\Mod(S)$-action on $\T$ continuously extends to $\T \cup \partial^{GM} \T$ as well.

Both the Thurston compactification and the Gardiner--Masur compactification are obtained by embedding $\T$ into the space $\R_{\ge 0}^{\mathcal{S}}$ of functions on the set $\mathcal{S}$ of isotopy classes of essential simple closed curves on $S$, and then taking the closure in the projective space $\mathbb{P} (\R_{\ge 0}^{\mathcal{S}} )$. Recall that  $\T$ can also be viewed as the space of all marked hyperbolic structures on $S$, Thurston embedded $\T$ using the hyperbolic length, and Gardiner--Masure embedded $\T$ using the extremal length (see Equation \eqref{eqn:extlength}). In this regard, Gardiner and Masur proved that the Thurston boundary $\PML$ sits in the Gardiner--Masur boundary $\partial^{GM} \T$ as a proper subset \cite[Theorem 7.1]{gardiner1991extremal}.

 \subsection{Uniquley ergodic laminations and boundary comparison}

 A measured lamination is called \emph{uniquely ergodic} if the underlying geodesic lamination admits a unique transverse measure up to scaling. We denote by $\UE = \UE(S) \subset \PML$ the subset of projective classes of uniquely ergodic measured laminations on $S$.

 On the subset $\UE \subset \PML$, the accumulation of points in $\T$ is well-behaved: if $[\xi] \in \UE$, then  for each $x \in \T$, there exists a \emph{unique} Teichm{\"u}ller (geodesic) ray $\gamma_{x}: [0, +\infty) \rightarrow \T$ based at $x$ such that $\lim_{t \rightarrow +\infty} \gamma_{x}(t) = [\xi]$ in the Thurston compactfication (\cite{hubbard1979quadratic}, \cite{masur1980uniquely}, \cite{masur1982boundaries}). 
Furthermore, for every $x, y \in \T$, two rays $\ga_x$ and $\ga_y$ are asymptotic, i.e., there exists $T \in \R$ such that $\lim_{t \rightarrow +\infty} d_{\T}(\ga_x (t), \ga_y(t + T)) = 0$ \cite{masur1980uniquely}. 

Moreover, while the identity map $\T \to \T$ does not continuously extends to an embedding of the Thurston compactification into the Gardiner--Masur compactification, it extends to a homeomorphism
\begin{equation} \label{eqn:UEextension}
    \begin{tikzcd}
    \T \cup \UE \arrow[d, phantom, sloped, "\subset"] \arrow[r] & \T \cup \UE  \arrow[d, phantom, sloped, "\subset"] \\
    \T \cup \PML & \T \cup \partial^{GM}\T
\end{tikzcd}
\end{equation}
as proved by Masur \cite{masur1982boundaries} and recoverd by Miyachi \cite{miyachi2013teichmuller}. 

Thanks to these facts, we can regard $\UE$ as a topological subspace of $\PML$ and as a topological subspace of $\partial^{GM} \T$ at the same time. As $\partial^{GM} \T$ is identified with the horofunction boundary of $(\T, d_{\T})$, this enables us to employ the theory of horofunctions in studying  Radon measures on $\mathcal{ML}$ that are supported on uniquely ergodic ones.

A sufficient condition for a measured lamination to be uniquely ergodic is the recurrence of the associated Teichm\"uller geodesic ray to a compact subset in the moduli space $\mathcal{M} = \Mod(S) \ba \T$. Namely, the \emph{Masur criterion} due to  Masur  \cite[Theorem 1.1]{masur1992hausdorff} asserts that for a Teichm\"uller geodesic ray  $\ga :~[0, + \infty) \to \T$, if there exist a compact subset $K \subset \mathcal{M}$ and a sequence $t_n \to + \infty$ such that $\ga(t_n) \in \T$ projects into $K \subset \mathcal{M}$ for all $n \in \N$, then $\gamma$ converges to a uniquely ergodic lamination in both the Thurston compactification and the Gardiner--Masur compactification.

 \subsection{Busemann cocycles}

For $x, y \in \T$, the function $d_{\T}(x, \cdot) - d_{\T}(y, \cdot)$ on $\T$ may not continuously extend to the Thurston boundary $\PML$. Nevertheless,  Miyachi showed that if $\xi \in \ML$ is uniquely ergodic, then for every sequence $\{z_n\}_{n \in \N} \subset \T$ converging to $[\xi] \in \UE$ the limit
 $$
\beta_{\xi} (x, y) := \lim_{n \to + \infty} d_{\T}(x, z_n) - d_{\T}(y, z_n)
 $$
exists and is independent of the choice of $\{z_{n}\}_{n \in \N}$ \cite[Corollary 1]{miyachi2013teichmuller}. The function $\beta$ is called the \emph{Busemann cocycle}. Indeed, it satisfies the cocycle relation: for $w, x, y \in \T$,
 $$
 \beta_{\xi}(w, y) = \beta_{\xi}(w, x) + \beta_{\xi}(x, y).
 $$
 Moreover, for $g \in \Mod(S)$, it is easy to see
 $$
\beta_{g \xi}(g x, g y) = \beta_{\xi}(x, y).
 $$

Although $\beta_{\xi}$ only depends on the projective class $[\xi]$, not $\xi$, we use the notation $\beta_{\xi}$ for a later purpose.

\section{Hubbard--Masur coordinates for measured laminations} \label{sec:HubbardMasur}

In \cite{hubbard1979quadratic},  Hubbard and Masur proved that given a point $x_0 \in \T$, the space $\mathcal{Q}(S, x_0)$ of holomorphic quadratic differentials on  Riemann surface $(S, x_0)$ is homeomorphic to $\ML$, where the homeomorphism is given by the vertical measured foliation of a given holomorphic quadratic differential on $x_0$.

This enables us to consider certain coordinate systems on $\ML$, which we call Hubbard--Masur coordinates, by specifying a homeomorphism
$$
\ML \simeq \PML \times \R.
$$
For this, we employ the notion of extremal lengths.

\subsection{Extremal lengths}

Given a point $x \in \T$ and the isotopy class $\alpha$ of a simple closed curve on $ S$, the \emph{extremal length} of $\alpha$ on Riemann surface $(S, x)$ is defined as
\begin{equation} \label{eqn:extlength}
\Ext_x(\alpha) := \sup_{\sigma} \frac{\ell_{\alpha}(\sigma)^2}{\operatorname{Area}(\sigma)}
\end{equation}
where the supremum is over all metrics $\sigma$ conformally equivalent to $x$, and $\ell_{\alpha}(\sigma)$ is the length of $\alpha$ in the metric $\sigma$. The extremal length continuously extends to the function 
$$\Ext_x : \ML \to \R$$
 such that $\Ext_x(t \xi) = t^2 \xi$ for all $\xi \in \ML$ and $t > 0$, by Kerckhoff \cite[Proposition 3]{kerckhoff1980asymptotic}. 

Miyachi \cite[Corollary 2]{miyachi2013teichmuller} and  Walsh \cite[Section 6]{walsh2019the-asymptotic} showed that for $[\xi] \in \UE$ and $x, y \in \T$, the following holds:
\begin{equation} \label{eqn:buseext}
\beta_{\xi}(x, y) = \frac{1}{2} \log \frac{\Ext_{x}(\xi)}{\Ext_{y} (\xi)}.
\end{equation}

 \subsection{Hubbard--Masur coordinates} \label{subsec:HM}
 Fixing a basepoint $x_0 \in \T$, we now define the \emph{Hubbard--Masur coordinates} for $\ML$ (with respect to $x_0$) as follows:
$$
\begin{aligned}
\HM : \ML  & \to \PML \times \R \\
\xi & \mapsto \left([\xi], \,\,\frac{1}{2} \log \Ext_{x_0}(\xi)\right)
\end{aligned}
$$
which is a homeomorphism. We define the $\Mod(S)$-action on $\PML \times \R$ by
$$
g \cdot ([\xi], t) := \left(g [\xi], \,\,t + \frac{1}{2} \log \frac{\Ext_{g^{-1} x_0}(\xi)}{\Ext_{x_0} (\xi)} \right)
$$
for $g \in \Mod(S)$, $\xi \in \ML$,  and $t \in \R$. Note that this is well-defined independent of the choice of $\xi \in \ML$. 

By Equation \eqref{eqn:buseext}, this can be rephrased in terms of the Busemann function. For a uniquely ergodic lamination $\xi$ and a mapping class $g$,
$$
g \cdot ([\xi], t) = (g [\xi], \,t + \beta_{\xi}(g^{-1} x_0, x_0)).
$$
We now show the equivariance of the action.

\begin{proposition}
    The map $\HM : \ML \to \PML \times \R$ is $\Mod(S)$-equivariant.
\end{proposition}

\begin{proof}
    Let $g \in \Mod(S)$ and $\xi \in \ML$.
It follows from the definition that
$$
\HM(g \xi) = \left( g [\xi], \,\,\frac{1}{2} \log \Ext_{x_0} (g \cdot \xi) \right).
$$
Since $\Ext_{x_0} (g  \xi) = \Ext_{g^{-1} x_0} (\xi)$, we have
$$
\HM(g \xi) = \left( g  [\xi], \,\,\frac{1}{2} \log \Ext_{x_0} (\xi) + \frac{1}{2} \log \frac{\Ext_{g^{-1} x_0}(\xi)}{\Ext_{x_0} (\xi)} \right). \qedhere
$$
\end{proof}

\subsection{Liu--Su and Walsh's coordinates on the horofunction boundary}\label{subsection:liuSuWalsh}

Liu and Su proved in \cite{liu2014the-horofunction} and Walsh proved in \cite{walsh2019the-asymptotic} that the horofunction boundary for $(\T, d_{\T})$ is equal to the Gardiner--Masur boundary. In both works, fixing a basepoint $x_0 \in \T$, the authors constructed a continuous injection \[
 \T \cup \partial^{GM} \T \rightarrow \{\textrm{1-Lipschitz functions on $\T$ vanishing at $x_0$}\},
\]
and its restriction to the Gardiner--Masur boundary \[
 \partial^{GM} \T \rightarrow \{\textrm{horofunctions on $\T$ vanishing at $x_0$}\}
\]
is a homeomorphism. See Section \ref{subsection:horo} for more discussion on horofunctions.

 \section{Squeezing property of pseudo-Anosov axes} \label{sec:pAsqueezing}

We now discuss elements of $\Mod(S)$, the mapping classes of $S$. The celebrated Nielsen--Thurston classification (\cite{thurston1988classification}, \cite{1979travaux}) asserts that a mapping class $\varphi \in \Mod(S)$ is either
\begin{itemize}
    \item periodic, i.e., $\varphi^n = \id$ for some $n \in \N$,
    \item reducible, i.e., there exists a multicurve on $S$ invariant under $\varphi$, or
    \item pseudo-Anosov, i.e., there exists a pair of transverse measured laminations $\xi^+, \xi^- \in \ML$ and $\la  > 1$ such that 
    \begin{equation} \label{eqn:pAinv}
        \varphi(\xi^+) = \la \cdot \xi^+ \quad \text{and} \quad \varphi(\xi^-) = \frac{1}{\la} \cdot \xi^-.
    \end{equation}
    The measured laminations $\xi^+$ and $\xi^-$ are called unstable and stable measured laminations respectively, and the constant $\la > 1$ is called the stretch factor of $\varphi$.
\end{itemize}

Among the three categories we are particularly interested in pseudo-Anosov mapping classes. Let $\varphi \in \Mod(S)$ be a pseudo-Anosov mapping class. We summarize some standard facts: 

\begin{enumerate}
\item It follows from Equation \eqref{eqn:pAinv} that the unstable and stable measured laminations of $\varphi$ give two fixed points $[\xi^+], [\xi^-] \in \PML$ in the Thurston boundary. Moreover, on the Thurston compactification $\T \cup \PML$, $\varphi$ exhibits the north-south dynamics.

More precisely, for each compact subset $K \subset (\T \cup \PML) \smallsetminus \{ [\xi^-] \}$, we have $\varphi^{n} K \to [\xi^+]$ as $n \rightarrow +\infty$. Similarly, for each compact subset $K \subset (\T \cup \PML) \smallsetminus \{ [\xi^+] \}$, we have $\varphi^{n} K \to [\xi^-]$ as $n \rightarrow -\infty$. 

\item The projective measured laminations $[\xi^+], [\xi^-] \in \PML$ are called attracting and repelling fixed points of $\varphi$ repsecitvely, and they are in fact uniquely ergodic, i.e., $[\xi^+], [\xi^-] \in \UE$. 

\item 
There exists a unique bi-infinite Teichm{\"u}ller geodesic $\gamma \subset \T$ whose endpoints are $[\xi^{\pm}] \in \PML$ in the Thurston compactification. Moreover, $\ga$ is invariant under $\varphi$, and the action of $\varphi$ on $\ga$ is the translation by $\log \la$, where $\la$ is the stretch factor of $\varphi$. 

The invariant geodesic $\ga$ is called the \emph{axis} of pseudo-Anosov $\varphi$.
\end{enumerate} 

The first two items are part of Thurston's proof of the Nielsen--Thurston classification using the Thurston compactification. We refer the readers to Thurston's exposition \cite{thurston1988classification} and  textbooks (\cite{1979travaux}, \cite{farb2012primer}). The last item is proved in \cite{1979travaux}, \cite[Theorem 3.1]{gardiner1991extremal}, and \cite[Theorem 9.2]{marden1993a-characterization}.

\subsection{Contracting property and squeezing property} \label{subsection:squeezepA}

The two central dynamical notions in this paper are contracting and squeezing properties of isometries. In \cite{minsky1996quasi-projections},  Minsky proved that Teichm\"uller geodesics \emph{precompact} in the moduli space $\mathcal{M} = \Mod(S) \ba \T$ are contracting. More precisely, for each compact subset $K \subset \mathcal{M}$, there exists $D = D(K) > 0$ such that every geodesic in $\T$ that projects into $K \subset \mathcal{M}$ is  $D$-contracting.  

Recall that a pseudo-Anosov mapping class has the invariant Teichm{\"u}ller geodesic, called \emph{axis}, which descends to a closed loop in $\mathcal{M}$. Therefore, axes of pseudo-Anosov mapping classes are contracting. They in fact enjoy a stronger property, 
\emph{squeezing} property, thanks to their peridocity.
This essentially follows from Minsky's contraction theorem in \cite{minsky1996quasi-projections} and Masur's stability theorem in \cite{masur1980uniquely}. We give a proof for completeness.

\begin{proposition} \label{prop.pAsqueezing}
    The axis of a pseudo-Anosov mapping class is squeezing.
\end{proposition}

\begin{proof}
    Let $\varphi \in \Mod(S)$ be pseudo-Anosov and denote by $\ga \subset \T$ its axis. We fix a unit-speed parametrization $\ga : \R \to \T$. Suppose to the contrary that $\ga$ is not squeezing. Then there exists $\epsilon > 0$ such that for each $n \in \N$, there exist sequences $\{x_n\}_{n \in \N}, \{y_n\}_{n \in \N} \subset\T$ and $\{t_n\}_{n \in \N} \subset \R$ such that $\ga(t_n - a_n) \in \pi_{\ga}(x_n)$ and $\ga(t_n + b_n) \in \pi_{\ga}(y_n)$ for some $a_n, b_n > n$ while $d_{\T}( \ga(t_n), [x_n, y_n]) \ge \epsilon$. Since $\varphi$ acts on $\ga$ by a translation, we may assume that $t_n$ is bounded.

    By \cite[Contraction Theorem]{minsky1996quasi-projections}, $\ga$ is contracting. Hence by applying \cite[Lemma 2.2]{chawla2023genericity} (see Lemma \ref{lem:BGIPFellow}), there exists $C > 0$ such that for each $n \in \N$, there exist $w_n, z_n \in [x_n, y_n]$ satisfying
    \begin{itemize}
        \item $d_{\T}(w_n, \ga(t_n - a_n)) \le C$,
        \item $d_{\T}(z_n, \ga(t_n  + b_n)) \le C$, and
        \item $[w_n, z_n]$ is contained in the $C$-neighborhood of $\ga$.
    \end{itemize}
 Since $a_n, b_n \to + \infty$ and $t_n$ is bounded, after passing to a subsequence, $[w_n, z_n]$ converges to a bi-infinite geodesic contained in the $2C$-neighborhood of $\ga$. Since both endpoints of $\ga$ are contained in $\UE$, it follows from \cite[Lemma 1.4.2]{kaimanovich1996poisson} that the limit of $[w_n, z_n]$ is the bi-infinite geodesic between endpoints of $\ga$, which must be $\ga$.
 
 On the other hand, $t_n$ is bounded and $d_{\T}(\ga(t_n), [w_n, z_n]) \ge \epsilon$ for all $n \in \N$, which is a contradiction. This finishes the proof.
\end{proof}

The squeezing property can be considered as a version of $\CAT(-1)$ property along special directions. It is sensible to generalize some dynamical phenomena in $\CAT(-1)$ spaces to metric spaces with squeezing isometries. For example, in \cite{choi2022random5} the first author studied continuity and differentiability of the drift of a random walk on the Teichm{\"u}ller metric using the squeezing property of pseudo-Anosov mapping classes.

\section{Contracting and squeezing isometries}\label{section:prelim}

The previous sections tell us that Teichm{\"u}ller space is a metric space with squeezing isometries. In this section, we further develop this perspective.

Throughout this section, let $(X, d)$ be a proper geodesic metric space and fix a basepoint $x_{0} \in X$. For $x, y \in X$, we denote by $[x, y] \subset X$ an arbitrarily chosen geodesic connecting $x$ to $y$. For $w, z \in [x, y]$, we intrinsically assume $[w, z]$ to be a segment of $[x, y]$. Every parametrization of a geodesic is of unit speed. We denote the isometry group of $X$ by $\Isom(X)$.

We say that two geodesics $\gamma_1, \gamma_2 \subset X$ are $C$-equivalent if their Hausdorff distance is at most $C$ and if their beginning/ending points are pairwise $C$-close.
When two reals $a, b \in \R$ differ by at most $C$, we write $a =_{C} b$.

\subsection{Contracting subsets} \label{subsection:contracting}
For a closed subset $A \subset X$, we denote by $\pi_{A}(\cdot) : X \rightarrow 2^{A}$ the \emph{nearest point projection}. That means, we define \[
\pi_{A}(x) := \left\{ a \in A : d(x, a) = \inf_{z \in A} d(x, z)\right\}.
\]

\begin{definition}\label{dfn:contracting}
Let $C \ge 0$. We say that a closed subset $A\subset X$ is \emph{$C$-contracting} if for every geodesic $\gamma \subset X$ with $d(\gamma, A) \ge C$ we have $\diam \pi_{A}(\gamma)  \le C$. We say that $A$ is \emph{(strongly) contracting} if it is $C$-contracting for some $C\ge0$.
\end{definition}

By definition, if $A \subset X$ is $C$-contracting, then $\diam \pi_A(x) \le 2C$ for all $x \in X$. It is also easy to see that if $A \subset X$ is $C$-contracting and $x, y \in X$ are in the $C$-neighborhood of $A$, then $[x, y]$ is contained in the $2.5C$-neighborhood of $A$.

As an example, every geodesic in a $\delta$-hyperbolic space is $C(\delta)$-contracting for some $C(\delta) \ge 0$ depending only on $\delta$. 
In particular, geodesics in a simplicial tree or $\mathbb{H}^{2}$ are contracting with a uniform contracting constant. More generally, every $(K, B)$-quasigeodesic in a $\delta$-hyperbolic space is $C(\delta, K, B)$-contracting for a constant $C(\delta, K, B) \ge 0$.

The following is a reminiscent of the Morse lemma in a Gromov hyperbolic space. We use the version in \cite{chawla2023genericity} due to its conciseness, but we note that similar results were already observed in (\cite[Proposition 10.2.1]{coornaert1990geometrie}, \cite[Lemma 2.4, 2.5]{sisto2013projections}, \cite[Proposition 2.9, Lemma 2.10, Lemma 2.11]{arzhantseva2015growth}, \cite[Proposition 3.1]{yang2014growth}).

\begin{lem}[{\cite[Lemma 2.2]{chawla2023genericity}}]\label{lem:BGIPFellow}
Let $\ga \subset X$ be a $C$-contracting geodesic and $x, y \in X$. Suppose that $\diam \pi_{\ga}([x, y]) > C$. Then there exist points $p, q \in [x, y]$, with $p$ closer to $x$ than $q$, such that
\begin{itemize}
\item  $\pi_{\ga}([x, y])$ and $[p, q]$ are $4C$-equivalent,
\item $\diam(\pi_{\ga}([x, p]) \cup \{p\}) \le 2C$,
\item $\diam(\pi_{\ga}([q, y]) \cup \{q\}) \le 2C$, and
\item for all $x' \in \pi_{\ga}(x)$ and $y' \in \pi_{\ga}(y)$, $[x', y']$  and $[p, q]$ are $10C$-equivalent. 
\end{itemize} 
\end{lem}

The following is an immediate corollary of Lemma \ref{lem:BGIPFellow}.

\begin{cor}\label{cor:BGIPFellow}
For a $C$-contracting geodesic $\gamma \subset X$, the following holds.
\begin{enumerate}
\item The map $\pi_{\gamma}(\cdot)$ is $(1, 4C)$-Lipschitz: for each $x, y \in X$,
\[
\diam \pi_{\gamma} (\{x, y\}) \le d(x, y) + 4C.
\]
\item Let $x \in X$ and $\gamma(t) \in \pi_{\gamma}(x)$. Then for every $s \in \R$, we have 
\begin{equation}\label{eqn:approxDist}
d(x, \gamma(s)) =_{4C} d(x, \gamma(t)) + |t-s|.
\end{equation}
\end{enumerate}
\end{cor}

\subsection{Contracting isometries}

\begin{definition}\label{dfn:axial}
We say that an isometry $g \in \Isom(X)$ is \emph{axial} if there exists a bi-infinite geodesic $\gamma : \R \to X$ invariant under $g$ such that $$g \cdot \ga(t) = \ga (t + \tau_g) \quad \text{for all } t \in \R$$
for some $\tau_g > 0$. We call $\gamma$ an \emph{axis} of $g$ and $\tau_g$ the \emph{translation length} of $g$.

An axial isometry $g \in \Isom(X)$ is called \emph{$C$-contracting} for $C \ge 0$ if it has a $C$-contracting axis.
\end{definition}

Given an axial isometry $g \in \Isom(X)$, note that 
$$
\tau_g = \lim_{n \to +\infty} \frac{d(x, g^n x)}{n} > 0 \quad \text{for each } x \in X.
$$
Then we can observe the following:
$$
\tau_g  = \inf_{x \in X} d(x, g x) \quad \text{and} \quad \tau_{g^k} = |k| \tau_g \quad \text{for each } k \in \Z.
$$
For each $h \in \Isom(X)$, $hgh^{-1}$ is also axial and $\tau_{hgh^{-1}} = \tau_g$.

\subsection{Squeezing isometries}

\begin{definition}\label{dfn:squeezing}
We say that a bi-infinite geodesic $\gamma : \mathbb{R} \rightarrow X$ is \emph{squeezing} if for each $\epsilon>0$ there exists $L = L(\epsilon)>0$ such that for each $x, y \in X$ and $t \in \R$ with $\ga(t - a) \in \pi_{\ga}(x)$ and $\ga(t + b) \in \pi_{\ga}(y)$ for some $a, b \ge L$, we have
$$
d\left([x, y], \ga(t)\right) \le \epsilon.
$$
We call an axial isometry $g \in \Isom(X)$ \emph{squeezing} if it has a squeezing axis.
\end{definition}

By definition, squeezing isometries are contracting. Note that both squeezing and contracting properties are invariant under conjugations. If an isometry $g \in \Isom(X)$ is squeezing, then it has a unique axis (up to reparametrization). We denote it by $\A_g$. For $h \in \Isom(X)$, $\A_{hgh^{-1}} = h \A_g$.

As we will see later, squeezing geodesics are well-suited for studying horofunctions due to the following lemma.

\begin{lem}\label{lem:squeezing}
Let $\gamma : \R \to X$ be a squeezing geodesic. Fix $\epsilon >0$ and let $L = L(\epsilon) > 0$ as in Definition \ref{dfn:squeezing}. Let $x_{1}, x_{2}, y_{1}, y_{2} \in X$ and suppose that for some $t \in \R$, we have  \[
\pi_{\gamma}(x_{i}) \cap \gamma \left((-\infty, t-L]\right) \neq \emptyset \quad \text{and} \quad \pi_{\gamma}(y_{i}) \cap \gamma\left( [t+L, +\infty)\right) \neq \emptyset \quad \text{for }i = 1, 2.
\]
Then we have
 \[
d(x_{1}, y_{1}) - d(x_{1}, y_{2}) =_{8\epsilon} d(x_{2}, y_{1}) - d(x_{2}, y_{2}).
\]
\end{lem}

\begin{proof}
Let $i \in \{1, 2\}$. By the squeezing property, there exists $p \in [x_{i}, y_{1}]$ and $q \in [x_{i}, y_{2}]$ that are $\epsilon$-close to $\gamma(t)$. By the triangle inequality, we have \[
d(x_{i}, y_{1}) - d(x_{i}, y_{2}) =_{4\epsilon} d\left(\gamma(0), y_{1}\right) - d\left(\gamma(0), y_{2}\right).
\]
This gives the desired estimate.
\end{proof}

\subsection{Alignment}
We denote the closed $K$-neighborhoods  by $\mathcal{N}_{K}(\cdot)$.

\begin{definition}[Alignment]
Let $w, x, y, z \in X$. For a geodesic $[x, y] \subset X$ and $K \ge 0$, we say that the sequence $(w, [x, y])$ is \emph{$K$-aligned} if \[
\pi_{[x, y]}(w) \subset \mathcal{N}_{K}(x).
\]
Similarly, we call that the sequence $([x, y], z)$ is \emph{$K$-aligned} if $(z, [y, x])$ is $K$-aligned.

Finally, we say that the sequence $(w, [x, y], z)$ is \emph{$K$-aligned} if both sequences $(w, [x, y])$ and $([x, y], z)$ are $K$-aligned. See Figure \ref{fig:alignment}.

\end{definition}

\begin{figure}[h]

\begin{tikzpicture}[scale=0.8]

\draw[very thick] (-2.6, 0) -- (2.6, 0);

\draw[dashed, thick] (-3.5, 3) -- (-1.9, 0.2);
\draw[dashed, thick] (-3.5, 3) -- (-2.4, 0.2);
\draw[thick] (-2.55, 0.3) -- (-2.15, 0) -- (-1.75, 0.3);

\begin{scope}[xscale=-1]

\draw[dashed, thick] (-3.5, 3) -- (-1.9, 0.2);
\draw[dashed, thick] (-3.5, 3) -- (-2.4, 0.2);
\draw[thick] (-2.55, 0.3) -- (-2.15, 0) -- (-1.75, 0.3);
\end{scope}

\draw(-2.9, 0) node {$x$};
\draw(2.9, 0) node {$y$};
\draw (-3.5, 3.3) node {$w$};
\draw (3.5, 3.3) node {$z$};

\fill[opacity=0.1] (-2.6, 0) circle (1);
\fill[opacity=0.1] (2.6, 0) circle (1);
\draw[<->] (-2.6, -1.1) -- (-1.6, -1.1);
\draw (-2.1, -1.33) node {$K$};

\begin{scope}[xscale=-1]
\draw[<->] (-2.6, -1.1) -- (-1.6, -1.1);
\draw (-2.1, -1.33) node {$K$};
\end{scope}

\end{tikzpicture}
\caption{Alignment of geodesics and points.}
\label{fig:alignment}
\end{figure}
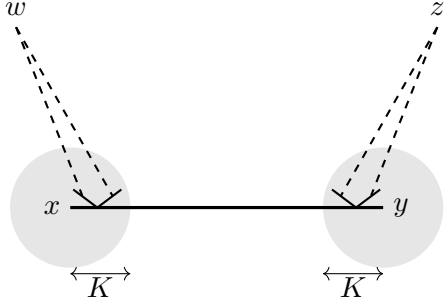

The following is immediate.

\begin{lem}\label{lem:alignDich}
    Let $\ga \subset X$ be a geodesic of length $L \ge 0$, let $0 \le D \le L$ and let $x \in X$. Then $(\gamma, x)$ is not $D$-aligned or $(x, \gamma)$ is not $(L-D)$-aligned.
\end{lem}

If two closed sets are within finite Hausdorff distance, the contracting property of one set implies that of the other (\cite[Lemma 2.8]{arzhantseva2015growth}, \cite[Proposition 2.4.(2)]{yang2019statistically}). Furthermore, every subsegment of a contracting geodesic is contracting with a uniform contracting constant (\cite[Lemma 3.2]{bestvina2009higher}, \cite[Proposition 2.2.(3)]{yang2020genericity}). These facts have the following consequence, whose proof is included for completeness.

\begin{lem}\label{lem:BGIPHeredi}
Let $g \in \Isom(X)$ be a contracting isometry with an axis  $\gamma: \R \rightarrow X$ and let $x_{0} \in X$. Then there exists $C=C(g, \gamma, x_{0})>0$ such that the following holds. 
\begin{enumerate}
    \item[(0)] $\gamma$ is $C$-contracting.
\item  $[x_{0}, g^{k} x_{0}]$ is $C$-contracting for all $k \in \Z$.
\item $d\left(g^{k} x_{0}, \gamma(\tau_{g} k)\right) < C$ for all $k \in \Z$.
\item Let $k \in \N$,  let $x \in X$, and let $K \ge C$. Then 
$$
\left(x, [x_{0}, g^{k} x_{0}]\right) \text{ is not $K$-aligned} \quad \Longrightarrow \quad \pi_{\gamma}(x) \subset \gamma \left( [K-C, +\infty) \right).
$$
\item Let $k\in \N$, let $x \in X$, and let $0 \le K \le \tau_{g} k - C$. Then
$$
\left(x, [x_{0}, g^{k} x_{0}]\right) \text{ is $K$-aligned} \quad \Longrightarrow \quad \pi_{\gamma}(x) \subset \gamma \left( (-\infty,  K + C]\right).
$$
\end{enumerate}
Moreover, $C$ can be chosen so that $C(g, \ga, x_0) = C(g^{k}, \ga, x_0)$ for all $k \in \N$ and $C(g^{-1}, \widehat{\ga}, x_0) = C(g, \ga, x_0)$ where $\widehat{\ga}$ is the inversion of $\ga$.
\end{lem}
We often write $C(g) = C(g, \ga, x_0)$ by implicitly choosing its axis $\ga$.

\begin{proof}

    By definition, there exists $C_0 > 0$ so that $\ga$ is $C_0$-contracting. Item (0) holds for all $C \ge C_0$.
Item (2) is also immediate for any $C > d(x_{0}, \gamma(0))$.

Let
$$
D = 10\Big(C_{0} + d\big(x_0, \gamma(0)\big)\Big).$$
We will see that $C = 100 D$ plays the desired role.

\medskip
We first show Item (1). By Lemma \ref{lem:BGIPFellow}, $[x_{0}, g^{k} x_{0}]$ and $\gamma([0, \tau_{g} k])$ are $D$-equivalent for each $k \in \Z$. It suffices to consider $k \in \N$ with $\tau_g k > 98 D$. Fix such $k \in \N$ and observe from Corollary \ref{cor:BGIPFellow}(2) that for $x \in X$,
\begin{itemize}
    \item if $\pi_\ga(x) \cap \ga((-\infty, 0)) \neq \emptyset$, then $\pi_{[x_0, g^k x_0]}(x) \subset \mathcal{N}_{5D}(\ga(0))$.
    \item if $\pi_{\ga}(x) \cap \ga((\tau_g k, +\infty)) \neq \emptyset$, then $\pi_{[x_0, g^k x_0]}(x) \subset \mathcal{N}_{5D}(\ga(\tau_g k))$.
    \item if $\pi_{\ga}(x) \cap \ga([0, \tau_g k]) \neq \emptyset$, then $\pi_{[x_0, g^k x_0]}(x) \subset \mathcal{N}_{5D}(\pi_{\ga}(x))$.
\end{itemize}

Let $\eta \subset X$ be a geodesic such that $d(\eta, [x_0, g^k x_0]) \ge 100 D$. We have two cases. First, if $\diam \pi_{\ga}(\eta) \ge D$, then there exists a subsegment of $\eta$ that is $4C$-equivalent to $\pi_{\ga}(\eta)$. In this case, $\pi_{\ga}(\eta)$ cannot intersect $\gamma([0, \tau_g k])$; otherwise we have $d\left(\eta, \gamma ([0, \tau_g k])\right) \le 4C$ and $d\left(\eta, [x_0, g^k x_0 ]\right) \le 2D$, a contradiction. Hence, either $\pi_{\ga}(\eta) \subset \ga((-\infty, 0))$ or $\pi_{\ga}(\eta) \subset \ga ((\tau_g k, +\infty))$ holds. By the above observation, we have $\diam \pi_{[x_0, g^k x_0]}(\eta) \le 100D$.

If $\diam \pi_{\ga}(\eta) < D$, then we again have $\diam \pi_{[x_0, g^k x_0]}(\eta) \le 100D$ by the above observation. Therefore, Item (1) holds for $C = 100D$.

\medskip
 We now show Item (3). Let $k \in \N$, let $K \ge 100D$ and suppose that $(x, [x_{0}, g^{k} x_{0}])$ is not $K$-aligned. Recall that $[x_{0}, g^{k} x_{0}]$ and $\gamma([0, \tau_{g} k])$ are $D$-equivalent. Hence there exist $p \in \pi_{[x_0, g^k x_0]}(x)$ and $t_p \in [0, \tau_g k]$ such that
 $$
d(x_0, p) > K \quad \text{and} \quad d\left(p, \ga(t_p)\right) \le D.
 $$
 Note also that $d(x_0, \ga(0)) \le D$. We then have
 $$
t_p \ge K - 2D.
 $$
 If there exists $t \in (-\infty, K - 6D]$ such that $\ga(t) \in \pi_{\ga}(x)$, then 
 $$
 \begin{aligned}
d(x, \ga(t_p)) & \ge d(x, \ga(K - 6D)) + t_p - (K - 6D) - D \\
& \ge d(x, \ga(K - 6D)) +  3D
 \end{aligned}
 $$
 by  Corollary \ref{cor:BGIPFellow}(2). Since $0 \le K - 6D \le t_p \le \tau_g k$, and since $[x_0, g^k x_0]$ and $\ga([0, \tau_g k])$ are $D$-equivalent, we have
 $$
d(x, p) \ge d(x, \ga(t_p)) - D \ge d(x, \ga(K - 6D)) + 2D \ge d(x, [x_0, g^k x_0]) + D,
 $$
 which contradicts to $p \in \pi_{[x_0, g^k x_0]}(x)$. Therefore,
 $$
\pi_{\ga}(x) \subset \ga((K - 6D, + \infty))
 $$
 and hence Item (3) holds for $C=100 D$.

\medskip

For Item (4), let $k \in \N$ and $0 \le K \le \tau_{g} k - 100D$, and suppose that $(x, [x_{0}, g^{k} x_{0}])$ is $K$-aligned. Then $([x_0, g^k x_0], x)$ is not $(d(x_0, g^k x_0) - K)$-aligned by Lemma \ref{lem:alignDich}. Note that \[
d(x_0, g^{k} x_0) - K \ge \tau_{g} k - K \ge 100D.
\]
Hence, we can apply (a symmetric version of) Item (3) and deduce that
$$
\pi_{\gamma}(x)  \subset \ga((-\infty, \tau_g k - (d(x_0, g^k x_0) - K) + 6D)) \subset \ga((-\infty, K + 6D)).
$$
Therefore, Item (4) holds for $C = 100 D$ as well.

\medskip

The ``Moreover'' part is straightforward.
\end{proof}

\subsection{Non-elementary subgroups of isometries}

The class of subgroup we mainly consider is the following:

\begin{definition} \label{def:noneltsubgp}
    A subgroup $\Ga < \Isom(X)$ is called  \emph{non-elementary} if
    \begin{itemize}
        \item $\Ga$ is not virtually cyclic,
        \item the $\Ga$-action on $X$ is proper, and
        \item $\Ga$ contains a contracting isometry.
    \end{itemize}
\end{definition}

We  say that two contracting isometries $g, h \in \Isom(X)$ are \emph{independent} if their orbits $\{g^{i} x_{0}\}_{i \in \Z}$ and $\{h^{i} x_{0}\}_{i \in \Z}$ have infinite Hausdorff distance. This is equivalent to saying that $\{g^{i} x_{0}\}_{i \in \Z}$ and $\{h^{i} x_{0}\}_{i \in \Z}$ have bounded nearest-point projections onto each other.

These two notions are related by the following  well known fact: see \cite[Proposition 6.5]{bestvina2009higher}, \cite[Corollary 4.4]{sisto2018contracting}, \cite[Lemma 2.23]{arzhantseva2015growth} and  \cite[Lemma 2.11, Lemma 2.12]{yang2019statistically}

\begin{lem}\label{lem:indep}
Let $\Gamma < \Isom(X)$ be a non-elementary subgroup. For a contracting isometry $g \in \Ga$, there exists $h \in \Ga$ such that $hgh^{-1}$ and $g$ are independent. Moreover,  there are infinitely many pairwise independent contracting isometries in $\Ga$.
\end{lem}

The following is a variant of the so-called \emph{extension lemma} of Yang. We include the proof of this variant for the sake of completeness. 

\begin{lem}[Extension lemma {\cite[Lemma 1.13]{yang2019statistically}}]\label{lem:extension}
Let $\Gamma < \Isom(X)$ be a non-elementary subgroup. Then for each contracting isometry $\varphi \in \Ga$, there exist $a_{1}, a_{2}, a_{3}\in \Gamma$ and $\kappa = \kappa(\varphi)>0$  such that 
\begin{itemize}
\item $[x_0, \varphi^n x_0]$ is $\kappa$-contracting for all $n \in \N$ and
\item for each $x, y \in X$, there exists $a\in \{a_{1}, a_{2}, a_{3}\}$ that  makes 
$$
(x, a \cdot [x_{0}, \varphi^{n}x_{0}], a \varphi^{n} a \cdot y) \quad \text{$\kappa$-aligned for all } n \in \N.
$$
\end{itemize}
Moreover, $\kappa$ can be chosen so that $\kappa(\varphi^k) = \kappa(\varphi)$ for all $k \in \Z$.
\end{lem}

\begin{proof}
    Note that the ``Moreover'' part is straightforward. Hence we prove the first claim. Let $\gamma :  \R \to X$ be an axis of $\varphi$, which is contracting. By \cite[Corollary 4.4]{sisto2018contracting} and \cite[Lemma 2.11]{yang2019statistically}, the set  
\[
E(\varphi) := \{h \in \Gamma : \textrm{$\{\varphi^i x_0\}_{i \in \Z}$ and $h \{\varphi^i x_0\}_{i \in \Z}$ have  finite Hausdorff distance}\}
\]
is a finite extension of $\langle \varphi \rangle$ and $\pi_{\gamma}( h \gamma)$ has finite diameter for every $h \notin E(\varphi)$. Since $\Ga$ is non-elementary, it is neither virtually cyclic nor a union of two virtually cyclic subgroups. Hence, there exist $h, h' \in \Gamma$ such that $\pi_{u \gamma}(v\gamma)$ has finite diameter for every distinct  pair of elements $u, v \in \{\id, h, h'\}$.

Let $C = C(\varphi, \gamma, x_0)$ be as in Lemma \ref{lem:BGIPHeredi}. We choose $\kappa \ge 100C$ such that 
\begin{itemize}
\item $\diam \{x_0, h x_0, h'x_0\}  < 0.01 \kappa$ and 
\item $\pi_{u \gamma}(v\gamma) \subset u\gamma ([-0.01 \kappa, 0.01 \kappa])$ for distinct $ u, v \in \{\id, h, h'\}$.
\end{itemize}

Now let $x \in X$. Suppose  $(x, [x_0,\varphi^m x_0])$ is not $\kappa$-aligned for some $m \in \N$. 
\begin{claim*}
Then  $(x, h[x_0, \varphi^n x_0])$ is $\kappa$-aligned for all $n \in \N$.
\end{claim*} To show this, suppose to the contrary that $(x, h [x_0, \varphi^n x_0])$ is also not $\kappa$-aligned for some $n \in \N$. Then $\pi_{h[x_0, \varphi^n x_0]}(x)$ is at least $0.9 \kappa$-far from $hx_0$, whereas $\pi_{h[x_0, \varphi^n x_0]}(x_0)$ is $2d(x_0, hx_0)$-close to $hx_0$, i.e., it is $0.02 \kappa$-close to $hx_0$. Lemma \ref{lem:BGIPFellow} tells us that $[x_0, x]$ has a subsegment $[p, q]$ that is contained in the $0.04 \kappa$-neighborhood of $h[x_0, \varphi^n x_0 ]$, and such that $d(p, q) \ge 0.8 \kappa$ and  $d(x_0, p) \le d(x_0, h x_0) + 0.02 \kappa \le 0.05 \kappa$. Similarly, $[x_0, x]$ has a subsegment $[p', q']$ that is contained in the $0.04 \kappa$-neighborhood of $[x_0, \varphi^m x_0]$, and such that $d(x_0, q') \ge 0.8 \kappa$ and $d(x_0, p') \le 0.05 \kappa$.

Then two subsegments $[p, q]$ and $[p', q']$ of $[x_0, x]$ have an overlap of length at least $0.7 \kappa$. Hence, there exist points $P \in [x_0, \varphi^m x_0 ]$ and $Q \in h[x_0, \varphi^n x_0]$ that are $0.65 \kappa$-far from $x_0$ and are $0.08 \kappa$-close to each other. We can take $P' \in \gamma$ that is $0.01 \kappa$-close to $P$ and $Q' \in h \gamma$ that is $0.01 \kappa$-close to $Q$. Then $d(P', Q') \le 0.1 \kappa$ and $\pi_{\gamma}(Q')$ is $0.2 \kappa$-close to $P'$. Since $d(x_0, P') \ge 0.64 \kappa$, $\pi_{\gamma}(h \gamma)$ contains a point that is $0.44 \kappa$-far from $x_{0}$.

On the other hand, $\pi_{\ga}(h \ga) \subset \ga([-0.01 \kappa, 0.01 \kappa])$ and $d(x_0, \ga(0)) \le  0.01\kappa$. This is a contradiction.
Therefore, $(x, h[x_0, \varphi^n x_0])$ is $\kappa$-aligned for all $n \in \N$. 

\medskip

Similarly, we conclude that $(x, h' [x_0, \varphi^n x_0])$ is $\kappa$-aligned for all $n \in \N$. The same argument applies after replacing $id$ with $h$ or $h'$. We conclude that there exist at least two elements $a \in \{\id, h, h'\}$ such that $(x, a[x_0, \varphi^n x_0])$ is $100 \kappa$-aligned for all $n \in \N$. Likewise, for every $y \in X$, $(a^{-1} [\varphi^{-n}x_0, x_0], y)$ is $100 \kappa$-aligned for all $n \in \N$ for at least two elements $a \in \{\id, h, h'\}$. Hence, we can choose $a \in \{ \id, h, h' \}$ that works for both $x$ and $y$.
\end{proof}

\subsection{Horofunctions and cocycles} \label{subsection:horo}

Let $\Lip^{1}(X)$ be the space of $\R$-valued 1-Lipschitz functions on $X$ and let $\Lip^{1}_{x_{0}}(X)$ be its subspace vanishing at the basepoint $x_{0}$, i.e., \[\begin{aligned}
\Lip^{1}(X) &:= \{ f : X \rightarrow \mathbb{R} : \textrm{$f$ is $1$-Lipschitz}\},\\
\Lip^{1}_{x_{0}}(X) &:= \{ f \in \Lip^{1}(X) : f(x_{0}) = 0\},
\end{aligned}
\]
equipped with the compact-open topology. Here, $\Lip^{1}_{x_0}$ is closed in $\Lip^{1}(X)$.

Recall that $X$ is separable as it is given a proper metric. Therefore, $\Lip^{1}_{x_{0}} (X)$ is compact, Hausdorff, and second countable \cite[Proposition 3.1]{maher2018random}. Hence, it is completely metrizable and is  Polish. We identify $\Lip^1(X)$ and $\Lip_{x_0}^1(X) \times \R$ via the homeomorphism
\begin{equation} \label{eqn:lipandlip1}
f \in \Lip^1(X) \mapsto \left( f - f(x_{0}),  f(x_{0})\right).
\end{equation}

The group $\Isom(X)$ naturally acts on $\Lip^{1}(X)$ by $g \cdot f := f \circ g^{-1}$ for $g \in \Isom(X)$ and $f \in \Lip^1(X)$.  However, this action does not leave $\Lip_{x_{0}}^{1} (X)$ invariant.

Due to this subtlety, we identify $\Lip_{x_0}^1(X)$ with the space of $\R$-valued 1-Lipschitz cocycles on $X$, i.e., $c : X \times X \to \R$ such that $\abs{c(x, y)} \le d(x, y)$ and $c(x, z) = c(x, y) + c(y, z)$ for all $x, y, z \in X$.
For each $f \in \Lip^1(X)$, we define the associated cocycle
$
\beta_f : X \times X \to \R
$
by
$$
\beta_f(x, y) = f(x) - f(y).
$$
Its restriction to $\Lip_{x_0}^1(X)$ gives the homeomorphism between $\Lip_{x_0}^1(X)$ and the space of all $\R$-valued continuous cocycles. Then the identifiaction $\Lip^1(X) \simeq \Lip_{x_0}^1(X) \times \R$ in Equation \eqref{eqn:lipandlip1} can be rephrased as
$$
f \mapsto (\beta_f, f(x_0)).
$$
The $\Isom(X)$-action on $\Lip^1(X)$ is now given as follows: for $g \in \Isom(X)$ and $f \in \Lip^1(X)$,
$$
g \cdot (\beta_f, f(x_0)) = (\beta_{g \cdot f}, f(x_0) + \beta_f (g^{-1}x_0, x_0)).
$$
Note that on the first component, which corresponds to $\Lip^1_{x_0}(X)$, we have $\beta_f \mapsto \beta_{g \cdot f}$.

There is a natural embedding $\iota : X \hookrightarrow \Lip_{x_0}^{1}(X)$, defined by \[
\iota : z \in X \quad \mapsto \quad  \left[ f_{z}( \cdot) := d(\cdot, z) - d(x_{0}, z) \right].
\] 
The closure of $\iota (X) \subset \Lip_{x_0}^{1}(X)$ is called the \emph{horofunction compactification} of $X$ and is denoted by $\overline{X}^{h}$. The complement $\overline{X}^{h} \smallsetminus \iota(X)$ is called the \emph{horofunction boundary} (or \emph{horoboundary}) of $X$  and is denoted by $\partial^{h} X$.

As the space $\Lip_{x_0}^1(X)$ is identified with the space of 1-Lipschitz cocycles, we also regard elements of $\partial^{h} X$ as \emph{Busemann cocycles}, by identifying
$$f \in \partial^{h}X \quad \longleftrightarrow \quad \beta_{f}(\cdot, \cdot).$$ 
Using this identification, the $\Isom(X)$-action on $\partial^{h}X$ is given by
$$
\beta_{f}(\cdot, \cdot) \quad \mapsto \quad \beta_{g \cdot f} (\cdot, \cdot) \quad \text{for }g \in \Isom(X) \text{ and } f \in \partial^{h} X.
$$
This is the continuous extension of the isometric action of $\Isom(X)$ on $X$ in the following sense. Let $\{z_i\}_{i \in \N} \subset X$ be a sequence such that $f_{z_i} \to f \in \Lip^{1}_{x_0}(X)$. Then for every  $g \in \Isom(X)$, we have 
$$ \begin{aligned}
d(\cdot, g z_i) - d(x_0, g z_i) & = d( g^{-1}(\cdot), z_i) - d(x_0, z_i) + d(x_0, z_i) - d( g^{-1} x_0, z_i) \\
& \to f( g^{-1} (\cdot)) - f(g^{-1} x_0).
\end{aligned}
$$
This implies
$$
\beta_{f_{z_i}}(\cdot, \cdot) \to \beta_{g \cdot f} (\cdot, \cdot).
$$

In terms of the identification $\Lip^{1}(X) \simeq \Lip_{x_{0}}^{1}(X) \times \mathbb{R}$, the subspace of $\Lip^{1}(X)$ corresponding to $\partial^{h} X$ is the space
\begin{equation} \label{eqn:horofol}
\mathcal{H} := \partial^{h} X \times \R,
\end{equation}
which is  $\Isom(X)$-invariant. As a subspace of $\Lip^1(X)$, $\mathcal{H}$ does not depend on the choice of the basepoint $x_0 \in X$. What depends on the choice of $x_0$ is the description of $\Isom(X)$-action on $\mathcal{H}$ in terms of the identification $\mathcal{H} = \partial^{h} X \times \R$.

We call elements of $\partial^{h} X \times \mathbb{R}$ \emph{horofunctions}. Horofunctions that differ by an additive constant correspond to the same Busemann cocycle.

Both $\partial^{h}X$ and $\mathcal{H} = \partial^{h} X \times \mathbb{R}$ are Polish. Hence, every locally finite Borel measure on these spaces is Radon, i.e., it is both inner and outer regular on Borel subsets.

We now extend the notion of alignment to horofunctions.

\begin{definition}\label{dfn:alignHoro}

    Let $\xi \in \partial^{h} X$ and $\ga \subset X$ be a compact geodesic. For $K \ge 0$, we say that $(\xi, \ga)$ is \emph{$K$-aligned} if for every sequence $\{z_{i}\}_{i \in \N} \subset X$ converging to $\xi$, $(z_{i}, \gamma)$ is $K$-aligned eventually (i.e., for all large $i \in \N$). We define the alignment for $(\ga, \xi')$ and $(\xi, \ga, \xi')$ similarly for $\xi' \in \overline{X}^{h}$.
    
    By abuse of notation, for a Busemann cocycle $\beta_{\xi}$ that corresponds to $\xi \in \partial^{h} X$, we say that $(\beta_{\xi}, \gamma)$ is $K$-aligned when $(\xi, \gamma)$ is $K$-aligned. Lastly, for an element $f \in \mathcal{H}$, 
     we say that $(f, \gamma)$ is $K$-aligned if $(\beta_{f}, \gamma)$ is $K$-aligned. We define the alignment for triples similarly. 
\end{definition}

For contracting geodesics, we observe the following:

\begin{lemma} \label{lem:eventualdiam}
Let $\ga \subset X$ be a $C$-contracting compact geodesic for $C > 0$. Let $\xi \in \partial^{h} X$ and $\{z_i \}_{i \in \N} \subset X$ be a sequence converging to $\xi$. Then we have
$$
\limsup_{N \rightarrow +\infty} \,\,\diam \bigcup_{k \ge N} \pi_{\gamma}(z_{k}) \le 9C.
$$
\end{lemma}

\begin{proof}

Since $\gamma$ is compact and $z_n \to \xi$, there exists $N>0$ such that
$$d(x, z_{k}) - d(y, z_{k})=_{0.1C} d(x, z_{l}) - d(y, z_{l}) \quad \text{ for all $k, l > N$ and $x, y \in \gamma$.}
$$ Now suppose to the contrary that $\diam (\pi_{\gamma}(z_{k}) \cup\pi_{\gamma}(z_{l})) >9C$ for some $k, l > N$. Then by Lemma \ref{lem:BGIPFellow}, there exist $p, q \in [z_{k}, z_{l}]$ such that 
$$\diam (\{p\} \cup \pi_{\gamma}(z_{k})) \le 2C \quad \text{and} \quad \diam (\{q\} \cup \pi_{\gamma}(z_{l})) \le 2C.$$
 Note that $d(p, q) > 5C$. Let $p' \in \pi_{\gamma}(z_{k})$ and $q' \in \pi_{\gamma}(z_{l})$ be arbitrary points. We then have 
\[
d(z_{k}, p') \le d(z_{k}, p) + 2C \le \left(d(z_{k}, q) - 5C \right) + 2C \le d(z_{k}, q') - C.
\]
For a similar reason, we have $d(z_{l}, q') \le d(z_{l}, p') - C$. This contradicts the condition for $k, l> N$.
\end{proof}

The following version of extension lemma can easily be deduced from Lemma \ref{lem:extension} and Lemma \ref{lem:eventualdiam}:

\begin{lem}[Extension lemma] \label{lem:extensionHoro}
    Let $\Ga < \Isom(X)$ be a non-elementary subgroup. Let $\varphi \in \Ga$ be a contracting isometry, and let $\kappa = \kappa(\varphi) >0$ and $a_{1}, a_{2}, a_{3} \in \Gamma$ be as given in Lemma \ref{lem:extension}.  Then for each $\xi, \xi' \in \overline{X}^{h}$, there exists $a \in \{a_1, a_2, a_3 \}$ such that 
    $$
(\xi, a \cdot [x_{0}, \varphi^{n}x_{0}], a \varphi^{n} a \cdot \xi') \quad \text{is $10\kappa$-aligned for all } n \in \N.
$$
\end{lem}

\begin{rmk}
There is a way to extend the the nearest-point projection $\pi_{\gamma}(\cdot)$ to  $\overline{X}^{h}$ (cf. \cite[Definition 3.8]{coulon2024patterson-sullivan}). For $\xi \in \partial^{h} X$, we can define \[
\pi_{\gamma}(h) := \left\{x \in \gamma : \beta_{\xi}(x, x_0) = \inf_{y \in \gamma} \beta_{\xi}(y, x_0) \right\}.
\]
Then one can show that if $\{z_i\}_{i \in \N} \subset X$ is a sequence converging to $\xi \in \partial^{h} X$, then $\pi_{\gamma}(z_{i}) \rightarrow \pi_{\gamma}(\xi)$ up to a uniform error. One can define alignment in terms of this extended nearest-point projection as well.
\end{rmk}

\subsection{Conical limit sets}

We finish this section by defining conical limit sets, which are also called radial limit sets.

\begin{definition} \label{def:conical}
    Let $\Ga < \Isom(X)$ be a subgroup acting properly on $X$. A point $\xi \in \partial^{h} X$ is called a \emph{conical limit point} of $\Ga$ if there exist $K > 0$ and an infinite sequence $ \{g_n \}_{n \in \N} \subset \Ga $ such that
    $$
    \beta_{\xi}(x_0, g_n x_0) \ge d(x_0, g_n x_0) - K  \quad \text{for all } n \in \N.
    $$
    We denote the \emph{conical limit set} by $\La_c(\Ga) \subset \partial^{h} X$.
\end{definition}

For example, given a geodesic ray $\gamma \subset X$ whose $K$-neighborhood contains infinitely many poitns in a $\Gamma$-orbit, the horofunction made as a limit point of $\gamma$ is conical. One can see that $\La_c(\Ga)$ is $\Ga$-invariant.

\section{Patterson--Sullivan theory with squeezing isometries}\label{section:ps}

In this section, we review Coulon's and Yang's extensions of the Patterson--Sullivan theory about conformal densities (\cite{coulon2024patterson-sullivan}, \cite{yang2024conformal}). Continuing the setting of Section \ref{section:prelim}, we let $(X, d)$ be a proper geodesic metric space. Given a non-elementary subgroup $\Ga < \Isom(X)$ with a squeezing isometry, we focus on a certain subset of the conical limit set of $\Ga$ and study its properties from the viewpoint of theories of Coulon and Yang.

In the rest of this section, we fix a basepoint $x_0 \in X$. The notion of conforaml density plays an important role.

\begin{definition}
    For $\Ga < \Isom(X)$ and $\delta \ge 0$, a family of Borel measures $\{ \nu_x \}_{x \in X}$ on $\partial^{h} X$ is called a \emph{$\delta$-dimensional conformal density} of $\Ga$ if 
    \begin{itemize}
\item {\rm ($\Gamma$-invariance)} for every $g \in \Ga$ and $x \in X$,
$$g_{*} \nu_{x} = \nu_{gx},$$
\item {\rm ($\delta$-conformality)} for every $x, y \in X$, two measures $\nu_x$ and $\nu_y$ are in the same class and \[
\frac{d\nu_{x}}{d\nu_{y}}(\xi) = e^{-\delta \cdot \beta_{\xi}(x, y)} \quad \textrm{a.e., and}
\]
\item {\rm (normalization)} $\nu_{x_0}(\partial^{h}X) = 1$.
\end{itemize}
\end{definition}

Note that the normalization is only for convenience, in order to have uniqueness of a conformal density in a certain case.

 In our setting, Coulon and Yang constructed conformal densities, extending the construction of  Patterson \cite{patterson1976the-limit} and Sullivan \cite{Sullivan1979density} for the case that $X$ is a real hyperbolic space. 
For a subgroup $\Ga < \Isom(X)$ that acts properly on $X$, the \emph{critical exponent} $\delta_{\Ga} \ge 0$ of $\Ga$ is defined as the abscissa of convergence of the Poincar\'e desires $s \mapsto \sum_{g \in \Ga} e^{-s d(x, g x)}$, $x \in X$.

\begin{prop}[{\cite[Proposition 4.3]{coulon2024patterson-sullivan}}, {\cite[Lemma 6.3]{yang2024conformal}}]\label{prop:patterson}
    Let $\Ga < \Isom(X)$ be a non-elementary subgroup such that $\delta_{\Ga} < + \infty$. Then there exists a $\delta_{\Ga}$-dimensional conformal density of $\Ga$.
\end{prop}

\subsection{Guided limit sets}

The following notion is a variant of Coulon's contracting limit sets \cite{coulon2024patterson-sullivan} and Yang's $(L, \mathcal{F})$-limit sets \cite{yang2024conformal}.

\begin{definition}
    Let $\Ga < \Isom(X)$ be a non-elementary subgroup. Let $\varphi \in \Ga$ be  a contracting isometry and let $C(\varphi) > 0$ be as in Lemma \ref{lem:BGIPHeredi} and fix $K\ge C(\varphi)$. We say that $\xi \in \partial^{h} X$ is a \emph{$(\varphi, K)$-guided limit point} of $\Ga$ 
if for each sufficiently large $n \in \N$, there exists $h \in \Gamma$ such that 
$$
(x_{0}, h[x_{0}, \varphi^{n} x_{0}], \xi) \quad \text{is $K$-aligned.}
$$
The collection of $(\varphi, K)$-guided limit points of $\Ga$ called the \emph{$(\varphi, K)$-guided limit set} of $\Gamma$. We denote it by $\La_{\varphi, K}(\Ga)$.

\end{definition}

Later, $\varphi \in \Ga$ is often assumed to be squeezing. In such cases, we give  special names to $(\varphi,K)$-guided limit points and the  $(\varphi,K)$-guided  limit set,  $(\varphi,K)$-squeezed limit points and the $(\varphi,K)$-squeezed limit set, respectively. 

We first discuss some properties of guided limit sets.

\begin{lem}\label{lem:squeezedInv}
    Let $\Ga < \Isom(X)$ be a non-elementary subgroup. Let $\varphi \in \Ga$ be  a contracting isometry and let $C = C(\varphi) > 0$ be as in Lemma \ref{lem:BGIPHeredi}.  Then for each $K>C$, 
    $$
    \La_{\varphi, K}(\Ga) = \La_{\varphi, C}(\Ga).
    $$
     Moreover, $\La_{\varphi, C}(\Ga)$ is $\Gamma$-invariant.
\end{lem}

\begin{proof}
Let $\gamma : \R \rightarrow X$ be an axis of $\varphi$ chosen for the constant $C = C(\varphi)$ in Lemma \ref{lem:BGIPHeredi}. Fix $K > C$. We then set \[
N = \lceil (K + 100C)/\tau_{\varphi} \rceil.
\]
Now pick an arbitrary $\xi \in \La_{\varphi, K}(\Ga)$ and let $\{ z_i \}_{i \in \N} \subset X$ be a sequence converging to $\xi$. Since $\xi$ is $(\varphi, K)$-guided, for each large enough $n \in \N$ there exists $h \in \Gamma$ such that $(x_0, h[x_{0}, \varphi^{n+2N} x_{0}], z_{i})$ is $K$-aligned for all large $i \in \N$
Since $d(x_0, \varphi^{n + 2N} x_0) \ge (n+2N) \tau_\varphi > K +C$, Lemma \ref{lem:BGIPHeredi}(4) tells us that \[\begin{aligned}
\pi_{h \gamma}(x_{0}) 
&\subset h \gamma \big( (-\infty, K+C] \big) \quad \text{and} \\
\pi_{h \gamma}(z_{i}) &\subset h \gamma \Big(\big[ (n+2N)\tau_\varphi - K - C, +\infty \big)\Big) \quad \textrm{for all large $i \in \N$}.
\end{aligned}
\]

We now show that
$$
(x_0, h \varphi^N [ x_0, \varphi^n x_0], \xi) \quad \text{is } C\text{-aligned.}
$$
Suppose to the contrary that $(x_{0}, h \varphi^{N} [x_{0}, \varphi^{n} x_{0}])$ is not $C$-aligned. Then by Lemma \ref{lem:BGIPHeredi}(3), we have $\pi_{h \gamma}(x_{0}) \subset h \gamma \left( [\tau_\varphi N , +\infty) \right)$. Since $K+C < \tau_\varphi N$, this is a contradiction. Therefore, $(x_{0}, h \varphi^{N} [x_{0}, \varphi^{n} x_{0}])$ is $C$-aligned. 

Similarly, using $ (n+2N)\tau_{\varphi} - K - C > (n  + N) \tau_{\varphi}$, we deduce that
$\left( h\varphi^{N} [x_{0}, \varphi^n x_0], z_{i} \right)$ is $C$-aligned for all large $i \in \N$.
Since this is the case for arbitrary sequence $\{z_{i}\}_{i\in \N} \subset X$ convering to $\xi$, $(x_{0}, h\varphi^N [x_{0}, \varphi^{n} x_{0}] , \xi)$ is $C$-aligned. We conclude $\xi \in \La_{\varphi, C}(\Ga)$, proving the first statement.

\medskip
We now show that $\La_{\varphi, C}(\Ga)$ is  $\Ga$-invariant. Fix  $\xi \in \La_{\varphi, C}(\Ga)$  and $g \in \Ga$. Then for each sufficiently large $n \in \N$, there exists $h \in \Ga$ such that $(x_0, h [x_0, \varphi^n x_0], \xi)$ is $C$-aligned. It is clear that $(g x_0, g h [x_0, \varphi^n x_0], g \xi)$ is $C$-aligned. By Corollary \ref{cor:BGIPFellow}(1), this implies that $(x_0, g h [x_0, \varphi^n x_0], g \xi)$ is $(5C + d(x_0, g x_0))$-aligned. Hence, $g \xi$ is a $(\varphi, 5C + d(x_0, g x_0))$-guided limit point of $\Ga$, and therefore $g \xi \in \La_{\varphi, C}(\Ga)$ by the first statement. This shows the desired $\Ga$-invariance.
\end{proof}

\begin{definition}
    We say that a non-elementary subgroup $\Ga < \Isom(X)$ is of \emph{divergence type} if $\delta_{\Ga} < + \infty$ and its Poincar\'e series diverges at $\delta_{\Ga}$, i.e., we have $\sum_{g \in \Ga} e^{-\delta_{\Ga} d(x_0 , g x_0 )} = + \infty$. 
\end{definition}

As a part of their generalizations of Hopf--Tsuji--Sullivan dichotomy, Coulon and Yang proved the following:

\begin{prop}[{\cite[Theorem 1.14]{yang2024conformal}, \cite[Theorem 1.4]{coulon2024patterson-sullivan}}]\label{prop:pattersonSqueeze}

Let $\Ga < \Isom(X)$ be a non-elementary subgroup of divergence type. Let $\varphi \in \Ga$ be a  contracting isometry and let $C = C(\varphi) > 0$ be as in Lemma \ref{lem:BGIPHeredi}. Then for every $\delta_{\Ga}$-dimensional conformal density $\{ \nu_x \}_{x \in X}$  of $\Ga$ and for every $x \in X$, $\La_{\varphi, C}(\Ga)$ is  $\nu_x$-conull.
\end{prop}

\subsection{Squeezed limit sets and ergodic properties}

As mentioned above, we call the guided limit sets for squeezing isometries the squeezed limit sets.

\begin{definition}\label{lem:squeezingHoro}
    Let $\Ga < \Isom(X)$ be a non-elementary subgroup containing a squeezing isometry $\varphi \in \Ga$. Let $C(\varphi) > 0$ be as in Lemma \ref{lem:BGIPHeredi} and fix $K\ge C(\varphi)$. We say that $\xi \in \partial^{h} X$ is a \emph{$(\varphi, K)$-squeezed limit point} of $\Ga$ 
if for each sufficiently large $n \in \N$, there exists $h \in \Gamma$ such that 
$$
(x_{0}, h[x_{0}, \varphi^{n} x_{0}], \xi) \quad \text{is $K$-aligned.}
$$
The collection of $(\varphi, K)$-squeezed limit points of $\Ga$ called the \emph{$(\varphi, K)$-squeezed limit set} of $\Gamma$. We use the same notation $\La_{\varphi, K}(\Ga)$ for this.
\end{definition}

This notion of squeezed limit sets play a key role in this paper.  Importantly, the following observation leads us to have ergodicity of conformal densities on $\partial^{h} X$, as we will see. In the following, we regard points in $\partial^{h}X$ as horofunctions $X \to \R$ vanishing at $x_0$.

\begin{lem}\label{lem:squeezingHoroVal}
Let $\Ga < \Isom(X)$ be a non-elementary subgroup containing a squeezing isometry $\varphi \in \Ga$. Let $ C = C(\varphi) > 0$ be as in Lemma \ref{lem:BGIPHeredi}.
For $\xi, \zeta \in \partial^{h}X$, if $\xi \in \La_{\varphi, C}(\Ga)$ and $\norm{\xi - \zeta}_{\infty} < + \infty$, then  $\xi = \zeta$.
\end{lem}

\begin{proof}
Let   $\gamma : \R \rightarrow X$ be the unique axis of $\varphi$. 
Let $\xi, \zeta \in \partial^{h} X$ given as in the statement, regarding them as horofunctions vanishing at $x_{0}$. Let $B := \|\xi - \zeta\|_{\infty} < +\infty$. This implies that for every $x \in X$ we have \begin{equation}\label{eqn:xiZetaB}
\xi(x) =_{B} \zeta(x).
\end{equation}

We fix an arbitrary $x \in X$ and an arbitrary $\epsilon >0$. Our goal is to prove that $\xi(x)=_{\epsilon}\zeta(x)$. Let $L= L(0.01\epsilon) > 0$ be as in Definition \ref{dfn:squeezing} for $\ga$, and let $n \in \N$ be a sufficiently large integer such that \[
n > \frac{100(L + C + B+ d(x_0, x))}{\tau_{\varphi}}.
\]

Now, let $\{z_i\}_{i \in \N} \subset X$ be a sequence converging to $\xi$. Since $\xi \in \La_{\varphi, C}(\Ga)$,  there exists $h \in \Gamma$ such that $(x_0, h[x_0, \varphi^{n} x_0], z_{i})$ is  $C$-aligned for all large $i \in \N$.
Since $n \tau_{\varphi} - C > C$, Lemma \ref{lem:BGIPHeredi}(4) tells us that \begin{align}
\begin{aligned}
\pi_{h \gamma} (x_0) &\subset h \gamma \left( (-\infty, 2C] \right) \quad \text{and} \\ \label{eqn:ziApprox}
\pi_{h \gamma} (z_i) &\subset h \gamma \left( [n\tau_\varphi - 2C, +\infty) \right) \quad \textrm{for all large $i \in \N$}.
\end{aligned}
\end{align}
Since $\pi_{\gamma}(\cdot)$ is $(1, 4C)$-Lipschitz by Corollary \ref{cor:BGIPFellow}(1), we have 
$$\pi_{h \gamma}(x) \subset h \gamma \left( (-\infty, 6C + d(x_0, x) ] \right).$$
By Corollary \ref{cor:BGIPFellow}(2), Equation \eqref{eqn:ziApprox} implies that \begin{equation}\label{eqn:ziResult}
d\left(z_i, h\gamma(n\tau_{\varphi} - 2C)\right) - 
d\left(z_i, h\gamma(n\tau_{\varphi} - 19C-2B)\right) =_{8C} -(2B + 17C).
\end{equation}

Now let $\{z'_i\}_{i \in \N} \subset X$ be a sequence converging to $\zeta$. 
\begin{claim*} We have 
    \begin{equation}\label{eqn:ziPrimeCorrectApp}
\pi_{h \gamma}(z'_i) \subset h\gamma \left( [ n\tau_{\varphi} - 20C - 2B, +\infty) \right) \quad \textrm{for all large $i \in \N$}.
\end{equation}
\end{claim*}
Suppose to the contrary that, passing to a subsequence, we have 
\begin{equation}\label{eqn:ziPrimeApprox}
\pi_{h \gamma}(z'_i) \subset h\gamma \left( (-\infty, n\tau_{\varphi} - 18C-2B] \right) \quad \text{for all } i \in \N.
\end{equation}
Here, the fact that $\pi_{h\gamma}(z_i')$ is $2C$-small is used. By Corollary \ref{cor:BGIPFellow}(2), Equation \eqref{eqn:ziPrimeApprox} implies that \[
d\left(z'_i, h\gamma(n\tau_{\varphi} - 2C)\right) - 
d\left(z'_i, h\gamma(n\tau_{\varphi} - 18C-2B)\right) =_{8C} 2B + 16C.
\]
Comparing this with Equation \eqref{eqn:ziResult}, we observe that $\xi(a) - \xi(b)$ and $\zeta(a) - \zeta(b)$ differ by more than $4B$ for two points $a := h\gamma(n\tau_{\varphi}-2C)$ and $b :=h\gamma(n\tau_{\varphi} - 19C - 2B)$. This contradicts Equation \eqref{eqn:xiZetaB}. The claim follows.

\medskip
Hence, Equation \eqref{eqn:ziPrimeCorrectApp} holds. Therefore, for each large enough $i$ we have
$$\begin{aligned}
\pi_{h\gamma}(x_{0}), \pi_{h \gamma}(x) & \subset h\gamma \left( (-\infty, 6C + d(x_{0}, x)]\right) \quad \text{and} \\
\pi_{h\gamma}(z_{i}), \pi_{h\gamma}(z'_{i}) & \subset h\gamma \left( [n\tau_{\varphi} - 20C-2B, +\infty) \right).
\end{aligned}
$$
Since $n\tau_{\varphi} - 20C - 2B \ge 6C + d(x_0, x) + 2L$, we conclude from Lemma \ref{lem:squeezing} that $d(x_{0}, z_{i}) - d(x, z_{i}) =_{\epsilon} d(x_{0}, z_{i}') - d(x, z_{i}')$. Since both $\xi$ and $\zeta$ are horofunctions vanishing at $x_0$, taking the limit $i \to + \infty$ yields $\xi(x) =_{\epsilon} \zeta(x)$.
\end{proof}

To discuss ergodicity of conformal densities, Coulon \cite{coulon2024patterson-sullivan} considered conformal densities restricted to the \emph{reduced horoboundary} instead of the usual horoboundary. Yang \cite{yang2024conformal} considered another notion encompassing the reduced horoboundary, namely, the \emph{reduced convergence boundary}. 

Instead of defining these objects precisely, let us point out that Coulon's and Yang's ergodicity results are formulated in terms of saturated Borel subsets of $\partial^{h} X$, that is, $E \subset \partial^{h} X$ with a property that if $\xi \in E$ and $\zeta \in \partial^{h} X$ satisfy $\norm{\xi - \zeta}_{\infty} < + \infty$, then $\zeta \in E$. Hence, by Proposition \ref{prop:pattersonSqueeze} and Lemma \ref{lem:squeezingHoroVal}, their ergodicity results can be stated as follows:

\begin{prop}[{\cite[Theorem 1.16]{yang2024conformal}, \cite[Theorem 1.5]{coulon2024patterson-sullivan}}]\label{prop:pattersonSqueeze2}

Let $\Ga < \Isom(X)$ be a non-elementary subgroup containing a squeezing isometry, and let $\{ \nu_x \}_{x \in X}$ be a $\delta_{\Ga}$-dimensional confomal density of $\Ga$. If $\Ga$ is of divergence type, then the $\Ga$-action on $(\partial^{h}X, \nu_x)$ is ergodic  for all $x \in X$.
\end{prop}

The above ergodicity indeed implies the uniqueness of $\delta_{\Ga}$-dimensional conformal density when $\Ga$ is as in Proposition \ref{prop:pattersonSqueeze2}. Indeed, if $\{\nu_x \}_{x \in X}$ and $\{\nu_x'\}_{x \in X}$ are $\delta_{\Ga}$-dimensional conformal densities of $\Ga$, then $\{ (\nu_x + \nu_x')/2 \}_{x \in X}$ is also a $\delta_{\Ga}$-dimensional conformal density of $\Ga$. Since $\nu_x$ is absolutely continuous with respect to $(\nu_x + \nu_x')/2$ for each $x \in X$, the ergodicity applies to their Radon--Nikodym derivatives which are $\Ga$-invariant due to the conformality. Therefore, $\nu_x = \nu_x'$ for all $x \in X$.

\section{Rigidity of ergodic invariant Radon measures} \label{sec:UE}

Let $(X, d)$ be a proper geodesic metric space.
In this section, we prove the rigidity of ergodic invariant Radon measures (Theorem \ref{thm:main2Conv}).

\subsection{Candidates for measures} 

We first define a Radon measure which will  be the unique measure with desired properties in our rigidity theorem. To do this, we fix a basepoint $x_0 \in X$ and identify $\mathcal{H} = \partial^{h} X \times \R$ as in Equation \eqref{eqn:horofol}. Recall that  the $\Isom(X)$-action on $\mathcal{H}$ is written as follows: for $g \in \Isom(X)$ and $(\xi, t) \in~\mathcal{H}$, we have
$$
g (\xi, t) = (g \xi, t + \beta_{\xi}(g^{-1} x_0, x_0)).
$$
Via this identification, we define a Radon measure on $\mathcal{H}$ as follows:

\begin{definition} \label{def:candidateergodicmeasure}
    Let $\Ga < \Isom(X)$ be a non-elementary subgroup containing a squeezing isometry and $\nu := \{ \nu_x \}_{x \in X}$ be a $\delta_{\Ga}$-dimensional conformal density of $\Ga$. We define a Radon measure $\mu_{\nu}$ on $\mathcal{H} = \partial^{h}X \times \R$ by setting
    $$
    d\mu_{\nu}(\xi, t) := e^{\delta_{\Ga} \cdot t} \cdot d \nu_{x_0} (\xi) \, dt.
    $$
    When $\Ga$ is of divergence type in addition, we write
    $$
    \mu_{\Ga} := \mu_{\nu}.
    $$
\end{definition}

\begin{remark}
It follows from the conformality of $\nu$ that $\mu_{\nu}$ is $\Ga$-invariant. Moreover, considering $\mathcal{H}$ as a subspace of $\Lip^1(X)$, the measure $\mu_{\Ga}$ does not depend on the choice of $x_0 \in X$. When $\Ga$ is of divergence type, there exists a unique $\delta_{\Ga}$-dimensional conformal density $\{ \nu_{x} \}_{x \in X}$ of $\Ga$ by Proposition \ref{prop:patterson} and Proposition \ref{prop:pattersonSqueeze2}. This is a reason for writing $\mu_{\Ga} = \mu_{\nu}$ in this case.
\end{remark}

To present a precise statement, we also consider the following notion for the distribution of translation lengths of squeezing isometries.

\begin{definition}
    For a subgroup $\Ga < \Isom(X)$, its \emph{squeezing spectrum} is     $$
    \Spec(\Ga) := \{ \tau_g \in \R : g \in \Ga \text{ is a squeezing isometry} \}.
    $$
    We call that the squeezing spectrum is \emph{non-arithmetic} if it generates a dense additive subgroup of $\R$.
\end{definition}

\subsection{Rigidity of measures}
Let us restate Theorem \ref{thm:main2Conv}, our main rigidity theorem. Recall the notion of conical limit set from Definition \ref{def:conical}.

\begin{theorem} \label{thm:uniqueRadon}
Let $\Ga < \Isom(X)$ be a non-elementary subgroup with non-arithmetic squeezing spectrum. 
Suppose that there exists a $\Ga$-invariant Radon measure $\mu$ on $\mathcal{H}$ supported on $\La_c(\Ga) \times \R$. Then $\Ga$ is of divergence type, and
    $$
    \mu \text{ is a constant multiple of $\mu_{\Ga}$.}
    $$

\end{theorem}

The rest of this section is devoted to the proof of Theorem \ref{thm:uniqueRadon}. 
We prove the theorem by establishing a robust relation between invariant Radon measures and squeezed limit sets. Note that due to ergodic decompositions, it suffices to consider invariant \emph{ergodic} measures.

\subsection{Concentration on squeezed limit sets}
We first show that invariant ergodic Radon measures on $\mathcal{H}$ are charged on squeezed limit sets. We emphasize that in the following, we consider an arbitrary Radon measure on $\mathcal{H}$, not necessarily induced from a conformal density.

\begin{theorem}\label{thm:radonCharge}

Let $\Ga < \Isom(X)$ be a non-elementary subgroup, let $\varphi$ be a contracting isometry in $\Gamma$ and let $C = C(\varphi)$ be as in Lemma \ref{lem:BGIPHeredi}.
Let $\mu$ be a $\Ga$-invariant ergodic Radon measure on $\mathcal{H}$ supported on $\La_{c}(\Ga) \times~\R$. Then the measure $\mu$ is supported on
$$
 \La_{\varphi, C}(\Ga) \times \R \subset \mathcal{H}.
$$

\end{theorem}

\begin{proof}
Let $\kappa(\varphi) > 0$ and $a_{1}, a_{2}, a_{3} \in \Gamma$  be as in Lemma \ref{lem:extension} for $\Gamma$. Let $ C(\varphi) > 0$ be as in Lemma \ref{lem:BGIPHeredi} for $g=\varphi$. We set $C_0 := 10(\kappa(\varphi) +C(\varphi))$. 

For each $K>0$ let 
\[
\mathcal{H}_{K} :=  \left\{ (\xi, t) \in \mathcal{H} : \begin{matrix}
\exists \text{ an infinite sequence } \{g_i\}_{i \in \N} \subset \Ga \text{ s.t.}\\
\beta_{\xi}(x_{0}, g_{i} x_{0}) \ge d(x_{0}, g_{i}x_{0}) - K \text{ for all } i \in \N
\end{matrix}  \right\}.
\]
Then $\Gamma \cdot \mathcal{H}_{K} \subset \mathcal{H}$ is  $\Gamma$-invariant. Moreover,
 \[
\Lambda_{c} (\Gamma) \times \mathbb{R} = \bigcup_{K > 0} \Gamma \cdot \mathcal{H}_{K}.
\]
Since $\Lambda_c \Gamma \times \R$ has positive $\mu$-value,
$$\Gamma \cdot \mathcal{H}_{K} \quad \text{has positive $\mu$-value for all large $K > 0$.}
$$
We fix such $K > 100 C_{0} + 2 \sum_{i = 1}^3 d(x_0, a_i x_0)$. Then it follows from the $\Ga$-invariance of $\mu$ that $\mu(\mathcal{H}_{K}) > 0$. 
For each $R > 0$, we set $\mathcal{H}_{K, R} := \{ (\xi, t) \in \mathcal{H}_{K} : -R \le t \le R\}$. Since $\mathcal{H}_{K} = \cup_{R=1}^{\infty} \mathcal{H}_{K, R}$, $$
\mu(\mathcal{H}_{K, R}) > 0 \quad \text{for all large } R > 0.
$$
We fix such $R>0$. 

Now we pick $n>100(C_{0}+K)/\tau_{\varphi}$ and $k > 0$. We define a map
$$
F = F_{n, k}  : \mathcal{H}_{K, R} \to \mathcal{H}
$$
as follows. For each $\Xi \in \mathcal{H}_{K, R}$, there exists $g \in \Gamma$ such that 
\begin{equation} \label{eqn:defofgXi}
d(x_0, gx_0) > k \quad \text{and} \quad \beta_{\Xi}(x_0, g x_0) \ge d(x_0, gx_0) - K.
\end{equation} Among many such $g$'s, take the one with minimal  $d(x_0, gx_0)$ and call it $g_{\Xi}$.\footnote{There exists a technicality when several candidates tie. An easy rescue is to first enumerate $\Gamma = \{g^{(1)}, g^{(2)}, \ldots\}$, and we choose the earliest whenever there is a tie.} Then the map $\Xi \in \mathcal{H}_{K, R} \mapsto g_{\Xi}$ is Borel measurable.
By Lemma \ref{lem:extensionHoro}, there exists $a_{\Xi} \in \{a_{1}, a_{2}, a_{3}\}$ such that \footnote{Again, when more than one of $\{a_1, a_2, a_3\}$ do the job we choose the earliest.}
\begin{equation} \label{eqn:defofFmap}
    \left(x_0, g_{\Xi} \cdot a_{\Xi} [x_0, \varphi^{n} x_{0}], g_{\Xi} \cdot a_{\Xi} \varphi^{n} a_{\Xi} \cdot g_{\Xi}^{-1} \Xi \right) \quad \text{is $C_{0}$-aligned}.
\end{equation}
 This map $\Xi \mapsto a_{\Xi}$ is also Borel measurable. 
We now set 
 \[
F (\Xi) := g_{\Xi} \cdot a_{\Xi} \varphi^{n} a_{\Xi} \cdot g_{\Xi}^{-1} \Xi.
\]
Let \[
D := 100\left(C_{0} + \tau_{\varphi} n + \sum_{i=1}^{3} d(x_0, a_i x_0) \right).
\]

\begin{claim*}
We have that 
\begin{equation} \label{eqn:finitetoone}
F \text{ is at most } 3 \cdot \# \{ g \in \Gamma : d(x_0, gx_0) \le D\}\text{-to-one}.
\end{equation}
\end{claim*}

To prove this claim, suppose that we have $\Xi, \Xi' \in \mathcal{H}_{K, R}$ with the same image $F(\Xi) = F(\Xi') =: (\xi_0, t_0) \in \partial^{h} X \times \R$. 

Let $\{u_i\}_{i \in \N} \subset X$ be a sequence converging to $\xi_0$. Up to a subsequence, 
$$
\left( x_0, g_{\Xi} a_{\Xi} [x_0, \varphi^{n} x_0], u_i \right) \quad \text{is $C_{0}$-aligned for all $i \in \N$}.
$$
Here, note that $d(x_0, \varphi^n x_0) \ge \tau_{\varphi} n \ge 100C_{0}$. It follows from Lemma \ref{lem:BGIPFellow} that for each $i \in \N$, there exists $P \in [x_0, u_i]$ such that
$d(g_{\Xi} a_{\Xi} x_0, P) \le 3C_{0}$, and similarly, there exists $Q \in [x_0, u_i]$ such that $d(g_{\Xi'} a_{\Xi'} x_0, Q) \le 3C_{0}$.
We have three cases. 
\begin{enumerate}
\item If $d(x_0, g_\Xi x_0) < d(x_0, g_{\Xi'} x_0) - 0.5D$: In this case, along $[x_0, u_i]$, $P$ is closer to $x_0$ than $Q$ is, and $d(P, Q) \ge 0.5D - 6C_{0}$. 
For convenience we write $\eta := g_{\Xi} a_{\Xi} [x_0, \varphi^{n} x_0]$. Note that by Corollary \ref{cor:BGIPFellow}(2),
$$
\begin{aligned}
d(u_i, \eta) & \ge  d(u_i, g_{\Xi} a_{\Xi} x_0) - d(g_{\Xi} a_{\Xi}x_0, g_{\Xi} a_{\Xi}\varphi^n x_0) \\
& \ge d(u_i, P) - 3C_{0} - d(x_0, \varphi^n x_0) \\
& \ge d(u_i, Q) + 0.2 D \\
& > d(u_i, g_{\Xi'} a_{\Xi'} x_0) + d(a_{\Xi'} x_0, x_0) + C_{0}\\
& \ge d(u_i, g_{\Xi'} x_0) + C_{0}.
\end{aligned}
$$
This implies that $d(\eta, [u_i, g_{\Xi'} x_0]) \ge C_{0}$. Since $\eta$ is $C_{0}$-contracting, $\diam \pi_{\eta} ([u_i, g_{\Xi'}x_0]) \le C_{0}$. Since $(x_0, \eta, u_i)$ is $C_{0}$-aligned, we also have that
\begin{equation} \label{eqn:case1}
(x_0, \eta, g_{\Xi'} x_0) \quad \text{is $2C_{0}$-aligned.}
\end{equation}

Denote by $\xi' \in \partial^{h} X$ the Busemann cocycle component of $\Xi'$ and let $\{z_i'\}_{i \in \N} \subset X$ be a sequence converging to $\xi'$. We claim that
\begin{equation} \label{eqn:case1-2}
(\eta, z_i') \quad \text{is $(K + 30C_{0})$-aligned for all large $i \in \N$.}
\end{equation}
Suppose to the contrary that, passing to a subsequence, $(\eta, z'_i)$ is never $(K+30C_{0})$-aligned. We choose \[
p \in \pi_{\eta}(x_{0}), \quad q \in \pi_{\eta}(z_{i}'), \quad \text{and} \quad r \in \pi_{\eta}(g_{\Xi'} x_{0}).
\]
By Equation \eqref{eqn:case1} and the assumption that $(\eta, z_i')$ is not $(K + 30C_{0})$-aligned, $r$ is closer to $g_{\Xi}a_{\Xi} \varphi^{n}x_0$ than $p$ and $q$ are. Moreover, we have $d(p, r), d(q, r) \ge K + 25C_{0}$. Now, it follows from Lemma \ref{lem:BGIPFellow}  that 
 \[\begin{aligned}
d(x_0, z'_i) & \le d(x_0, \eta) + d(p, q) + d(\eta, z'_i), \\
d(g_{\Xi'} x_0, z'_i) & =_{8C_{0}} d(g_{\Xi'}x_0, \eta) + d(r, q) + d(\eta, z'_i), \quad \text{and} \\
d(x_0, g_{\Xi'} x_0) & =_{8C_{0}} d(x_0, \eta) + d(p, r) + d(\eta, g_{\Xi'} x_0).
\end{aligned}
\]
Then we have
$$\begin{aligned}
d(x_0, g_{\Xi'} x_0) & + d(g_{\Xi'} x_0, z'_i) - d(x_0, z'_i) \\
& \ge 2 d(g_{\Xi'} x_0, \eta) + 2 \min ( d(p, r), d(q, r)) - 16 C_{0} \\
& \ge 2K + 34 C_{0}
\end{aligned}
$$
In particular, 
$d(x_0, z'_i) - d(g_{\Xi'} x_0, z'_i) \le d(x_0, g_{\Xi'} x_0) - K - C_{0}$.
Taking the limit $i \to + \infty$, we have
$$
\beta_{\Xi'}( x_0, g_{\Xi'} x_0) \le d(x_0, g_{\Xi'} x_0) - K - C_{0}.
$$
This contradicts  the definition of $g_{\Xi'}$, and therefore Equation \eqref{eqn:case1-2} follows.

By Lemma \ref{lem:BGIPFellow}, it follows from Equation \eqref{eqn:case1} and Equation \eqref{eqn:case1-2} that $[x_0, z_{i}']$ passes through the $4C_{0}$-neighborhood of $g_{\Xi} a_{\Xi} x_0$, for all large $i \in \N$. This implies
$$
\beta_{\Xi'}(x_0, g_{\Xi}a_{\Xi}x_0) \ge d(x_0, g_{\Xi} a_{\Xi} x_0) - 8C_{0}.
$$
Since $K > 100 C_{0} + 2 \sum_{i = 1}^3 d(x_0, a_i x_0)$, we have 
$$
\begin{aligned}
\beta_{\Xi'}(x_0, g_{\Xi} x_0) & = \beta_{\Xi'}(x_0, g_{\Xi} a_{\Xi} x_0) + \beta_{\Xi'}(g_{\Xi} a_{\Xi} x_0, g_{\Xi} x_0) \\
& \ge d(x_0, g_{\Xi} a_{\Xi} x_0) - 8C_{0} - d(a_{\Xi} x_0, x_0) \\
& \ge d(x_0, g_{\Xi} x_0) - 8 C_{0} - 2 d( a_{\Xi} x_0, x_0) \\
& > d(x_0, g_{\Xi} x_0)  - K.
\end{aligned}
$$
Meanwhile, we also have $k < d(x_0, g_{\Xi} x_0) < d(x_0, g_{\Xi'}x_0) - 0.5 D$. This contradicts to the definition of $g_{\Xi'}$ that $d(x_0, g_{\Xi'} x_0)$ is minimal among the elements of $\Ga$ satisfying Equation \eqref{eqn:defofgXi}.

\item  If $d(x_0, g_{\Xi'} x_0) < d(x_0, g_{\Xi} x_0) - 0.5D$: In this case, one can obtain a similar contradiction as in (1).

\item If $d(x_0, g_{\Xi}x_0) =_{0.5D} d(x_0, g_{\Xi'} x_0)$: Recall that for each fixed $i \in \N$, we have $P, Q \in [x_0, u_i]$ such that $d(g_{\Xi}a_{\Xi} x_0, P), d(g_{\Xi'}a_{\Xi'} x_0, Q)\le 3C_{0}$. Hence, we have
$$
d(g_{\Xi} x_0, P), d(g_{\Xi'} x_0, Q) \le 0.1 D.
$$
Since both $P$ and $Q$ belong to the geodesic $[x_0, u_i]$, this, together with $d(x_0, g_{\Xi} x_0) =_{0.5D} d(x_0, g_{\Xi'} x_0)$, implies
$$
d(g_{\Xi} x_0, g_{\Xi'} x_0) \le D.
$$

Now when $\Xi$ is given (and hence $a_{\Xi}, g_{\Xi}$ are given as well),
\[
\Xi' = g_{\Xi'} a_{\Xi'}^{-1} \varphi^{-n} a_{\Xi'}^{-1} (g_{\Xi'}^{-1} g_{\Xi}) a_{\Xi} \varphi^{n} a g_{\Xi}^{-1} \Xi
\]
is determined by $g_{\Xi'}^{-1} g_{\Xi}$ and $a_{\Xi'} \in \{a_1, a_2, a_3\}$. The number of these choices is at most $3 \cdot \# \{ g \in \Gamma : d(x_0, gx_0) \le D\}$. 
\end{enumerate}
Therefore,  Equation \eqref{eqn:finitetoone} follows.

\medskip

We simply write $M := 3 \cdot \# \{ g \in \Gamma : d(x_0, gx_0) \le D\}$. Then we have
\[\begin{aligned}
\mu(F(\mathcal{H}_{K, R})) &= \mu \left( \bigcup_{g \in \Gamma, a \in \{a_1, a_2, a_3\}} F \left( \{\Xi \in \mathcal{H}_{K, R} : g_{\Xi} = g, a_{\Xi} = a\} \right) \right) \\
&\ge \frac{1}{M}\sum_{g \in \Gamma, a \in \{a_1, a_2, a_3\}} \mu \left( F \left( \{\Xi \in \mathcal{H}_{K, R} : g_{\Xi} = g, a_{\Xi} = a\} \right) \right) \\
&=  \frac{1}{M}\sum_{g \in \Gamma, a \in \{a_1, a_2, a_3\}} \mu \left( ga\varphi^n a g^{-1} \{\Xi \in \mathcal{H}_{K, R} : g_{\Xi} = g, a_{\Xi} = a\} \right) \\
&= \frac{1}{M}\sum_{g \in \Gamma, a \in \{a_1, a_2, a_3\}} \mu \left( \{\Xi \in \mathcal{H}_{K, R} : g_{\Xi} = g, a_{\Xi} = a\} \right) \\
&= \frac{1}{M}\mu(\mathcal{H}_{K, R}).
\end{aligned}
\]

Now to see the image of $F$, let $\Xi = (\xi, t) \in \mathcal{H}_{K, R}$. For simplifity, write $g := g_{\Xi}$ and $a := a_{\Xi}$. Then
$$
F(\Xi) = ( ga \varphi^n a g^{-1} \xi, t + \beta_{\xi} ( (g a \varphi^n a g^{-1})^{-1} x_0, x_0) )
$$
For each sequence $\{u_i\}_{i \in \N} \subset X$ converging to $ga \varphi^n a g^{-1} \xi \in \partial^{h} X$, we have 
$$\begin{aligned}
\beta_{\xi} ( (g a \varphi^n a g^{-1})^{-1} x_0, x_0) & = \beta_{g a \varphi^n a g^{-1}\xi}(  x_0, g a \varphi^n a g^{-1}x_0) \\
& = \lim_{i \to + \infty} d(x_0, u_i) - d( g a \varphi^n a g^{-1} x_0, u_i).
\end{aligned}
$$
By Equation \eqref{eqn:defofFmap} and Lemma \ref{lem:BGIPFellow}, we have for all large $i \in \N$ that
$$\begin{aligned}
\beta_{\xi} ( (g a \varphi^n a g^{-1})^{-1} x_0, x_0) & =_{15C_{0}} d(x_0, g a x_0) + d(x_0, \varphi^{n} x_0) + d(g a \varphi^n x_0, u_i) \\
& \qquad \quad - d( g a \varphi^n a g^{-1} x_0, u_i) \\
& =_{C_{0}} d(x_0, g a x_0) + d(x_0, \varphi^{n} x_0) \\
& \qquad \quad + \beta_{g a \varphi^n a g^{-1}\xi} ( g a \varphi^n x_0, g a \varphi^n a g^{-1} x_0) \\
& =_{d(x_0, a x_0)} d(x_0, g x_0) +  d(x_0, \varphi^{n} x_0)   \\
& \qquad \quad + \beta_{\xi}(g a^{-1} x_0, g x_0) + \beta_{\xi}(g x_0, x_0) \\
& =_{d(x_0, a x_0)} d(x_0, g x_0) +  d(x_0, \varphi^{n} x_0) +  \beta_{\xi}(g x_0, x_0) \\
& =_K d(x_0, \varphi^{n} x_0).
\end{aligned}
$$
Therefore,
$$
t + \beta_{\xi} ( (g a \varphi^n a g^{-1})^{-1} x_0, x_0) \in [-R - D, R + D].
$$
In addition, by Equation \eqref{eqn:defofFmap}, we have $d(x_0, ga x_0) > k - \sum_{i = 1}^3 d(x_0, a_i x_0)$ and that $(x_0, ga [x_0, \varphi^n x_0], F(\Xi))$ is $C_{0}$-aligned.

This implies that $F(\mathcal{H}_{K, R})$ is contained in
 \[
B_{k;n} := \left\{ (\zeta, s) \in \mathcal{H} : 
\begin{matrix}
-R-D \le s \le R+D \text{ and } \exists h \in \Ga \text{ such that}\\
d(x_0, h x_0) > k - \sum_{i = 1}^3 d(x_0, a_i x_0) \text{ and} \\ 
(x_0, h [x_0, \varphi^n x_0], \zeta) \text{ is $C_{0}$-aligned}
\end{matrix}\right\}.
\]
Hence, we have
$$
\mu (B_{k;n}) \ge \mu(\mathcal{H}_{K, R})/M > 0.
$$
Note that the set $B_{k;n}$ is decreasing in $k$. Since $\mu$ is a Radon measure and $B_{k;n} \subset \partial^{h} X \times [-R - D, R + D]$ which is \emph{compact}, we have $\mu(B_{k;n}) < + \infty$. Therefore, setting
$$
B_n := \bigcap_{k > 0} B_{k;n},
$$
we have
$$
\mu(B_n) = \lim_{k \to + \infty} \mu(B_{k;n}) \ge \mu(\mathcal{H}_{K, R})/M > 0.
$$

Now, $\Gamma \cdot B_{n}$ is a $\Gamma$-invariant set of positive $\mu$-measure. Hence, by the $\Ga$-ergodicity of $\mu$, we have that $\Ga \cdot B_n$ is $\mu$-conull, and therefore
$$
\bigcap_{n} \Ga \cdot B_n \quad \text{is $\mu$-conull}.
$$
We then show that for each $(\zeta, s) \in \bigcap_{n} \Ga \cdot B_n $, we have $\zeta \in \La_{\varphi, 2C_{0}}(\Ga)$. This finishes the proof by Lemma \ref{lem:squeezedInv}.

Let $(\zeta, s) \in \bigcap_{n} \Ga \cdot B_n$. Then for each large enough $n \in \N$, there exists  $h_0 \in \Ga$ so that
$$
(x_0, h [x_0, \varphi^n x_0], h_0^{-1}\zeta) \quad \text{is $C_{0}$-aligned for infinitly many } h \in \Ga.
$$
In other words,
$$
( h_0 x_0, h_0 h [x_0, \varphi^n x_0], \zeta) \quad \text{is $C_{0}$-aligned for infinitely many } h \in \Ga.
$$
Among infinitely many such $h \in \Ga$, we can choose one such that
$$d(h_0 x_0, h_0 h [x_0, \varphi^n x_0]) > d(x_0, h_0 x_0) + C_{0}$$
and hence 
$$
d([h_0 x_0, x_0], h_0 h [x_0, \varphi^n x_0]) > C_{0}.
$$
Since $h_0 h [x_0, \varphi^n x_0]$ is $C_{0}$-contracting, we now have that $(x_0, h_0 h [x_0, \varphi^n x_0])$ is $2C_{0}$-aligned. Therefore,
$$
( x_0, h_0 h [x_0, \varphi^n x_0], \zeta) \quad \text{is $2C_{0}$-aligned.}
$$
Since this holds for all large $n \in \N$, we conclude $\zeta \in \La_{\varphi, 2C_{0}}(\Ga)$.
\end{proof}

Theorem \ref{thm:radonCharge} applies to each contracting isometry of $\Ga$. We thus define:

\begin{definition}\label{dfn:Myr}
    Let $\Ga < \Isom(X)$ be a non-elementary subgroup. For each contracting isometry $\varphi \in \Ga$, let $C(\varphi)> 0$ as in  Lemma \ref{lem:BGIPHeredi}. We then define \[
\Lambda_{\rm Myr} (\Gamma) := \bigcap_{\varphi \in \Ga, \text{ contracting}} \La_{\varphi, C(\varphi)}(\Ga).
\]
In other words, every point in $\La_{\rm Myr} (\Ga)$ is a $(\varphi, C(\varphi))$-guided limit point of $\Ga$ for all contracting $\varphi \in \Ga$.
\end{definition}

\begin{remark}\label{rem:Myr}
The set $\Lambda_{\rm Myr} (\Ga) $ is closely related to the \emph{Myrberg limit set}. A point  $\xi \in \partial^{h} X$ is a called a \emph{Myrberg limit point} of $\Ga$ if there exists $c>0$ such that for each $g_1, g_2 \in \Ga$, there exists $h \in \Ga$ such that
$$
h[x, \xi) \text{ is $c$-close to $hx_0$ and $hgx_0$, with $hx_0$ coming earlier.}
$$
The Myrberg limit set is the set of all Myrberg limit points of $\Ga$.

When $\Gamma < \Isom(X)$ is non-elementary, Yang proved in \cite[Corollary 4.17]{yang2024conformal} that the Myrberg limit set of $\Ga$ is equal to the intersection of $(\varphi, C(\varphi))$-guided limit sets for all contracting isometries $\varphi \in \Gamma$. 
\end{remark}

Considering Lemma \ref{lem:squeezedInv}, Theorem \ref{thm:radonCharge} implies that:
\begin{Cor}\label{cor:radonCharge}
Let $\Ga < \Isom(X)$ be a non-elementary subgroup. Let $\mu$ be a $\Ga$-invariant ergodic Radon measure on $\mathcal{H}$ supported on $\La_{c}(\Ga) \times \R$. Then the measure $\mu$ is supported on
$$
 \La_{\rm Myr}(\Ga) \times \R \subset \mathcal{H}.
$$
\end{Cor}

\subsection{Neighborhoods in squeezed limit sets}

We have seen that the squeezed limit sets is the genuine region for invariant ergodic measures.
Let us  now introduce a notion of neighborhoods of squeezed limit points. For $g, \varphi \in \Isom(X)$, $C > 0$, and $n \in \N$, we set
$$
U_C ( g; \varphi, n) := \left\{ \xi \in \partial^{h} X : \textrm{$(x_{0}, g[x_{0}, \varphi^{n} x_{0}], \xi)$ is $C$-aligned}\right\}.
$$
The interesting case is  where $\varphi$ is a squeezing isometry.

\begin{lem}\label{lem:nbdBasis}
    Let $\Ga < \Isom(X)$ be a non-elementary subgroup containing a squeezing isometry $\varphi \in \Ga$, and let $C=C(\varphi) > 0$ be as in Lemma \ref{lem:BGIPHeredi}. Then 
    $$\{U_{C}(g; \varphi, n) : g \in \Gamma, n \in \N \}$$ 
    forms a basis for the topology of $\La_{\varphi, C}(\Ga) \subset \partial^{h}X$.

In other words, for each  $\xi \in \La_{\varphi, C}(\Ga)$, for each open set $O \subset \partial^{h} X$ with $\xi \in O$ and for each $N \in \N$, there exist $g \in \Ga$, $n > N$, and an open set $V \subset \partial^{h} X$ such that
\[
\xi \in V \subset U_{C}(g; \varphi, n) \subset O.
\]
\end{lem}

\begin{proof}
Let $\gamma : \R \rightarrow X$ be the unique axis of $\varphi$. We fix $\xi \in \La_{\varphi, C}(\Ga)$, an open set $O \subset \partial^{h} X$ containing $\xi$, and $N \in \N$.

Let us first recall that $\partial^{h}X$ is given a compact-open topology. Hence, the set of the form 
\[
O_{R, \epsilon} := \{ \zeta \in \partial^{h}X : \abs{\xi(x) - \zeta(x)} < \epsilon \textrm{ for all }  x \in \mathcal{N}_{R}(x_{0})\} \quad \text{for $R, \epsilon > 0$}
\]
forms a local basis for $\xi$. Fix $R, \epsilon > 0$ such that $O_{R, \epsilon} \subset O$. We may assume that $\epsilon < C$.

Let $L = L(0.01\epsilon) > 0$ be the constant as in Definition \ref{dfn:squeezing}  for $\ga$. 
Let $R' := \lceil \frac{R+100C}{\tau_{\varphi}} \rceil$ and $L' := \lceil \frac{2L + 100C}{\tau_{\varphi}} \rceil$, and take $k \in \N$ such that 
 \[
k > R'+ L' + \lceil100C/\tau_{\varphi} \rceil + N.
\]
Let $\{z_{i}\}_{i \in \N} \subset X$ be a sequence converging to  $\xi$. Since $\xi \in \La_{\varphi, C}(\Ga)$, there exists $h \in \Gamma$ such that $\xi \in U_{C}(h; \varphi, k)$, i.e., \[
\left(x_0, h[x_0, \varphi^k x_0], z_{i}\right) \quad \textrm{is $C$-aligned for all  large $i \in \N$}.
\]
Since $\tau_{\varphi} k > 2C$, Lemma \ref{lem:BGIPHeredi}(4) tells us that 
\begin{equation}\label{eqn:x0GanB}
    \begin{aligned}
\pi_{h \gamma}(x_{0}) &\subset h  \gamma \left( (-\infty, 2C] \right) \quad \text{and} \\
\pi_{h \gamma}(z_{i}) &\subset h\gamma([\tau_{\varphi} k - 2C, +\infty)) \quad \textrm{for all  large $i \in \N$}.
\end{aligned}
\end{equation}
By the $(1, 4C)$-Lipschitzness of $\pi_{h\gamma}(\cdot)$ in Corollary \ref{cor:BGIPFellow}(1), we have 
\begin{equation}\label{eqn:xInPlaceOf}
\pi_{h\gamma}(x) \subset h\gamma \left( (-\infty, 6C+ R]\right) \quad \text{for all } x \in \mathcal{N}_{R}(x_0).
\end{equation}
In view of Equation \eqref{eqn:x0GanB} and Lemma \ref{lem:BGIPHeredi}(3), $(x_{0}, h\varphi^{R'} [x_0, \varphi^{L'}x_{0}])$ is $C$-aligned. Similarly, $(h\varphi^{R'} [x_0, \varphi^{L'}x_{0}], z_i)$ is $C$-aligned eventually. Therefore, 
$$\xi \in U_{C}(h\varphi^{R'}; \varphi, L').$$

\medskip

Now let $\zeta \in U_{C}(h \varphi^{R'}; \varphi, L')$ and take a  sequence $\{z'_{i}\}_{i \in \N} \subset X$ converging to $\zeta$. Then by Lemma \ref{lem:BGIPHeredi}(4),
$$\pi_{h \gamma}(z_{i}') \subset h\gamma\left([(R'+L')\tau_{\varphi} - 2C, +\infty)\right) \quad \text{for all large } i \in \N.
$$
Combining this with Equation \eqref{eqn:x0GanB} and Equation  \eqref{eqn:xInPlaceOf}, we can apply Lemma \ref{lem:squeezing} and conclude that for each $x \in \mathcal{N}_R(x_0)$, 
\[
d(x, z_{i}) - d(x_{0}, z_{i}) =_{\epsilon/2} d(x, z'_{i}) - d(x_{0}, z_{i}') \quad \text{for all large } i \in \N.
\]
This implies that $\abs{\xi(x) - \zeta(x)} < \epsilon$ for all $x \in \mathcal{N}_R(x_0)$, and therefore
\[
\xi \in U_{C}(h \varphi^{R'}; \varphi, L') \subset O_{R, \epsilon} \subset O.
\]

\medskip

Now, set \[
V := \{ \zeta \in \partial^{h}X : \abs{\xi(x) - \zeta(x)} < \epsilon \text{ for all } x \in h\gamma([0, \tau_{\varphi} k]) \}
\]
which is an open neighborhood of $\xi$. We then show $V \subset U_{C}(h \varphi^{R'}; \varphi, L')$, which finishes the proof.

Let $\zeta \in V$ and $\{z'_i\}_{i \in \N} \subset X$ a sequence converging to $\zeta$. 
\begin{claim*} 
    We have \begin{equation}\label{eqn:zetaWinsP}
\pi_{h\gamma} (z'_i) \subset h\gamma \left( [ \tau_{\varphi} k - 20C, +\infty)\right)\quad \text{for all large } i \in \N.
\end{equation}
\end{claim*}
If not, then after passing to a subsequence,
$$
d\left( z'_i, h\gamma(\tau_{\varphi} k - 2C) \right) - d\left( z'_i, h\gamma(\tau_{\varphi} k - 20C) \right)  =_{8C} 18C \quad \text{for all } i \in \N
$$
by Corollary \ref{cor:BGIPFellow}(2). On the other hand, by Equation \eqref{eqn:x0GanB}, it follows from Corollary \ref{cor:BGIPFellow}(2) that after passing to a subsequence,
$$
d\left( z_i, h\gamma(\tau_{\varphi} k - 20C) \right) - d\left( z_i, h\gamma(\tau_{\varphi} k - 2C) \right) =_{8C} 18C \quad \text{for all } i \in \N.
$$
These imply that
$$\begin{aligned}
 \beta_{\zeta} ( h \ga ( \tau_{\varphi} k - 2C), h \ga ( \tau_{\varphi} k - 20C) ) & =_{8C} 18 C \quad \text{and} \\
\beta_{\xi} ( h \ga ( \tau_{\varphi} k - 20C) , h \ga ( \tau_{\varphi} k - 2C)) & =_{8C} 18 C .
\end{aligned}
$$
Equivalently,
$$\begin{aligned}
 \zeta( h \ga ( \tau_{\varphi} k - 2C)) - \zeta( h \ga ( \tau_{\varphi} k - 20C) ) & =_{8C} 18 C \quad \text{and} \\
\xi( h \ga ( \tau_{\varphi} k - 2C)) - \xi( h \ga ( \tau_{\varphi} k - 20C) ) & =_{8C} -18 C .
\end{aligned}
$$
This contradicts the fact that $\abs{\xi(x) - \zeta(x)} < \epsilon$ for all $x \in h \ga ([0, \tau_{\varphi}k])$.

\medskip

Hence, Equation \eqref{eqn:zetaWinsP} holds, and Lemma \ref{lem:BGIPHeredi}(3) tells us that 
$$
(x_0, h\varphi^{R'} [x_0, \varphi^{L'}x_0], \zeta) \quad \text{is $C$-aligned.}
$$
This shows $\zeta \in U_{C}(h \varphi^{R'}; \varphi, L')$, and therefore $V \subset U_{C}(h \varphi^{R'}; \varphi, L')$ as desired.
\end{proof}

\subsection{Quasi-invariance under translations} \label{subsec:trqi}

For $a \in \R$, consider a map $T_a : \mathcal{H} \to \mathcal{H}$ given by 
$
(\xi, t) \mapsto (\xi, t + a)
$.
For a Radon measure $\mu$ on $\mathcal{H}$, we consider its pullback  measure $T_a^*\mu$: for each Borel subset $E \subset \mathcal{H}$,
$$T_a^* \mu (E) := \mu(T_a E).$$
For a contracting  $g \in \Isom(X)$, we simply write $T_{g} := T_{\tau_g}$.
We show that invariant ergodic measures on $\mathcal{H}$ are quasi-invariant under this translation.

\begin{theorem} \label{thm:trbytrlengthqi}
    Let $\Ga < \Isom(X)$ be a non-elementary subgroup containing a squeezing isometry. Let $\mu$ be a $\Ga$-invariant ergodic Radon measure on $\mathcal{H}$ supported on $\La_c(\Ga) \times \R$. Then for a squeezing isometry $\varphi \in \Ga$, there exists $\la \ge 0$ such that
    $$
    \frac{ d T_{\varphi}^* \mu}{d \mu} = e^{\la} \quad \text{a.e.}
    $$

\end{theorem}

\begin{proof}
Let $\varphi \in \Ga$ be a squeezing isometry and let $C= C(\varphi) > 0$ be as in Lemma \ref{lem:BGIPHeredi}, with the choice of axis $\ga : \R \to X$. Note that for every $n \in \N$, $\tau_{\varphi^n} = n \cdot \tau_{\varphi}$ and $C(\varphi^n) = C(\varphi)$. Since $T_{\varphi}$ commutes with the $\Ga$-action on $\mathcal{H}$, $T_{\varphi^{-n}}^*\mu$ is also $\Ga$-invariant and ergodic. Hence, if $\frac{d T_{\varphi^n}^* (T_{\varphi^{-n}}^*\mu)}{d T_{\varphi^{-n}}^*\mu} = e^{\la_1}$ and $\frac{d T_{\varphi^{n+1}}^* (T_{\varphi^{-n}}^*\mu)}{d T_{\varphi^{-n}}^*\mu} = e^{\la_2}$ for some $\la_1, \la_2 \in \R$, then $\frac{d T_{\varphi}^* \mu}{d \mu} = e^{\la_2 - \la_1}$. Therefore, it suffices to consider the case that
$$
\tau_{\varphi} > 100 C.
$$

We first aim to show that 
\begin{equation} \label{eqn:trabscont}
(T_{\varphi}^{*} \nu)(E) \ge \nu(E)
\end{equation}
for each Borel subset $E \subset \mathcal{H}$. By Corollary \ref{cor:radonCharge}, $\mu$ is supported on $\La_{\varphi, C}(\Ga) \times \R$.

\medskip
{\bf \noindent Step 1.} First consider the case that $E = K \times I$ for a compact subset $K \subset \La_{\varphi, C}(\Ga)$ and a compact interval $I \subset \R$. 

We fix some open subset $O \subset \partial^{h}X$ such that $K \subset O$ and $\epsilon > 0$. Let $L = L(0.001\epsilon)>0$ be as in Definition \ref{dfn:squeezing} for $\gamma$. By Lemma \ref{lem:nbdBasis}, for each $\xi \in K$, there exist an element $g(\xi) \in \Gamma$ and $n(\xi) > (2L+100C)/\tau_{\varphi} + 4$ such that 
  \[
\xi \in U_{C} \left(g(\xi); \varphi, n(\xi) \right) \subset O.
\]
Let $\mathcal{U} := \left\{ U_C \left(g(\xi); \varphi, n(\xi)\right) : \xi \in K \right\}$, which is a countable collection of sets. For convenience, let us enumerate $\mathcal{U}$ based on their distance from $x_0$, i.e, let \[
\mathcal{U} = \{U_{1}, U_{2}, \ldots\}
\]
where 
$
U_i := U_C (g_i; \varphi, n_i)
$
for each $i \in \N$
so that \[
d(x_0, g_1 \varphi^{n_1} x_0) \le 
d(x_0, g_2 \varphi^{n_2} x_0) \le \cdots.
\]

We will now define a subcollection 
$$\mathcal{V} := \{U_{i(1)}, U_{i(2)}, \ldots\} \subset \mathcal{U}$$ by inductively defining $i(1), i(2), \ldots$. We let $i(1) = 1$. Now, having defined $i(1), \ldots, i(N)$, define $i(N+1)$ as the smallest $j \in \N$ such that $U_{j}$ is disjoint from $U_{i(1)} \cup \cdots \cup U_{i(N)}$.

For each $l \in \N$, we set 
\begin{equation} \label{eqn:defofCl}
C_{l} := U_{i(l)} \cup \bigcup \left\{ U_{k} : k \ge i(l), U_{k} \cap U_{i(l)} \neq \emptyset \right\}.
\end{equation}
Then $\{C_{l} : l \in \N \}$ is a covering of $K$ contained in $O$. Indeed, for $k \ge 1$, if $k \neq i(l)$ for all $l \in \N$, then $U_k$ intersects $U_{i(1)} \cup \cdots \cup U_{i(l_0)}$ where $i(l_0)$ is the maximal index less than $k$.

\begin{claim*}
For each $l \in \N$, 
\begin{equation} \label{eqn:trqiclaimClinU}
    C_{l} \subset U_C \left(g_{i(l)}; \varphi, n_{i(l)} - 1\right).
\end{equation}
\end{claim*}

To see this claim, let $k \ge i(l)$ be such that $U_{k} = U_C (g_k; \varphi, n_k)$ and $U_{i(l)}$ have a common element $\xi$. Then for every $z \in X$ close enough to $\xi$ in $\overline{X}^{h}$, 
\begin{equation} \label{eqn:trqiclaim}
\left( x_0, g_{i(l)} [x_0, \varphi^{n_{i(l)}} x_0], z\right) \quad \text{is $C$-aligned.}
\end{equation}
 By Lemma \ref{lem:BGIPFellow}, there exists $q \in [x_0, z]$ that is $3C$-close to $g_{i(l)} \varphi^{n_{i(l)}}x_0$. Similarly, by the condition $\xi \in U_{k}$, $[x_0, z]$ contains a point $p$ that is $3C$-close to $g_{k} \varphi^{n_k} x_0$. Since $k \ge i(l)$,  our enumerating convention  tells us that \[
d(x_0, p) \ge d(x_0, g_k \varphi^{n_k} x_0) - 3C \ge d(x_0, g_{i(l)} \varphi^{n_{i(l)}} x_0 ) - 3C \ge d(x_0, q) - 6C.
\]
In other words, 
$$d(z, p) \le d(z, q) + 6C.$$

If $d(z, p) \le 10C$, then we have $d(z, g_{k} \varphi^{n_k} x_0) \le 13C$. By the $(1, 4C)$-Lipschitzness of $\pi_{g_{i(l)} [x_0, \varphi^n_{i(l)} x_0 ]}(\cdot)$ in Corollar \ref{cor:BGIPFellow}(1), we have\[
\diam  \pi_{g_{i(l)}[x_0, \varphi^{n_{i(l)}} x_0]} (\{z, g_{k} \varphi^{n_{k}} x_0\}) \le 17C.
 \]

 If $d(z, p) > 10C$, then we take a point $p^{\dagger} \in [z, p]$ such that $d(p, p^{\dagger}) = 10C$. Then $\diam [p^{\dagger}, z] = d(p,z) - 10C$, and hence
 $$\begin{aligned}
 d\left( g_{i(l)}[x_0, \varphi^{n_{i(l)}} x_0], [p^{\dagger}, z] \right) & \ge d\left(g_{i(l)}[x_0, \varphi^{n_{i(l)}}x_0], z\right) - \diam  [p^{\dagger}, z]  \\
 & \ge d(z, q) - 3C - d(z, p) + 10C \\
 & \ge C.
 \end{aligned}
$$
 By the $C$-contracting property of $g_{i(l)}[x_0, \varphi^{n_{i(l)}}x_0]$, we have \[
\diam \pi_{g_{i(l)}[x_0, \varphi^{n_{i(l)}} x_0]} (\{z, p^{\dagger}\})  \le C.
 \]
 Since $d(p, g_k \varphi^{n_k} x_0) \le 3C$, we have $d(p^{\dagger}, g_k \varphi^{n_k} x_0) \le 13C$, and hence
  \[
\diam \pi_{g_{i(l)}[x_0, \varphi^{n_{i(l)}} x_0]} (\{z, g_{k} \varphi^{n_{k}} x_0\}) \le 18C
 \]
 by Corollary \ref{cor:BGIPFellow}(1).

 Hence, in any case, together with Equation \eqref{eqn:trqiclaim}, we have that 
 \begin{equation} \label{eqn:secondineq19C}
(x_0, g_{i(l)}[x_0, \varphi^{n_{i(l)}} x_0], g_k \varphi^{n_k} x_0) \quad \text{is $19C$-aligned.}
 \end{equation}

Now let $\zeta \in U_k$ be arbitrary. Then for every $z'\in X$ close to $\zeta$ in $\overline{X}^{h}$,
$$
(x_0, g_k [x_0, \varphi^{n_k} x_0], z') \quad \text{is $C$-aligned}
$$
and hence 
$[x_0, z']$ passes through the $3C$-neighborhood of $g_k \varphi^{n_k} x_0$ as before. Hence, we have 
\[\begin{aligned}
\diam   \pi_{g_{i(l)}[x_0, \varphi^{n_{i(l)}} x_0]}  ([x_{0}, z'])  
& \ge d(g_{i(l)}x_0, g_{i(l)}\varphi^{n_{i(l)}} x_0 ) \\
& \quad - \diam \left( \pi_{g_{i(l)}[x_0, \varphi^{n_{i(l)}} x_0]} (x_{0}) \cup \{ g_{i(l)} x_{0}  \} \right)\\
&\quad - \left( d([x_0, z'], g_k \varphi^{n_k} x_0 ) + 4C \right)\\
& \quad - \diam \left( \pi_{g_{i(l)}[x_0, \varphi^{n_{i(l)}} x_0]} (g_k \varphi^{n_k} x_0) \cup \{ g_{i(l)} \varphi^{n_{i(l)}} x_{0} \} \right) \\
&\ge 100C - C - ( 3C+ 4C) - 19C \\
& \ge 70C
\end{aligned}
\]
where we applied Corollary \ref{cor:BGIPFellow}(1) in the first inequality, and Equation \eqref{eqn:trqiclaim}, Equation \eqref{eqn:secondineq19C}, and that $d([x_0, z'], g_k \varphi^{n_k} x_0) \le 3C$ in the second.

We then apply Lemma \ref{lem:BGIPFellow} and obtain two points $u, v \in [x_0, z']$ such that 
\begin{enumerate}
\item $[u, v]$ and $\pi_{g_{i(l)}[x_0, \varphi^{n_{i(l)}} x_0]}([x_0, z'])$ are  within Hausdorff distance $4C$,
\item  $\diam \left( \pi_{g_{i(l)}[x_0, \varphi^{n_{i(l)}} x_0]} ([x_0, u]) \cup \{u\} \right) \le 2C$, 
\item $\diam \left( \pi_{g_{i(l)}[x_0, \varphi^{n_{i(l)}} x_0]} ([v, z']) \cup \{v\} \right) \le 2C$, and
\item for each $u' \in \pi_{g_{i(l)}[x_0, \varphi^{n_{i(l)}} x_0]}(x_0)$ and $v' \in \pi_{g_{i(l)}[x_0, \varphi^{n_{i(l)}} x_0]}(z')$, $[u', v']$ and $[u, v]$ are within Hausdorff distance $10C$.
\end{enumerate}
Since $d([x_0, z'], g_k \varphi^{n_k} x_0) \le 3C$, it follows from Equation \eqref{eqn:secondineq19C} and Corollary \ref{cor:BGIPFellow}(1) that there exists $w \in [x_0, z']$ such that 
\begin{equation} \label{eqn:usingnewitem}
\diam \left( \pi_{g_{i(l)}[x_0, \varphi^{n_{i(l)}} x_0]}(w) \cup \{g_{i(l)} \varphi^{n_{i(l)}} x_0\} \right) \le 26 C.
\end{equation}
By the condition (2) above, Equation \eqref{eqn:secondineq19C}, and $d(x_0, \varphi^{n_i(l)}x_0) > 100 C$, we must have  $w \notin [x_0, u]$, and hence  $w \in [v, z']$ or $w \in [u, v]$.

\begin{itemize}
\item If $w \in [v, z']$, then it follows from (3) above that 
$$\diam \left( \pi_{g_{i(l)}[x_0, \varphi^{n_{i(l)}} x_0]}(z') \cup \{g_{i(l)} \varphi^{n_{i(l)}} x_0\} \right) \le 28 C.
$$

\item If $w \in [u, v]$, then it follows from (4) above that for each pair of two points $u' \in \pi_{g_{i(l)}[x_0, \varphi^{n_{i(l)}} x_0]}(x_0)$ and $v' \in \pi_{g_{i(l)}[x_0, \varphi^{n_{i(l)}} x_0]}(z')$, there exists $w' \in [u', v']$ such that $d(w, w') \le 10C$. We then have
$$
\diam \left( \pi_{g_{i(l)}[x_0, \varphi^{n_{i(l)}} x_0]}(w) \cup \{w'\} \right) \le 20 C.
$$
Together with Equation \eqref{eqn:usingnewitem}, $\diam \{w', g_{i(l)} \varphi^{n_{i(l)}} x_0\} \le 46 C$. Since $v'$ is inbetween $w'$ and $g_{i(l)} \varphi^{n_{i(l)}} x_0$, and $v'  \in \pi_{g_{i(l)}[x_0, \varphi^{n_{i(l)}} x_0]}(z')$ is arbitrary, we have
$$
\diam \left( \pi_{g_{i(l)}[x_0, \varphi^{n_{i(l)}} x_0]}(z') \cup \{ g_{i(l)} \varphi^{n_{i(l)}} x_0 \} \right) \le 46 C.
$$
\end{itemize}
Therefore, $(g_{i(l)}[x, \varphi^{n_{i(l)}} x_0], z')$ is $46C$-aligned.

By Lemma \ref{lem:BGIPHeredi}(4), we have 
\begin{equation} \label{eqn:Clclaimcontradiction}
    \begin{aligned}
\pi_{g_{i(l)} \gamma} (z') & \subset g_{i(l)} \gamma \left( [\tau_{\varphi} n_{i(l)} - 47C, +\infty) \right) \\
& \subset  g_{i(l)} \gamma\left( [\tau_{\varphi} (n_{i(l)}-1) + 3C, +\infty) \right) .
    \end{aligned}
\end{equation}
By Lemma \ref{lem:BGIPHeredi}(3), this implies that $(g_{i(l)}[x_0, \varphi^{n_{i(l)}-1} x_0], z')$ is $C$-aligned. This is the case for every $z'$ close to $\zeta \in U_k$, so we conclude 
$$\zeta \in U_{C} \left(g_{i(l)}; \varphi, n_{i(l)} - 1\right).$$ 
The claim is now established.

\medskip

For each $l \in \N$ we now define a map $F_{l} : C_{l} \times I \rightarrow  \mathcal{H}$. For $g= g_{i(l)}$ we let 
\begin{equation} \label{eqn:defofFl}
F_{l} : \Xi \mapsto g \varphi g^{-1} \Xi.
\end{equation}
We have $\mu \left( F_{l}(C_{l} \times I) \right) = \mu(C_l \times I)$ as $\mu$ is $\Gamma$-invariant. 

\begin{claim*}
    We have 
    \begin{equation} \label{eqn:claimforFl}
    F_{l} (C_{l}\times I) \subset  U_{i(l)} \times \text{$(\epsilon$-neighborhood of $I + \tau_{\varphi})$}.
    \end{equation}
\end{claim*}
To see this, we simply write $g = g_{i(l)}$ and $n = n_{i(l)} - 1$. We then fix $\Xi = (\xi, t) \in C_l \times I$. Note that
$$
F_l (\Xi) = (g \varphi g^{-1} \xi, t + \beta_{\xi}(g \varphi^{-1} g^{-1} x_0, x_0)).
$$
Let $\{z_i\}_{i \in \N} \subset X$ be a sequence converging to $\xi$.

We first show that $g \varphi g^{-1} \xi \in U_{i(l)}$, which follows once we show that 
$$
(x_0, g[x_0, \varphi^{n+1} x_0], g \varphi g^{-1} z_i) \quad \text{is $C$-aligned for all large $i \in \N$,}
$$
Since $U_{i(l)} \neq \emptyset$, we already have that $(x_0, g[x_0, \varphi^{n+1} x_0] )$ is $C$-aligned. Now suppose to the contrary that $(g[x_0, \varphi^{n+1} x_0], g \varphi g^{-1} z_i)$ is not $C$-aligned for large $i \in \N$. By Lemma \ref{lem:BGIPHeredi}(3), we have 
$$
\pi_{g \ga}(z_i) \subset g \ga ((-\infty, n \tau_{\varphi}]).
$$
This contradicts Equation \eqref{eqn:Clclaimcontradiction}. Therefore, $g \varphi g^{-1} \xi \in U_{i(l)}$.

For the second component, it suffices to show 
\begin{equation} \label{eqn:trqiclaim2goal2}
\abs{\beta_{\xi}(g \varphi^{-1} g^{-1} x_0, x_0) - \tau_{\varphi}} < \epsilon.
\end{equation} Note that 
$$
\beta_{\xi}(g \varphi^{-1} g^{-1} x_0, x_0) = \lim_{i \to + \infty} d(g \varphi^{-1} g^{-1} x_0, z_i) - d(x_0,  z_i).
$$
By Equation \eqref{eqn:trqiclaimClinU}, 
$$
(x_0, g [x_0, \varphi^{n} x_0 ], z_i ) \quad \text{is $C$-aligned for all large $i \in \N$.}
$$
Hence, for all large $i \in \N$, it follows from Lemma \ref{lem:BGIPHeredi}(4) that
\begin{equation} \label{eqn:trqiclaim2secondcomp}
    \pi_{g \gamma} (x_0) \subset g \gamma \left( (-\infty, 2C] \right) \quad \text{and} \quad \pi_{g \gamma} (z_{i}) \subset g \gamma \left( [n \tau_{\varphi} - 2C, +\infty) \right).
\end{equation}
Since  $n\tau_{\varphi} - 4C > 2L $ and $g \ga$ is squeezing  (Definition \ref{dfn:squeezing}), there exists $p \in [x_0, z_i]$ such that
$$
d(p, g\ga(n \tau_{\varphi}/2)) \le 0.001 \epsilon.
$$

Meanwihle, note that $\left( g\varphi^{-1} g^{-1} x_0, g [x_0, \varphi^{n} x_0] \right)$ is also $C$-aligned; otherwise, we have $\pi_{g \ga}(g \varphi^{-1} g^{-1} x_0) \subset g \ga([0, + \infty))$ by Lemma \ref{lem:BGIPHeredi}(3), and therefore $\pi_{g \ga}(x_0) \subset g \ga([\tau_{\varphi}, + \infty))$ which contradicts Equation \eqref{eqn:trqiclaim2secondcomp}. Hence, it follows from Lemma \ref{lem:BGIPHeredi}(4) that 
$$
\pi_{g\gamma} (g\varphi^{-1} g^{-1} x_0) \subset g\gamma \left( (-\infty, 2C] \right). 
$$
Together with Equation \eqref{eqn:trqiclaim2secondcomp} and $n \tau_{\varphi} - 4C > 2 L + 2\tau_{\varphi}$, the squeezing property of $g \ga$ implies that there exist  $q_{1},q_{2}\in [g\varphi g^{-1} x_0, z_i]$, with $q_1$ coming earlier than $q_2$, such that \[
d\left(q_1, g \gamma(n\tau_{\varphi}/2 - \tau_{\varphi})\right), d\left(q_2, g \gamma(n\tau_{\varphi}/2)\right) < 0.001\epsilon.
\]
 
Now we have \[
\begin{aligned}
d(g\varphi^{-1} g^{-1} x_0, z_i) - d(x_0, z_i) 
&= \left( d(g\varphi^{-1} g^{-1}x_0, q_1) + d(q_1, q_2) + d(q_2, z_i) \right) \\
& \quad - \left( d(x_0, p) + d(p, z_i) \right) \\
&=_{0.006\epsilon} d\left(g\varphi^{-1} g^{-1}x_0,  g\gamma(n\tau_{\varphi}/2 - \tau_{\varphi})\right) \\
& \qquad \quad + d\left(  g\gamma(n\tau_{\varphi}/2 - \tau_{\varphi}),   g\gamma(n\tau_{\varphi}/2)\right) \\
& \qquad \quad + d\left( g \gamma(n\tau_{\varphi}/2), z_i\right) \\
& \qquad \quad - d\left(x_{0},   g \gamma(n\tau_{\varphi}/2)\right) - d\left( g \gamma(n\tau_{\varphi}/2), z_i \right) \\
&= d\left(  g\gamma(n\tau_{\varphi}/2 - \tau_{\varphi}),   g\gamma(n\tau_{\varphi}/2)\right)  = \tau_{\varphi}.
\end{aligned}
\]
Taking the limit $i \to + \infty$, Equation \eqref{eqn:trqiclaim2goal2} follows. This completes the proof of the claim.

\medskip

Now by the above claim and disjointness of $U_{i(l)}$'s,  we have
 \[\begin{aligned}
\mu( O \times (\textrm{$\epsilon$-neighborhood of $I+\tau_{\varphi}$})) &\ge \mu\left( \bigcup_{l} F_l (C_{l} \times I) \right) \\
&= \sum_{l} \mu \left( F_l (C_l \times I)\right) \\
&= \sum_{l} \mu (C_{l} \times I) \\
& \ge \mu(K \times I).
\end{aligned}
\]
Note that $\mu( O \times (\textrm{$\epsilon$-neighborhood of $I+\tau_{\varphi}$})) < + \infty$ since $\mu$ is Radon.
Since $\epsilon > 0$ and an open set $O \supset K$ are arbitrary, we have 
$$
(T_{\varphi}^* \mu)(K \times I) = \mu(K \times (I + \tau_{\varphi})) \ge \mu(K \times I). 
$$

\medskip
{\bf \noindent Step 2.} Consider the case that $E = A \times B$ for Borel $A \subset \partial^{h}X$ and an interval $B \subset \mathbb{R}$. Since $\mu$ is supported on $\La_{\varphi, C}(\Ga) \times \R$, we may assume that $A \subset \La_{\varphi, C}(\Ga)$. By the inner regularity of $\mu$ and $T_{\varphi}^* \mu$, there exist compact subsets $E_1, E_2 \subset E$ such that
$$
\abs{\mu(E) - \mu(E_1)} < \epsilon \quad \text{and} \quad \abs{(T_{\varphi}^*\mu)(E) - (T_{\varphi}^*\mu)(E_2)  } < \epsilon.
$$
Considering projections of $E_1 \cup E_2$ to $A$ and $B$, we obtain compact subsets $K \subset A$ and $I \subset B$ so that 
$$
\abs{\mu(E) - \mu(K \times I)} < \epsilon \quad \text{and} \quad \abs{(T_{\varphi}^*\mu)(E) - (T_{\varphi}^*\mu)(K \times I)  } < \epsilon.
$$
By taking the convex hull of $I$ (recall that $B$ is an interval), we may assume that $I$ is a compact interval. Applying Step 1 to $K \times I$, we have
$$
(T_{\varphi}^*\mu)(E) \ge \mu(E) - 2 \epsilon.
$$
Since $\epsilon > 0$ is arbitrary, $(T_{\varphi}^*\mu)(E) \ge \mu(E)$ follows.

\medskip
{\bf \noindent Step 3.} When $E \subset \mathcal{H}$ is a finite union of open sets of the form $O_1 \times O_2$ for open sets $O_1 \subset \partial^h  X$ and open intervals $O_2 \subset \R$,  $E$ is a disjoint union of finitely many Borel subsets of the form $A \times B$, where $A \subset \partial^{h}X$ is Borel and $B \subset \R$ is an interval. Hence, $(T_{\varphi}^*\mu)(E) \ge \mu(E)$ follows from Step 2.

\medskip
{\bf \noindent Step 4.} When $E\subset \mathcal{H}$ is an open set, $E$ is a countable union of open sets of the form  $O_1 \times O_2$ for open sets $O_1 \subset \partial^h  X$ and open intervals $O_2 \subset \R$. Hence,  $(T_{\varphi}^*\mu)(E) \ge \mu(E)$ follows from Step 3.

\medskip
{\bf \noindent Step 5.} Finally, suppose that $E \subset \mathcal{H}$ is a Borel subset. Then it follows from Step 4 and the outer regularity of $\mu$ and $T_{\varphi}^*\mu$ that
$$
(T_{\varphi}^*\mu)(E) \ge \mu(E).
$$

\medskip
Now we have shown Equation \eqref{eqn:trabscont}. Hence, we can consider the Radon--Nikodym derivative $\frac{d \mu}{d T_{\varphi}^* \mu}$. Since both $\mu$ and $T_{\varphi}^*\mu$ are $\Ga$-invariant, $\frac{d \mu}{d T_{\varphi}^* \mu}$ is $\Ga$-invariant as well. By $\Ga$-ergodicity of $T_{\varphi}^*\mu$, $\frac{d \mu}{d T_{\varphi}^* \mu}$ is constant $T_{\varphi}^*\mu$-a.e., which must be positive. Hence, there exists $\la \in \R$ such that $\frac{d T_{\varphi}^* \mu}{d \mu} = e^{\la}$ $\mu$-a.e., and moreover, $\la \ge 0$ by Equation \eqref{eqn:trabscont}. This completes the proof.
\end{proof}

\subsection{Proof of the rigidity}

Let us now prove Theorem \ref{thm:uniqueRadon}.

\begin{proof}[Proof of Theorem \ref{thm:uniqueRadon}]
By ergodic decomposition, it suffices to consider a $\Ga$-invariant ergodic Radon measure $\mu$ on $\mathcal{H}$ supported on $\La_c(\Ga) \times \R$.

Let 
\[
A := \left\{ a \in \mathbb{R} : \exists \lambda (a) \in \R \text{ such that } \frac{d T_{a}^{*} \mu}{d \mu} = e^{\lambda(a)} \text{ a.e.} \right\}.
\]
It is straightforward that $A$ is an additive subgroup of $\mathbb{R}$ and $\lambda : A \rightarrow \R$ is an additive homomorphism. Moreover, by Theorem \ref{thm:trbytrlengthqi},
$$
\Spec(\Ga) \subset A.
$$
Hence, it follows from non-arithmeticity of $\Spec(\Ga)$ that $A \subset \R$ is dense.

\begin{claim*}
There exists $\delta \ge 0$ such that
$$
\la(a) = \delta \cdot a \quad \text{for all } a \in A.
$$
\end{claim*}

To see this claim, choose a nonzero $a \in A$ and set $\delta := \la(a)/a$. By Theorem \ref{thm:trbytrlengthqi}, we can choose $a \in A$ so that $\delta \ge 0$. It suffices to show that for every nonzero $a' \in A$, $\la(a') / a' = \delta$ as well. 
There are two cases.
\begin{enumerate}
\item If $a$ and $a'$ are are commensurable, i.e., $ma = n a'$ for some $m, n \in \Z$, then the conclusion follows from \[
m \lambda(a) = \lambda(ma) = \lambda(n a') = n \lambda(a').
\]
\item If $a$ and $a'$  are not commensurable, then suppose to the contrary that $\delta' := \lambda(a')/a'$ is distinct from $\delta$.

Let $R > 0$ be large enough so that $O := \partial^{h}X \times (-R, R)$ satisfies $0< \mu(O) < + \infty$. Let $K \subset O$ be an arbitrary compact subset. Then there exists $\epsilon > 0$ such that $K \subset \partial^{h}X \times (-R - \epsilon, R + \epsilon)$.

Since $a$ and $a'$ are not commensurable and $\delta \neq \delta'$, there exist $N, M \in \Z$ such that $N > \frac{\epsilon \delta + 1}{|a'||\delta - \delta'|}$ and $\abs{M a - Na'} < \epsilon$. Then setting $s := Ma - Na' \in A$, we have
$$\begin{aligned}
\abs{\la(s)} & = \abs{M\la(a) - N \la(a')} \\
& = \abs{ (Ma - Na')\delta + (\delta - \delta') N a'} \\
& \ge |\delta - \delta'||a'| N - \epsilon \delta > 1.
\end{aligned}
$$
Replacing $s$ with $-s$ if necessary, we may assume that $\la(s) > 1$. Then we have
$$
\mu(O) \ge \mu(T_s K) = (T_s^* \mu)(K) = e^{\la(s)} \mu(K) > e \mu(K).
$$
Since $K \subset O$ is arbitrary, we have $\mu(O) \ge e \mu(O)$, which is absurd.
\end{enumerate}
Therefore, the claim follows.

\medskip

Now consider an arbitrary $a \in \R$.  Since $A \subset \R$ is dense, there exists a sequence $\{a_i\}_{i \in \N} \subset A$ converging to $a$. Let $U \subset \mathcal{H}$ be an open subset and $K \subset U$ a compact subset. Then for all large $i \in \N$, $T_{a_i} K \subset T_a U$, and hence
$$
\mu(T_a U) \ge \mu(T_{a_i}K) = e^{\la(a_i)} \mu(K) = e^{\delta a_i} \mu(K).
$$
Taking the limit $i \to + \infty$, we have $\mu(T_a U) \ge e^{\delta a} \mu(K)$. Since $K \subset U$ is arbitrary, this implies $\mu(T_a U) \ge e^{\delta a} \mu(U)$. By the same argument, we also have $\mu(U) \ge e^{-\delta a} \mu(T_a U)$.
Hence, we have 
$
(T_a^* \mu)(U) = e^{\delta a}\mu(U)
$.
Since this holds for any open subset $U \subset \mathcal{H}$, we have
$$
T_a^* \mu = e^{\delta a} \cdot \mu.
$$
This implies that there exists a finite Borel measure $\nu_0$ on $\partial^{h} X$ so that $\mu$ is decomposed on $\mathcal{H} = \partial^{h} X \times \R$  as follows:
$$
d \mu(\xi, t) = e^{\delta t} \cdot d \nu_0(\xi) \, dt.
$$

By the $\Ga$-invariance of $\mu$, it is easy to see that for each $g \in \Ga$,
$$
\frac{d g_* \nu_0}{d \nu_0}(\xi) = e^{-\delta \beta_{\xi}( g x_0, x_0)} \quad \text{for $\nu_0$-a.e. $\xi \in \partial^{h}X$}.
$$
Then for $x \in X$, define the measure 
$\nu_{x}$ on $\partial^{h} X$ by setting 
$$d \nu_x(\xi) :=  \frac{e^{- \delta \beta_{\xi}(x, x_0)}}{\nu_0(\partial^{h} X)} d \nu_0 (\xi).$$
This is well-defined, and moreover the family $\{ \nu_x \}_{x \in X}$ is a $\delta$-dimensional conformal density of $\Ga$. Since $\{ \nu_x \}_{x \in X}$ is supported on $\La_c(\Ga)$,  $\delta = \delta_{\Ga}$ and $\Ga$ is of divergence type as a result of the generalized Hopf--Tsuji--Sullivan dichotomy (\cite[Corollary 4.25]{coulon2024patterson-sullivan}, \cite[Theorem 1.14]{yang2024conformal}). Therefore,
$$
\mu = \frac{1}{\nu_0(\partial^{h}X)} \cdot \mu_{\Ga},
$$
which completes the proof.
\end{proof}

\section{Existence of ergodic invariant Radon measures} \label{sec:ergodicity}

We continue the setting of Section \ref{sec:UE}. In this section, we prove the ergodicity of the invariant Radon measure defined in Definition \ref{def:candidateergodicmeasure}. This was stated as Theorem \ref{thm:main2} in the introduction.

\begin{theorem} \label{thm:ergodicdiv}
    Let $\Ga < \Isom(X)$ be a non-elementary subgroup with non-arithmetic squeezing spectrum. If $\Ga$ is of divergence type, then 
    $$
    \text{the $\Ga$-action on $(\mathcal{H}, \mu_{\Ga})$ is ergodic.}
    $$
    Moreover, $\mu_{\Ga}$ is supported on $\La_{c}(\Ga) \times \R \subset \mathcal{H}$. Furthermore, up to scalar, $\mu_{\Ga}$ is the unique $\Ga$-invariant Radon measure on $\mathcal{H}$ that is supported on $\Lambda_{c} (\Gamma) \times \R$.
\end{theorem}

Note that $\mu_{\Ga}$ being supported on $\La_{c}(\Ga) \times \R$ is due to Coulon \cite{coulon2024patterson-sullivan} and Yang \cite{yang2024conformal} (Proposition \ref{prop:pattersonSqueeze}). In addition, the unique ergodicity follows from Theorem \ref{thm:uniqueRadon}, once we show the ergodicity.
Hence, it suffices to show that $\mu_{\Ga}$ is $\Ga$-ergodic. This is a special case of the following, together with Proposition \ref{prop:pattersonSqueeze2}:

\begin{theorem} \label{thm:ergodicnormal}
    Let $\Ga < \Isom(X)$ be a non-elementary subgroup of divergence type. Let $\Ga_0 \triangleleft \Ga$ be a normal subgroup such that
    \begin{itemize}
        \item $\Spec (\Ga_0)$ is non-arithmetic and
        \item the $\Ga_0$-action on $\partial^{h} X$ is ergodic with respect to the $\delta_{\Ga}$-dimensional conformal density of $\Ga$.
    \end{itemize} 
    Then, 
    $$ 
    \text{the $\Ga_0$-action on $(\mathcal{H}, \mu_{\Ga})$ is ergodic.}
    $$
\end{theorem}

To prove the ergodicity, we consider the notion of essential subgroups, which was introduced by Schmidt \cite{Schmidt1977cocycles} and studied further by Roblin \cite{Roblin2003ergodicite}.
For a conformal density $\nu = \{ \nu_x \}_{x \in X}$, all measures in the family $\nu$ are in the same measure class. Hence, in discussing positivity of a Borel subset, we simply use the notation $\nu$.

\begin{definition} \label{def:ess}
    Let $\Ga < \Isom(X)$ and let $\nu$ be a conformal density of $\Ga$. We define the subset $\ess_{\nu}(\Ga) \subset \R$ as follows: $a \in \ess_{\nu}(\Ga)$ if for each $\epsilon > 0$ and a Borel subset $E \subset \partial^{h} X$ with $\nu(E) > 0$, there exists $g \in \Ga$ such that
    $$
\nu \left( E \cap g \varphi g^{-1} E \cap \{ \xi \in \partial^{h} X: \abs{\beta_{\xi}(x_0, g \varphi g^{-1} x_0) - a} < \epsilon \}\right) > 0.
    $$
    It is easy to see that $\ess_{\nu}(\Ga)$ is a closed subgroup of $\R$. We call $\ess_{\nu}(\Ga)$ the \emph{essential subgroup} for $\Ga$ and $\nu$.
\end{definition}

The size of the essential subgroup plays a role of criterion for the ergodicity of actions on $\mathcal{H}$. The following was proved in \cite{Schmidt1977cocycles} for abstract measurable dynamical systems, and more direct proof for a particular case of $\op{CAT}(-1)$ spaces was given in \cite{Roblin2003ergodicite}. The same proof works in our setting as well.

\begin{proposition}[{\cite{Schmidt1977cocycles}, \cite[Proposition 2.1]{Roblin2003ergodicite}}] \label{prop.essanderg}
    Let $\Ga < \Isom(X)$ and let $\nu$ be a conformal density of $\Ga$. Then the $\Ga$-action on $(\Hor, \mu_{\nu})$ is ergodic if and only if the $\Ga$-action on $(\partial^{h}X, \nu)$ is ergodic and $\ess_{\nu}(\Ga) = \R$.
\end{proposition}

In this perspective, the following is the main step in the proof of Theorem~\ref{thm:ergodicnormal}, which was proved by Roblin \cite{Roblin2003ergodicite} when $X$ is $\CAT(-1)$. The $\CAT(-1)$ property was crucially used in \cite{Roblin2003ergodicite} to have a nice visual metric on the boundary that guarantees Vitali covering relation of a specific form. In our setting, the lack of Gromov hyperbolicity is an obstruction to consider such a visual metric, and hence we present another proof that does not require metrizing the boundary.

\begin{lem}\label{lem:PSEssential}
    Let $\Ga < \Isom(X)$ be a non-elementary subgroup of divergence type  and $\nu$ a $\delta_{\Ga}$-dimensional conformal density of $\Ga$. Let $\varphi \in \Ga$ be a squeezing isometry. Then for each $\epsilon > 0$ and a Borel subset $E \subset \partial^{h} X$ with $\nu(E) > 0$, there exists $g \in \Ga$ such that $$
    \nu \left( E \cap g \varphi g^{-1} E \cap \{ \xi \in \partial^{h} X: \abs{\beta_{\xi}(x_0, g \varphi g^{-1} x_0) - \tau_{\varphi}} < \epsilon \}\right) > 0.
    $$
    In particular, if $\Ga_0 \triangleleft \Ga$ is a normal subgroup, then
    $$
    \Spec (\Ga_0) \subset \ess_{\nu}(\Ga_0).
    $$
\end{lem}

\begin{proof}
    Let $C = C(\varphi) > 0$ be as in Lemma \ref{lem:BGIPHeredi}. By Proposition \ref{prop:pattersonSqueeze}, $\nu$ is supported on $\La_{\varphi, C}(\Ga)$. Together with the inner regularity of $\nu$, it suffices to consider compact subsets of $\La_{\varphi, C}(\Ga)$.

    Let $K \subset \La_{\varphi, C}(\Ga)$ be a compact subset and let $\epsilon$ be a positive number smaller than $\tau_{\varphi}$. Suppose that for each $g \in \Ga$,
    $$
 \nu \left( K \cap g \varphi g^{-1} K \cap \{ \xi \in \partial^{h} X: \abs{\beta_{\xi}(x_0, g \varphi g^{-1} x_0) - \tau_{\varphi}} < \epsilon \}\right)  = 0.
    $$
    Our goal is to show $\nu(K) = 0$.

    To do this, let $O \subset \partial^{h} X$ be an open subset containing $K$. We will then construct a Borel subset $E(O) \subset O$ such that 
    \begin{equation} \label{eqn:fullessentialproof}
    \nu(K \cap E(O)) = 0 \quad \text{and} \quad \nu(E(O)) \ge  e^{-2 \delta_{\Ga} \tau_{\varphi}} \cdot \nu(K).
    \end{equation}
Before we proceed, let us see how this leads to our goal. Suppose that we have constructed $E(O)$. By $\nu( E(O) \smallsetminus K ) \ge e^{-2 \delta_{\Ga} \tau_{\varphi}} \cdot \nu(K)$ and the inner regularity, there exists a compact subset $K(O) \subset E(O) \smallsetminus K$ such that 
$$
\nu(K(O)) \ge 0.5 e^{-2\delta_{\Gamma} \tau_{\varphi}} \cdot \nu(K).
$$
In particular, $K(O)$ is disjoint from both $K$ and $\partial^{h}X \smallsetminus O$. Now we inductively define
$$
K_1 := K \left(\partial^{h} X \right) \quad \text{and} \quad K_i := K \left( \partial^{h} X \smallsetminus ( K_1 \cup \cdots \cup K_{i-1}) \right) \text{ for } i \in \N.
$$
Then $K_i$'s are pairwise disjoint subsets with $\nu(K_i) \ge e^{-2 \delta_{\Ga} \tau_{\varphi}} \cdot \nu(K)$ for all $i \in \N$. Since $\nu$ is finite, we must have $\nu(K) = 0$.

Hence, it remains to find a set $E(O) \subset O$ satisfying Equation \eqref{eqn:fullessentialproof}. We revisit the proof of Theorem \ref{thm:trbytrlengthqi}, considering the cover $\mathcal{U}$ and its subcollection  $\mathcal{V}$ for $K$ and $O$. Especially, for $l \in \N$, we consider $C_l \subset O$ in Equation \eqref{eqn:defofCl} and the restriction $F_l = g_{i(l)} \varphi g_{i(l)}^{-1} : C_l \to \partial^{h} X$ of the map in Equation \eqref{eqn:defofFl}, where $g_{i(l)} \in \Ga$ is given there. 

We now see that
$$
E(O) := \bigcup_{l \in \N} F_l(C_l \cap K)
$$
satisfies Equation \eqref{eqn:fullessentialproof}. First, note that $\bigcup_{l \in \N} F_l(C_l \cap K) \subset \bigcup_{l \in \N} C_l \subset O$ by Equation \eqref{eqn:claimforFl}.

It follows from  \eqref{eqn:trqiclaim2goal2} that for each $l \in \N$,
\begin{equation} \label{eqn:FlClisinBuse}
F_l(C_l) \subset \left\{ \xi \in \partial^{h}X : \abs{\beta_{\xi}(x_0, g_{i(l)} \varphi g_{i(l)}^{-1} x_0) - \tau_{\varphi}} < \epsilon \right\}.
\end{equation}
We then have
$$\begin{aligned}
 K & \cap F_l(C_l \cap K)  \\
& \subset K \cap g_{i(l)} \varphi g_{i(l)}^{-1} K \cap \left\{ \xi \in \partial^{h}X : \abs{\beta_{\xi}(x_0, g_{i(l)} \varphi g_{i(l)}^{-1} x_0) - \tau_{\varphi}} < \epsilon \right\}
\end{aligned}
$$
and hence our hypothesis on $K$ implies
$
\nu( K \cap F_l (C_l \cap K) ) = 0.
$
Therefore, 
$$
\nu \left(  K \cap \bigcup_{l \in \N} F_l(C_l \cap K)  \right)= 0,
$$
showing the first claim in Equation \eqref{eqn:fullessentialproof}.

Now it remains to estimate $\nu \left( \bigcup_{l \in \N} F_l(C_l \cap K) \right)$. By Equation \eqref{eqn:FlClisinBuse} and $\epsilon < \tau_{\varphi}$, we have for each $l \in \N$ that
$$
\nu(F_l (C_l \cap K)) = \int_{C_l \cap K} e^{-\delta_{\Ga} \beta_{\xi}( g_{i(l)} \varphi^{-1} g_{i(l)}^{-1} x_0, x_0)} d \nu(\xi) \ge  e^{-2 \delta_{\Ga} \tau_{\varphi}} \nu(C_l \cap K).
$$
By Equation \eqref{eqn:claimforFl}, we also have that $F_l(C_l \cap K)$'s are pairwise disjoint. Therefore,
$$\begin{aligned}
\nu \left( \bigcup_{l \in \N} F_l(C_l \cap K) \right) & = \sum_{l \in \N} \nu (F_l(C_l \cap K))  \\
& \ge e^{-2 \delta_{\Ga} \tau_{\varphi}} \sum_{l \in \N} \nu(C_l \cap K) \ge e^{-2 \delta_{\Ga} \tau_{\varphi}} \cdot \nu \left( \bigcup_{l \in \N} (C_l \cap K) \right) 
\end{aligned}
$$
Since $K \subset \bigcup_{l \in \N} C_l$ as in Equation \eqref{eqn:defofCl}, this implies
$$
\nu \left( \bigcup_{l \in \N} F_l(C_l \cap K) \right) \ge  e^{-2 \delta_{\Ga} \tau_{\varphi}} \cdot \nu(K).
$$
Therefore, the second claim in Equation \eqref{eqn:fullessentialproof} follows. \end{proof}

\begin{proof}[Proof of Theorem \ref{thm:ergodicnormal}]
    Let $\nu$ be the $\delta_{\Ga}$-dimensional conformal density of $\Ga$.
    By Lemma \ref{lem:PSEssential}, $\Spec (\Ga_0) \subset \ess_{\nu}(\Ga_0)$. Since $\Spec (\Ga_0)$ is non-arithmetic and $\ess_{\nu}(\Ga_0)$ is a closed subgroup of $\R$, we have $\ess_{\nu}(\Ga_0) = \R$. By the assumption that the $\Ga_0$-action on $(\mathcal{H}, \nu)$ is ergodic, it follows from Proposition \ref{prop.essanderg} that the $\Ga_0$-action on $(\mathcal{H}, \mu_{\nu})$ is ergodic. By definition, $\mu_{\nu} = \mu_{\Ga}$, and hence this completes the proof.
\end{proof}

\section{Subgroups of mapping classs groups and measure~classification} \label{sec:mcgsubgroup}

In the rest of this paper, let $S$ be a connected orientable surface of genus $g$ and with $p$ punctures with $3g - 3 + p \ge 1$. We apply results in previous sections to the case that $X$ is the Teichm\"uller space $\T = \T(S)$. In this section, we deduce Theorem \ref{thm:main1}, Theorem \ref{thm:main1Conv}, and Theorem \ref{thm:mainCC}.

\subsection{Non-elementary subgroups of the mapping class group} \label{subsection:subgpMod} 

For $\Ga < \Mod(S)$, $\Ga$ is non-elementary if and only if $\Ga$ contains two pseudo-Anosov mapping classes with disjoint sets of fixed points in the Thurston boundary $\PML$. Since the axis of a pseudo-Anosov mapping class is squeezing  as in Proposition \ref{prop.pAsqueezing}, a non-elementary subgroup $\Ga < \Mod(S)$ is a non-elementary subgroup $\Ga < \Isom(\T)$ with a squeezing isometry in the sense of Definition \ref{def:noneltsubgp}.

In fact, it follows from the Nielsen--Thurston classification that the class of pseudo-Anosov mapping calsses are precisely the class of squeezing isometries in $\Mod(S)$. Hence, the following gives the \emph{non-arithmeticity} of squeezing spectra:
\begin{theorem}[Non-arithmeticity, {\cite[Theorem 4.1]{gekhtman2023dynamics}}]\label{thm:nonarith} 

Let $\Gamma < \Mod(S)$ be a non-elementary subgroup. Then \[
\operatorname{Spec}_{pA}(\Gamma) := \{ \textrm{$d_{\T}$-translation length of $\varphi$}: \varphi \in \Gamma, \varphi \text{ is pseudo-Anosov}\}
\]
generates a dense additive subgroup of $\mathbb{R}$.
\end{theorem}

The notion of divergence-type is defined similarly.
Fixing a basepoint $x_0 \in \T$, the \emph{Poincar{\'e} series} of a non-elementary subgroup $\Gamma < \Mod(S)$ is
\[
\mathcal{P}_{\Gamma}(s) := \sum_{g \in \Gamma} e^{-s d_{\T}(x_0, gx_0)}.
\]
Since  $\Mod(S)$ has exponentially bounded growth \cite[Theorem 1.3.2]{kaimanovich1996poisson} and contains a free subgroup \cite[Theorem B]{mccarthy1985a-tits-alternative}, there exists $0 < \delta_{\Gamma} < + \infty$ such that $\mathcal{P}_{\Gamma}(s)$ diverges for $s < \delta_{\Gamma}$ and converges for $s > \delta_{\Gamma}$. 

We call $\delta_{\Gamma}$ the \emph{critical exponent} of $\Gamma$. If $\mathcal{P}_{\Ga}(\delta_{\Ga}) = + \infty$, we say that $\Ga$ is of \emph{divergence type}. Otherwise, $\Ga$ is of \emph{convergence type}.

\subsection{Ergodicity and Unique ergodicity} \label{subsec:theoremformcg}

We first deduce Theorem \ref{thm:main1} and Theorem \ref{thm:main1Conv} from our theory in Section \ref{section:prelim}, Section \ref{section:ps}, Section \ref{sec:UE}, and Section \ref{sec:ergodicity}, by  setting $(X, d) = (\T, d_{\T})$. We keep fixing a basepoint $x_{0} \in \T$.

Given a non-elementary subgroup $\Ga < \Mod(S)$, recall the notion of \emph{recurrence locus} for $\Ga$ from the introduction:
$$
\mathcal{R}_{\Ga} := \left\{ \xi \in \ML :
\begin{matrix}
    \text{Teichm\"uller geodesic ray given by } q_{\xi} \in \mathcal{Q}(S, x_0) \\
    \text{recurs to a compact subset in } \Ga \ba \T
\end{matrix} \right\}
$$
where $q_\xi \in \mathcal{Q}(S, x_0)$ is the holomorphic quadratic differential corresponding to $\xi \in \ML$, given by the Hubbard--Masur theorem.

In terms of the Hubbard--Masur coordinates we introduced in Section \ref{subsec:HM}, it follows from the Masur criterion \cite[Theorem 1.1]{masur1992hausdorff} that
$$
\HM(\mathcal{R}_{\Ga}) \subset \UE \times \R.
$$

Recall that $\UE$ sits in both $\PML$ and $\partial^{GM} \T$, with the same topology \cite[Theorem 2]{miyachi2013teichmuller}. Hence, $\UE \times \R$ sits in both $\PML \times \R \simeq \ML$ via the Hubbard--Masur coordinates and $\partial^{GM} \T \times \mathbb{R} \simeq \partial^{h} \T \times \mathbb{R}$ via the coordinates in Section \ref{subsection:liuSuWalsh}. That means, the identification \[
\iota : \UE \times \R \subset \ML \rightarrow \UE \times \R \subset \partial^{h} \T \times \mathbb{R} = \mathcal{H}
\]
is a homeomorphism. Hence, a Borel measure $\mu$ on $\ML$ supported on $\UE \times \R$ can be viewed as Borel measures on $\partial^{h} \T \times \mathbb{R}$ supported on $\UE \times \R$. Note that $\iota$ preserves the $\R$-coordinate. Combining altogether, we have that
\begin{equation} \label{eqn:recurrencelocusembedding}
(\iota \circ \HM) (\mathcal{R}_{\Ga}) = \La_{c}(\Ga) \times \R \subset \mathcal{H},
\end{equation}

Furthermore, if $\mu$ is locally finite on $\ML$, then $\iota_* \mu$ is also locally finite on $\partial^{h} \T \times \mathbb{R}$. Indeed, for a compact subset $K \subset \partial^{h} \T \times \mathbb{R}$, the set $\{t : (\xi, t) \in K \cap (\UE \times \mathbb{R})\}$ is bounded. This implies that the 
$\iota^{-1}(K \cap (\UE \times \mathbb{R}))$ is precompact, as it sits in $\PML \times [-R, R]$ for some large $R > 0$. If $\mu(\iota^{-1}(K\cap (\UE \times \mathbb{R})))$ is assumed to be finite, then $\iota_* \mu (K)$ is finite as well. In conclusion, locally finite measures on $\ML$ supported on $\UE \times \R$ are also locally finite when viewed on $\partial^{h} \T \times \mathbb{R}$. Since $\partial^{h} \T \times \mathbb{R}$ is Polish, such measures are Radon.

Hence, if $\Ga$ is of divergence type, then we can pullback the measure $\mu_{\Ga}$ on $\La_{c}(\Ga) \times \R$ defined in Definition \ref{def:candidateergodicmeasure}, via the embedding $\iota \circ \HM$. This gives the measure on $\mathcal{R}_{\Ga}$ which is the same as the one constructed in
Section \ref{subsec:constructmeasureintro}, also denoted by $\mu_{\Ga}$ abusing the notation.

Therefore, together with the non-arithmeticity in Theorem \ref{thm:nonarith}, the ergodicity (Theorem \ref{thm:main1}) and the unique ergodicity (Theorem \ref{thm:main1Conv}) follow from Theorem  \ref{thm:ergodicdiv} and
Theorem \ref{thm:uniqueRadon} respectively. Note that as in Theorem \ref{thm:ergodicnormal}, an analogous ergodicity theorem for normal subgroups can also be deduced.

\subsection{Convex cocompact subgroups of the mapping class group}

In \cite{farb2002convex}, B. Farb and L. Mosher introduced the following notion:

\begin{definition}
    A finitely generated subgroup $\Ga < \Mod(S)$ is called \emph{convex cocompact} if $\Ga$ has a quasi-convex orbit in $\T$.
\end{definition}

Some important features of convex cocompact subgroups are as follows:

\begin{theorem}[{\cite[Theorem 1.1]{farb2002convex}}] \label{thm:FMCC}
Let $\Ga < \Mod(S)$ be a convex cocompact subgroup. Then 
\begin{enumerate}
    \item $\Ga$ is a hyperbolic group,
    \item there exists a $\Ga$-equivariant embedding
    $$\partial \Ga \hookrightarrow \UE \subset \PML$$
    where $\partial \Ga$ denotes the Gromov boundary of $\Ga$, and
    \item denoting by $\La(\Ga) \subset \UE$ the image of the embedding in (2), $\Ga$ acts cocompactly on its weak-hull $\op{WH}(\Ga) \subset \T$, the union of all bi-infinite Teichm\"uller geodesics with endpoints in $\La(\Ga)$.
\end{enumerate}
\end{theorem}

We now discuss the divergence-type of non-elementary convex cocompact subgroups. 
Theorem \ref{thm:FMCC} asserts that every non-elementary  convex cocompact subgroup $\Gamma < \Mod(S)$ is Gromov hyperbolic when endowed with either a word metric or the Teichm{\"u}ller metric on its orbit. By \cite[Th{\'e}or{\`e}me 7.2]{coornaert1993mesures}, $\Gamma$ has \emph{purely exponential growth}, i.e., there exists $C > 1$ (depending on the choice of the basepoint $x_{0}$) such that \[
\frac{1}{C} e^{\delta_{\Gamma} r} \le \# \{ g \in \Gamma : d_{\T}(x_0, gx_0 ) < r \} \le C e^{\delta_{\Gamma} r} \quad \text{for all } r > C.
\]
This implies that $\Gamma$ is of divergence type. In fact, Gekhtman studied in \cite{gekhtman2013dynamics} the finiteness and the mixing property of the Bowen--Margulis--Sullivan measure associated with $\Gamma$. As a result, he obtained that the quantity $e^{-\delta_{\Gamma} }\# \{ g \in \Gamma : d_{\T}(x_0, gx_0 ) < r \}$ converges to a finite limit as $r$ tends to infinity.

In addition,  it follows from the work of McCarthy--Papadopoulos \cite{mccarthy1989dynamics} that the set $\La(\Ga) \subset \PML$ is the unique minimal subset of $\Ga$, and moreover the $\Ga$-action on the complement $\PML \smallsetminus \La(\Ga)$ is properly discontinuous. In this regard, the set $\La(\Ga) \subset \PML$ can be viewed as the \emph{limit set} of $\Ga$. 

Furthermore, the cocompactness in Theorem \ref{thm:FMCC}(3) and the embedding in Equation \eqref{eqn:recurrencelocusembedding} give a characterization of the recurrence locus for $\Ga$:
\begin{equation} \label{eqn:recurrenceCC}
\mathcal{R}_{\Ga} = \{ \xi \in \ML : [\xi] \in \La(\Ga) \}.
\end{equation}
Hence, the proper discontinuity of the $\Ga$-aciton on $\PML \smallsetminus \La(\Ga)$ implies that the $\Ga$-action on $\ML \smallsetminus \mathcal{R}_{\Ga}$ is properly discontinuous as well. Thus any $\Ga$-invariant ergodic measure on $\ML \smallsetminus \mathcal{R}_{\Ga}$ is a counting measure on a single $\Ga$-orbit, up to a constant multiple.
Since $\Gamma$ is of divergence type, Theorem \ref{thm:mainCC} follows from Theorem \ref{thm:main1Conv}.

\section{Classification of orbit closures} \label{sec:orbitclosure}

In this section, we prove the following classification of orbit closures:

\begin{theorem} \label{thm:orbitclosure}
Let $\Ga < \Mod(S)$ be a non-elementary convex cocompact subgroup. Then for each $\xi \in \ML$, either
$$ 
\Ga \cdot \xi \text{ is discrete} \quad \text{or} \quad \ov{\Ga \cdot \xi} = \mathcal{R}_{\Ga}.
$$
More precisely,
\begin{itemize}
    \item if $\xi \notin \mathcal{R}_{\Ga}$, then $\Ga \cdot \xi$ is discrete.
    \item if $\xi \in \mathcal{R}_{\Ga}$, then $\ov{\Ga \cdot \xi} = \mathcal{R}_{\Ga}$.
\end{itemize}

\end{theorem}

As mentioned earlier, the $\Ga$-action on $\PML \smallsetminus \La(\Ga)$ is properly discontinuous \cite{mccarthy1989dynamics}. Hence, in the viewpoint of the characterization of $\mathcal{R}_{\Ga}$ in Equation \eqref{eqn:recurrenceCC}, it suffices to show that the $\Ga$-action on $\mathcal{R}_{\Ga}$ is minimal. 

We do this by adjusting a standard argument of classifying horospherical orbit closures for Kleinian groups, which is based on nice shapes of horospheres in hyperbolic spaces. As we do not have such a well-shaped horosphere in the Teichm\"uller space, we use a recent theory of exapnding coarse cocycles for convergence group actions by Blayac--Canary--Zhu--Zimmer \cite{BCZZ_PS} and an interpretation of the action of convex cocompact subgroups as expanding coarse cocycles given by the second author and Zimmer \cite{KimZimmer_Rigidity}.
 
From the viewpoint of the characterization in Equation \eqref{eqn:recurrenceCC} and the discussion in Section \ref{subsection:horo}, the minimality of the $\Ga$-action on $\mathcal{R}_{\Ga}$ follows once we show that the $\Ga$-action on $\La(\Ga) \times \R$ given by
$$
g \cdot (\xi, t) = (g \xi, t + \beta_{\xi}(g^{-1} x_0,x_0))
$$
is minimal, recalling that $x_0 \in \T$ is a fixed basepoint.

Given a pseudo-Anosov $g \in \Mod(S)$, we denote by $g^+, g^- \in \PML$ its attracting and repelling fixed points, respectively.

\begin{lemma} \label{lem:occontainspA}
    For each $(\xi, t) \in \La(\Ga) \times \R$ and a pseudo-Anosov $g \in \Ga$, there exists $s \in \R$ such that $(g^+, s) \in \ov{\Ga \cdot (\xi, t)}$ or $(g^-, s) \in \ov{\Ga \cdot (\xi, t)}$.
\end{lemma}

\begin{proof}
    By Theorem \ref{thm:FMCC}(3), there exists a sequence $\{g_n\}_{n \in \N}\subset \Ga$ such that $g_n x_0 \to \xi$ within a bounded neighborhood of a Teichm\"uller geodesic ray towards $\xi \in \UE$. This implies 
    \begin{equation} \label{eqn:conicalconvergence}
        \beta_{\xi}(g_n x_0, x_0) \to -\infty \quad \text{as } n \to + \infty.
    \end{equation}

    This also implies the conical convergence in the Cayley graph of  a hyperbolic group $\Ga$, identifying $\partial \Ga = \La(\Ga)$ by Theorem \ref{thm:FMCC}(2). After passing to a subsequence, we may assume that $g_n^{-1} \xi$ converges to some $\zeta \in \La(\Ga)$. Suppose first that $g^- \neq \zeta$.

Fix an open neighborhood $U \subset \La(\Ga)$ of $\zeta$ such that $g^- \notin \ov{U}$. Then there exists a compact subset $K$ in the Cayley graph of $\Ga$ such that every bi-infinite geodesic with endpoints $g^- \in \partial \Ga$ and a point in $\ov{U}$ intersects $K$. Passing to a subsequence, we may assume that $g_n^{-1} \xi \in U$ for all $n \in \N$.  This implies
\begin{equation} \label{eqn:shadow}
    \sup_{k, n \in \N} \abs{\beta_{g_n^{-1} \xi}(g^{-k}x_0, x_0) - d(g^{-k} x_0, x_0)} < + \infty
\end{equation}
(see \cite[Proposition 12.6]{KimZimmer_Rigidity} and \cite[Section 5]{BCZZ_PS} for details).

Observe that for $k, n \in \N$,
$$
g^k g_n^{-1} (\xi, t) = (g^k g_n^{-1} \xi, t + \beta_{\xi}(g_n x_0, x_0) + \beta_{g_n^{-1} \xi}(g^{-k}x_0, x_0)).
$$
Then by Equation \eqref{eqn:shadow}, for each $n \in \N$, we can choose $k_n \in \N$ so that the sequence
$$
t + \beta_{\xi}(g_n x_0, x_0) + \beta_{g_n^{-1} \xi}(g^{-k_n}x_0, x_0) \quad \text{is bounded.}
$$
By Equation \eqref{eqn:conicalconvergence}, we also have $k_n \to + \infty$ as $n \to + \infty$.

After passing to a subsequence, we can set 
$$
s := \lim_{n \to + \infty}  \big(t + \beta_{\xi}(g_n x_0, x_0) + \beta_{g_n^{-1} \xi}(g^{-k_n}x_0, x_0) \big).
$$
 Since $g_n^{-1} \xi \in U$ and $g^- \notin \ov{U}$, we have
$$
g^{k_n} g_n^{-1} \xi \to g^+ \quad \text{as }n \to + \infty.
$$
This finishes the proof in this case.

If $g^- = \zeta$, then we have $g^+ \neq \zeta$. Hence, we can apply the same argument replacing $g$ with $g^{-1}$ and this completes the proof.
\end{proof}

\begin{proof}[Proof of Theorem \ref{thm:orbitclosure}]
   For each $s \in \R$, consider the map 
   $$\begin{aligned}
   a_s : \La(\Ga) \times \R &  \to \La(\Ga) \times \R \\
    (\xi, t) & \mapsto (\xi, t + s)
   \end{aligned}
   $$
   which gives the $\R$-action on $\La(\Ga) \times \R$, commuting with the $\Ga$-action. Hence, we can consider it as a right $\R$-action. Note that this was denoted by $T_s$ in Section \ref{subsec:trqi}; we use the new notation to consider it as an action on the right.

   To prove the desired minimality, we fix $(\xi, t) \in \La(\Ga) \times \R$ and $\epsilon > 0$. By the non-arithmeticity in Theorem \ref{thm:nonarith}, there exist pseudo-Anosovs $g_1, \dots, g_k \in \Ga$ such that the additive subgroup $\langle \tau_{g_1}, \dots,  \tau_{g_k} \rangle \subset \R$ generated by their $d_{\T}$-translation lengths is $\epsilon$-dense.

   By Lemma \ref{lem:occontainspA}, after replacing $g_1$ with $g_1^{-1}$ if necessary, we have 
   $$
    (g_1^+, s_1) \in \ov{\Ga \cdot (\xi, t)} \quad \text{for some } s_1 \in \R.
   $$
   Then for every $j \in \Z$,
   $$
    g_1^j (g_1^+, s_1) = (g_1^+, s_1 + \beta_{g_1^+}(g_1^{-j} x_0, x_0)) = (g_1^+, s_1 + j \tau_{g_1}) = (g_1^+, s_1)a_{j \tau_{g_1}}
   $$
   belongs to $\Ga \cdot (g_1^+, s_1) \subset \ov{\Ga \cdot (\xi, t)}$. This implies
   $$
    \ov{\Ga \cdot (g_1^+, s_1)} a_{\langle \tau_{g_1} \rangle} \subset \ov{\Ga \cdot (\xi, t)}.
   $$
   Applying the same argument to $\ov{\Ga \cdot (g_1^+, s_1)}$, we inductively obtain that for some $s_k \in \R$,
   $$
    \ov{\Ga \cdot (g_k^+, s_k)} a_{\langle \tau_{g_1}, \dots, \tau_{g_k} \rangle} \subset \ov{\Ga \cdot (\xi, t)}.
   $$

   Now let $(\zeta, s) \in \La(\Ga) \times \R$ be arbitrary. Since the $\Ga$-action on $\La(\Ga)$ is minimal, there exists a sequence $\{h_n\}_{n \in \N} \subset \Ga$ such that $h_n g_k^+ \to \zeta$ as $n \to + \infty$. Since $\langle \tau_{g_1}, \dots, \tau_{g_k} \rangle$ is $\epsilon$-dense, for each $n \in \N$, there exists $t_n \in \langle \tau_{g_1}, \dots, \tau_{g_k} \rangle$ such that
   $$
    \abs{s - ( s_k + \beta_{g_k^+}(h_n^{-1} x_0, x_0) + t_n)} < \epsilon.
   $$
   Hence, after passing to a subsequence, the sequence $h_n ( g_k^+, s_k) a_{t_n}$ converges to $(\zeta, s') \in \ov{\Ga \cdot (\xi, t)}$ for some $s' \in \R$ with $\abs{s - s'} \le \epsilon$. Since this holds for each $\epsilon > 0$, we have $(\zeta, s) \in \ov{\Ga \cdot (\xi, t)}$. This finishes the proof.
\end{proof}

\bibliographystyle{alpha} 
\bibliography{ML}

\end{document}